\documentclass[12pt,a4paper,reqno]{amsart}
\usepackage{amsmath,amssymb,graphics,epsfig,color,enumerate,mathrsfs}
\usepackage{dsfont}  \usepackage{verbatim}
\definecolor{refkey}{gray}{.75}
\definecolor{labelkey}{gray}{.5}

\usepackage[usenames,dvipsnames]{xcolor}

\definecolor{ballblue}{rgb}{0.13, 0.90,0.8}

\newtheorem{Theorem}{Theorem}[section]

\newtheorem{Lemma}[Theorem]{Lemma}
\newtheorem{Proposition}[Theorem]{Proposition}
\newtheorem{Corollary}[Theorem]{Corollary}
\newtheorem{Remark}[Theorem]{Remark}

\newtheorem{Claim}[Theorem]{Claim}
\newtheorem{Definition}[Theorem]{Definition}
\newtheorem{Warning}{Warning}[section]

 \definecolor{darkgreen}{rgb}{0,0.7,0.2}

\newcommand{\cA}{\ensuremath{\mathcal A}}
\newcommand{\cB}{\ensuremath{\mathcal B}}
\newcommand{\cC}{\ensuremath{\mathcal C}}
\newcommand{\cD}{\ensuremath{\mathcal D}}
\newcommand{\cE}{\ensuremath{\mathcal E}}
\newcommand{\cF}{\ensuremath{\mathcal F}}
\newcommand{\cG}{\ensuremath{\mathcal G}}
\newcommand{\cH}{\ensuremath{\mathcal H}}

\newcommand{\cL}{\ensuremath{\mathcal L}}
\newcommand{\cM}{\ensuremath{\mathcal M}}
\newcommand{\cN}{\ensuremath{\mathcal N}}
\newcommand{\cO}{\ensuremath{\mathcal O}}
\newcommand{\cP}{\ensuremath{\mathcal P}}
\newcommand{\cQ}{\ensuremath{\mathcal Q}}

\newcommand{\cS}{\ensuremath{\mathcal S}}

\newcommand{\cV}{\ensuremath{\mathcal V}}
\newcommand{\cW}{\ensuremath{\mathcal W}}

\newcommand{\bbB}{{\ensuremath{\mathbb B}} }

\newcommand{\bbE}{{\ensuremath{\mathbb E}} }

\newcommand{\bbG}{{\ensuremath{\mathbb G}} }

\newcommand{\bbI}{{\ensuremath{\mathbb I}} }

\newcommand{\bbL}{{\ensuremath{\mathbb L}} }

\newcommand{\bbN}{{\ensuremath{\mathbb N}} }

\newcommand{\bbQ}{{\ensuremath{\mathbb Q}} }
\newcommand{\bbR}{{\ensuremath{\mathbb R}} }

\newcommand{\bbZ}{{\ensuremath{\mathbb Z}} }

 \let\b=\beta   \let\d=\delta  \let\e=\varepsilon
 \let\g=\gamma     \let\k=\kappa  \let\l=\lambda
      \let\o=\omega      
  \let\s=\sigma \let\t=\tau   
  \let\z=\zeta
\let\D=\Delta   \let\G=\Gamma   
\let\O=\Omega      
\newcommand{\da}{\downarrow}
\newcommand{\rrr}{\textcolor{black}}
\newcommand{\utto}{\textcolor{black}}

\newcommand{\ovo}{\textcolor{black}}
\newcommand{\iva}{\textcolor{black}}
\newcommand{\green}{\textcolor{black}}
\newcommand{\vvv}{\textcolor{black}}

\newcommand{\toup}{\rightharpoonup}

\newcommand{\be}{\begin{equation}}
\newcommand{\en}{\end{equation}}
\newcommand{\bee}{\begin{multline}}
\newcommand{\ene}{\end{multline}}

\author[A.~Faggionato]{Alessandra Faggionato}
\address{Alessandra Faggionato.
  Dipartimento di Matematica, Universit\`a di Roma `La Sapienza'
  P.le Aldo Moro 2, 00185 Roma, Italy}
\email{faggiona@mat.uniroma1.it}

\newcommand{\ra}{\rangle}
\newcommand{\la}{\langle}

\thanks{This work has been partially supported by the ERC Starting Grant 680275 MALIG}

\title[Stochastic homogenization  of  random walks on point processes]{\ovo{Stochastic homogenization  of  random walks on point processes}}

\begin{document}

\begin{abstract}
We consider   random walks on the support of a random  purely atomic   measure on $\mathbb{R}^d$ with random jump probability rates.  The  jump range can be unbounded. The purely atomic measure is reversible for the random walk and  stationary for the action of the group $\mathbb{G}=\mathbb{R}^d$ or $\mathbb{G}=\mathbb{Z}^d$.  By combining two-scale convergence   and Palm theory for $\mathbb{G}$-stationary random measures and by developing a cut-off procedure, under  suitable second moment conditions  we prove for almost all environments the  homogenization  for the massive Poisson equation of the associated  Markov generator.   In addition, we obtain the quenched convergence of the $L^2$-Markov semigroup  and resolvent of the diffusively rescaled random walk to the corresponding ones  of the  Brownian motion with covariance matrix $2D$. For symmetric jump rates,  the above convergence   
plays a crucial role in the derivation of hydrodynamic limits when considering multiple random walks  with site-exclusion or zero range interaction. We do not require any ellipticity assumption, neither non-degeneracy of the homogenized matrix $D$. Our  results cover a large family of models, including   e.g. random conductance models on $\mathbb{Z}^d$  and on general lattices (possibly with long conductances), Mott variable range hopping, simple random walks on Delaunay triangulations,  simple random walks on  supercritical percolation clusters.

\smallskip


\smallskip


\end{abstract}


\maketitle
%

\section{Introduction}
Stochastic homogenization of random walks in random environment is a rich  research field initiated  in the Western school by Varadhan, Papanicolaou and coauthors and in the Russian school by Kozlov  (cf. e.g. \cite{AKM,Bi,EGMN,FH,FHS,KV,KLO,Ko85,Ko,Kum,Ku,MP,PR,PV} and references therein). Hence,
before describing our results, we illustrate some questions motivating the present work, thus providing a reading key to this article. 
  A first question concerns hydrodynamic limits (HL's)  of interacting particle systems in random environment. In \cite{F1,F2}  we proved the quenched HL for the simple exclusion process and the zero range process, respectively,  on the supercritical percolation cluster with random conductances.  The proof relies (between other) on a  weak form of  quenched convergence  of the Markov semigroup  and resolvent for a single random walk towards  their counterparts  for the  Brownian motion. This convergence was obtained from the homogenization of the massive Poisson equation in  \cite{F1}.  It  is then natural  to ask  if  these results  hold for  a much  larger class of models.
A second question motivating the present work concerns Mott's law, which predicts the anomalous conductivity decay at low temperature in  amorphous solids   with electron transport given by  Mott variable range hopping \cite{POF}. A mean--field model  is given by Mott random walk   on 
a marked simple\footnote{\ovo{We follow the terminology of   \cite{DV,DV1,DV2},  hence  \emph{simple} just means that points have unit multiplicity.}} point process in $\bbR^d$ \cite{FSS}. By invoking Einstein's relation, Mott's law can be stated  in terms of the dependence of the  effective diffusion matrix $D(\b)$ on the temperature $\b^{-1}$.  In \cite{CF2,CFP} we  proved a quenched invariance principle for Mott random walk  under suitable conditions on the point process with a  homogenization--type characterization of $D(\b)$, while in 
\cite{FM,FSS} we proved   bounds on $D(\b)$  in agreement with Mott's law.
 As Mott's law has a large class of universality, it is natural  to ask if a weak form of quenched  CLT with the same matrix  $D(\b)$ holds  for a much larger (compared to
  \cite{CF2,CFP})   class of point processes.  
In the present work   (cf. Theorems \ref{teo1} and \ref{teo2}) we give a positive answer to the above questions. In our companion work \cite{F_SE_RE} we then derive the quenched HL for  simple exclusion processes in symmetric random environments. The zero range process will be treated in a future work.

 Let us  briefly describe our results.
  We  consider
  here general random walks with state space given by  the support of a given random \ovo{purely atomic} measure  on  $\bbR^d$, possibly contained in $\bbZ^d$ or  a generic $d$--dimensional discrete lattice as e.g.  the hexagonal one. Also the  jump probability rates can be random.
    The above randomness of the environment is supposed to be stationary with respect to the action of the  group $\bbG=\bbR^d$ or $\bbG=\bbZ^d$. We assume some second moment conditions and that the above  random \ovo{purely atomic} measure  is reversible for the random walk.
 Our first main result is then given by the  quenched  
   homogenization of   the massive Poisson equation associated to  the diffusively rescaled random walk, leading to an  effective diffusive equation and a variational characterization of the homogenized matrix $D$  (see Theorem \ref{teo1}).   Our  second  
 main result concerns a quenched convergence of the Markov semigroup and  resolvent  of 
 the diffusively rescaled random walk   to their counterparts   for the Brownian motion with covariance matrix  $2 D$ (see Theorem \ref{teo2}).
  This form of convergence includes a spatial ``average'' on the initial point of the random walk, and it is  the proper form relevant to get the above  quenched HL's. It  can also  be thought of   as a weak form of quenched CLT.  
We point out that we do not require any ellipticity assumption and our results cover the case of  degenerate homogenized matrix as well.  
 
%
%
%
Our results cover  a broad class  of models. Just to list some examples: the random conductance models on $\bbZ^d$  and on general lattices (possibly with long conductances), Mott variable range hopping, simple random walks on Delaunay triangulations,  simple random walks on  supercritical percolation clusters with random conductances.  To gain such a generality we  have used  the theory of  $\bbG$-stationary random measures \cite{Ge,GL,Km}, thus allowing to have a common language for all models and to formulate  the 2-scale ergodic properties of the environment in terms of Palm distributions.   
 We have also  used   the method of 2-scale convergence  (\cite{A,AKM,Nu,ZP} and references therein). 

Our proof of Theorem \ref{teo1} is inspired by the strategy developed in \cite{ZP}, dealing with  random differential operators on singular structures, but 
 key technical obstructions due to possible arbitrarily long jumps (and therefore not present in \cite{ZP}) 
 have emerged. One main technical effort  here has been to deal with   2-scale convergence using only  $L^2$--concepts. Let us explain this issue. First, we observe that the standard gradient has to be replaced by an amorphous gradient, keeping knowledge of the variation of a function along all possible jumps. While the gradient of a regular function $\varphi$ on $\bbR^d$ with compact support maintains the same properties, the amorphous gradient of $\varphi$ is an  irregular object, which cannot be bounded in uniform norm. Hence, the 2-scale  method developed for random differential operators does not work properly, since several limits become now illegal.
  We have been able to overcome this difficulty by enhancing the standard method with suitable cut-off procedures (cf. Sections \ref{cut-off1} and \ref{cut-off2}).   Another main difference with \cite{ZP} is the following. In order to prove that for almost any environment homogenization holds for all massive Poisson equations, we need to restrict to countable dense families of testing objects in the definitions of 2-scale convergence and not always, in our context and differently from \cite{ZP}, one can reduce for free to continuous objects when  depending  on the environment,  as for  the space of solenoidal forms (cf. Section \ref{sec_quadro}). To overcome this obstruction, we have been able to deal only with countable families of    $L^2$ environment--dependent testing objects. This has also the advantage to avoid topological assumptions on the probability space with exception of the separability of $L^2(\cP_0)$, $\cP_0$ being the Palm distribution of the random \ovo{purely atomic} measure. On the other hand, a special care has been necessary in defining the right class of testing objects  (cf. Section \ref{topo}).

 We \ovo{now give}  some comments on reference \cite{FHS}
containing results on spectral homogenization for the discrete Laplacian on $\bbZ^d$ with random conductances and  unbounded length range. The conditions assumed  in \cite{FHS} are more complex, implying  both Poincar\'e and Sobolev inequalities (on the other hand, the authors face with a  different problem). Note that for the model in \cite{FHS} the  assumptions of our  Theorem \ref{teo1} and Theorem \ref{teo2} reduce to Assumption 1.1-(a),(b) in \cite{FHS}. In particular, we avoid the more sophisticated   Assumptions 1.1-(c) and 1.2 in \cite{FHS}.
  The conditions in \cite{FHS}  guarantee the  uniform  boundedness in $L^\infty$--norm of the solution of the Poisson equation with Dirichlet boundary conditions there (cf. \cite[Proposition 3.4]{FHS}).  As a consequence,  the  derivation of the structure result stated  in Lemma 5.15 (and somehow corresponding  to \cite[Lemma 5.4]{ZP}  and to Proposition \ref{oro} here) is  much simplified by the $L^\infty$--boundedness of the solutions and the obstruction mentioned above (solved by the  cut--off procedure) does not emerge in \cite{FHS}. Similarily, homogenization results for random walks on Delaunay triangulations have been obtained also in \cite{H} under the condition that  the diameters of the Voronoi cells  are uniformly bounded both  from below and from above (see Condition 1.2 in \cite{H}). 
Our analysis does not need such uniform bounds. 
 We also point out that some stronger assumptions  in \cite{FHS} imply the non--degeneracy of the homogenized matrix, and this property enters in their proof of the structure result given by \cite[Lemma 5.15]{FHS}. Finally, we remark that, when working with $\bbZ^d$ as in \cite{FHS}, the form of the Palm distribution is   much simpler  (compare \eqref{zazzera} with \eqref{palm_classica} and \eqref{Palm_Z}) and therefore several manipulations concerning square integrable forms become simple, differently from the general case (see e.g. Lemma \ref{lemma_siou}).

\ovo{Knowing the   non-degeneracy of the diffusion matrix  $D$ gives important information on the effective homogenized equation.  As possible techniques to prove the non-degeneracy we mention the use of lower Gaussian kernel bounds and  the sublinearity of the corrector (cf.~e.g.~\cite{Bi},\cite[Section~6.1]{BB},  \cite[Prop.~2.5]{DNS}, \cite{MP}). Alternatively, one can use  electrostatic  arguments by  combining the results  of   \cite{F_resistor} for resistor networks on point processes   with  Remark \ref{virale} below  (see \cite[Section 3]{F_resistor} and also \cite{demasi} for a special case with a more complex construction of the resistor network).  Further details are given in Section \ref{sec_esempi}.}
 
 \ovo{We conclude with some remarks on the quenched invariance principle of  random walks on point processes. The results contained in this work, and in particular Proposition \ref{oro}, allow to extend the analysis  of the corrector 
 in \cite{MP} to our general  class of random walks on point processes, for the part concerning 2-scale convergence. Mainly the proof of tightness would then require a further analysis. We also mention \cite{BCKW,CF2,CFP,Rou} where the  quenched invariance principle of random walks with long-range jumps is proved.}

\smallskip

{\bf Outline of the paper}. In Section \ref{flauto} we describe our setting and main assumptions. In Section \ref{setting2} we introduce the massive Poisson equation, the homogenized equation  and discuss convergence types. In Section \ref{sec_MR_hom} we present our main results. In Section \ref{sec_esempi} we discuss some examples. In Section \ref{sec_fluidifico} we show  that it is enough to consider the case $\bbG=\bbR^d$.
  In Section \ref{sec_minecraft} we derive some key properties of the Palm distribution.  In Sections \ref{sec_quadro}, \ref{sec_diff_matrix}, \ref{sec_cinghia} and \ref{hermione} we study square integrable forms.  In Section \ref{topo} we describe the set $\O_{\rm typ}$ of typical environments appearing in Theorems \ref{teo1} and \ref{teo2}. In Section \ref{sec_2scale} we discuss 2-scale convergence in our setting. \ovo{In Section \ref{mappa} we provide a roadmap for the proof of Theorem \ref{teo1}}.  In Sections \ref{cut-off1} and \ref{cut-off2} we develop the basis of our cut-off procedure. In  Sections \ref{sec_risso}  and \ref{sec_oro} we study bounded families of functions  in the Hilbert space $H^1_{\o,\e}$ (cf. Definition \ref{bic}).
          Having developed all the necessary machinery, in Sections \ref{robot} and \ref{dim_teo2} we prove respectively Theorems \ref{teo1} and \ref{teo2}. Finally, several technical facts have been collected in Appendixes \ovo{\ref{misurino78}}, \ref{app_bell}, \ref{gennaro}, \ref{app_no_TNT}, \ref{app_fluidifico}, \ref{app_compatto}. 

\section{Setting}\label{flauto}
In this section we introduce the basic concepts in our modelisation: the  group $\bbG$ ($\bbG=\bbR^d$ or $\bbG=\bbZ^d$)  acting  on a given   probability space  $(\O, \cF, \cP)$  and on the space $\bbR^d$;  the  random  $\bbG$--stationary \ovo{purely atomic} measure $\mu_\o$ on $\bbR^d$ and the family of transition rates $r_{x,y}(\o)$.  We also list our main assumptions, given by   (A1),...,(A8) below.  In Section \ref{sec_MR_hom} we will introduce  assumption (A9) for Theorem \ref{teo2}.

First we  fix 
 some basic notation. 
 We denote by $e_1, \dots, e_d$ the canonical basis of $\bbR^d$.  We denote by $\ell (A)$   the Lebesgue measure of the Borel set $A\subset \bbR^d$.  The standard scalar product of $a,b\in \bbR^d$ is denoted by $a\cdot b$.
Given a measure $\nu $, we  denote by $\la \cdot, \cdot \ra _{\nu}$   the scalar product   in $L^2(\nu)$.
 Given a topological space $W$,  without further mention, $W$ will be  thought of as measurable space endowed with the $\s$--algebra  $\cB(W) $   of its Borel subsets.

\subsection{Actions of the group $\bbG$}
$\bbG$ will be the abelian group $\bbR^d$ or $\bbZ^d$, which  are  endowed with the standard Euclidean topology and the discrete topology, respectively.

$\bullet$ {\sl Action of $\bbG$ on $(\O, \cF, \cP)$}.
The action of $\bbG$ on $(\O, \cF, \cP)$ 
is given by 
 a family of measurable maps  $(\theta_g)_{g\in \bbG}$ \ovo{with $\theta_g:\O\to\O$} such that 
\begin{itemize}
\item[(P1)] $ \theta_0=\mathds{1}$, 
\item[(P2)] $ \theta _g \circ \theta _{g'}= \theta_{g+g'}$ for all $g,g'\in \bbG$,
\item[(P3)] the map $\bbG\times \O \ni (g,\o) \mapsto \theta _g \o \in \O$ is measurable,
\item[(P4)] $\cP$ is left invariant by the $\bbG$--action, i.e. $\cP\circ \theta_g^{-1}=\cP$ \ovo{for all $g\in \bbG$}.
\end{itemize}
The last property (P4) corresponds to the so--called $\bbG$--stationarity of $\cP$.
A subset  $A\subset \O$ is called  translation invariant if  $\theta_g A =A$ for all $g\in \bbG$.

$\bullet$ {\sl Action of $\bbG$ on $\bbR^d$}.
The action  $(\t_g )_{g\in \bbG}$ of $\bbG$ on $\bbR^d$ is given  by  translations. More precisely, we assume that,  for a basis   $v_1, \dots, v_d$ in $\bbR^d$,
\be\label{trasferta}
\t_g x=x+ g_1v_1 + \cdots +g_d v_d \qquad \forall  x\in \bbR^d\,,\; g=(g_1,\dots, g_d) \in \bbG\,.
\en
Equivalently, by  thinking of $g$ as a column vector, we can write
\be\label{trasferta1}
\t_g x = x + V g\,,  \qquad V:=\left[v_1|v_2|\dots|v_d\right]
\en
($V$ is the $d\times d $--matrix with columns $v_1,v_2, \dots, v_d$).

$\bullet$ {\sl Orbits and representatives}.
 We set 
\be\label{simplesso}
\D:= \{t_1 v_1 + \cdots + t_d v_d \,:\, (t_1,\cdots, t_d) \in [0,1)^d\}\,.
\en
Given $x\in \bbR^d$, the  orbit of $x$ is set $\{\t_g x\,:\, g\in \bbG\}$. 

If $\bbG=\bbR^d$, then the orbit  of the origin of $\bbR^d$ equals $\bbR^d$ and we set
\be\label{attimino}
g(x):=g \text{ if } x= \t_g 0\,.
\en 
When $V=\bbI$ (as in many applications), we have $\t_g x=x+g$ and $g(x)=x$.

If  $\bbG=\bbZ^d$,  $\D$ is a set of  orbit representatives
for the action $(\t_g)_{g\in \bbG}$. We introduce the functions $\b: \bbR^d \to \D$ and $g: \bbR^d \to \bbG$ as follows:
\be\label{attimo}
x= \t_g a \text{ and } a\in \D \; \Longrightarrow \; \b(x):=a \,, \; g(x) := g\,.
\en

$\bullet$ {\sl Induced action of $\bbG$ on $\cM$}.
We denote by $\cM$ the metric space of locally finite nonnegative measures on $\bbR^d$ \cite[App.~A2.6]{DV1}.  
The action of $\bbG$ on $\bbR^d$ naturally induces an action of $\bbG$ on  $\cM$, which (with some abuse of notation) we still denote by $(\t_g)_{g\in \bbG}$. In particular,  $\t_g : \cM\to \cM$ is given by 
 \be\label{mignolo}
 \t_g \mathfrak{m} (A):= \mathfrak{m} (\t_{g} A)\,, \qquad \forall A\in \cB(\bbR^d)\,.
  \en
Setting  $ \mathfrak{m}[f]:=\int f d \mathfrak{m}$ for all $\mathfrak{m}\in \cM $, we get 
$
\t_g \mathfrak{m}[ f]= \int  f( \t_{-g} x )d \mathfrak{m}(x)$.
\subsection{$\bbG$--stationary random measure $\mu_\o$}
We suppose now to have a  random  \ovo{locally finite nonnegative} measure $\mu_\o$ \ovo{on $\bbR^d$}, i.e. a measurable  map $\O \ni \o \mapsto \mu_\o \in \cM$. We assume that $\mu_\o$ is a \ovo{purely atomic (i.e. pure point)} measure  \ovo{with locally finite support}  for any $\o\in \O$. In particular, we have 
\be\label{cibo}
\mu_\o = \sum _{x\in \hat \o} n_x(\o) \d_{x}\,, \;\;\; \; n_x(\o) :=\mu_\o(\{x\})\,,\;\;\;\;\hat \o:=\{x \in \bbR^d\,:\, n_x(\o) >0\}
\en
\ovo{and $\hat \o$ is  a locally finite  set. The map $\o\mapsto \hat \o$ then defines a simple point process}. 
 The fundamental relation between the above two actions of $\bbG$ and  the random measure $\mu_\o$ is given by the assumption that $\mu_\o$ is  $\bbG$--stationary (cf. \cite[Section~2.4.2]{Ge}, \cite[Eq.~(21)]{GL}): $
\mu_{\theta_g\o}= \t_g \mu_\o$ for $\cP$--a.a.  $ \o \in \O$ and for all  $g \in \bbG$.


\subsection{Palm distribution}\label{subsec_palm} We introduce the Palm distribution $\cP_0$ by distinguishing between two main cases and a special subcase. We will write $\bbE[\cdot]$ and $\bbE_0[\cdot]$ for the expectation w.r.t. $\cP$ and $\cP_0$, respectively\footnote{With some abuse, when $f$ has  a complex form, we will write $\bbE[f(\o)]$ instead of $\bbE[f]$, and similarly for $\bbE_0[f(\o)]$.}.

$\bullet$ 
{\sl Case $\bbG=\bbR^d$}. The intensity of the random measure $\mu_\o$ is defined as
\be\label{puffo1}
m:= \bbE\left[ \mu_\o \left( [0,1)^d \right)\right]\,.\en
As stated below, $m$ is assumed to be finite and positive.
By the $\bbG$-stationarity of $\cP$ we have $m \ell(U)=\bbE\left[ \mu_\o \left( U\right)\right]$ for any  $U\in \cB (\bbR^d)$.
Then  (see e.g. \cite{DV2,Ge,Km} and  \ovo{Appendixes \ref{app_bell},} \ref{gennaro}) the Palm distribution $\cP_0$ is the probability measure on $(\O,\cF)$ such that, for any $U\in \cB(\bbR^d)$  with $0<\ell (U)<\infty$, 
\be\label{palm_classica}
\cP_0(A):=\frac{1}{m  \ell(U) }\int _\O d\cP(\o) \int _{U} d\mu_\o (x)  \mathds{1}_A(\theta_{g(x)} \o)\,, \qquad \forall A\in \cF
\,.\en
$\cP_0$ has support inside the set $ \O_0:=\{\o\in \O\,:\, n_0(\o)>0\}$ (see Remark \ref{jolie}).

%
%
%
 %
 %

$\bullet$ {\sl  Case $\bbG=\bbZ^d$}.
The intensity of the random measure $\mu_\o$ is defined as
\be\label{puffo2}
m:= \bbE\left[ \mu_\o \left( \D \right)\right]/ \ell(\D) \,.
\en
 By the $\bbG$-stationarity of $\cP$,  $m \ell(A)=\bbE\left[ \mu_\o \left( A \right)\right]$ for any $A\in \cB(\bbR^d)$ which is an overlap of translated cells $\t_g \D$ with $g\in \bbG$. As stated below, $m$ is assumed to be finite and positive. 
 Then (see \ovo{Appendixes \ref{app_bell},} \ref{gennaro}) the  Palm distribution $\cP_0$  is the probability measure on $\left(\O\times \D,\cF\otimes \cB(\D)\right)$ such that
\be\label{Palm_Z}
\cP_0(A):=\frac{1}{m\, \ell(\D)}\int _\O d\cP(\o) \int _{\D} d\mu_\o (x)  \mathds{1}_A(  \o,x )\,, \qquad \forall A\in \cF\otimes \cB(\D)
\,.\en
 $\cP_0$ has support inside $\O_0:=\{(\o, x)\in \O\times \D\,:\,n_x(\o)>0\}$ (see Remark \ref{jolie}).

%
%
%
%


$\bullet$ {\sl Special discrete case: $\bbG=\bbZ^d$, $V=\bbI$ and $\hat \o \subset \bbZ^d$  $\forall \o\in \O$}. This is a subcase of the case $\bbG=\bbZ^d$ and in what follows we will call it simply  \emph{special discrete case}. As this case is very frequent in discrete probability, we discuss it apart pointing out some simplifications. As $\D=[0,1)^d$ intersects $\bbZ^d$ only at the origin, $\cP_0$ (see case   $\bbG=\bbZ^d$) is concentrated on $\{ \o \in \O:n_0(\o)>0\}\times \{0\}$. Hence  we can think of $\cP_0$ as a probability measure 
concentrated on the set $\O_0:=\{ \o \in \O:n_0(\o)>0\}$.
Formulas \eqref{puffo2} and \eqref{Palm_Z} 
then read
\be\label{zazzera}
 m:= \bbE[n_0]\,, \qquad 
 \cP_0(A):= \bbE\left[ n_0\,   \mathds{1}_A\right] / \bbE[n_0] \qquad \forall A\in \cF
 \,.
 \en
In what follows, when treating the special discrete case, we will use the above identifications without explicit mention.

%
%
%
%


\subsection{Rate jumps and assumptions}\label{tassometro}
All objects introduced so far concern the environment and not the particle dynamics. The latter is encoded in the \rrr{measurable} function 
\be r: \O\times \bbR^d \times \bbR^d \ni (\o, x, y)\mapsto r_{x,y}(\o) \in [0,+\infty)\,.\en 
As it will be clear below, only the value of $ r_{x,y}(\o)$ with $x\not =y$ in  $ \hat \o$ will be relevant.  Hence, 
without loss of generality, we take  
\be \label{lino}
r_{x,x}(\o)\equiv 0\,, \qquad  r_{x,y}(\o)\equiv 0 \qquad \forall  \{x,y\}\not \subset \hat \o\,.
\en

By identifying  the support $\hat \o$ of $\mu_\o$ with the measure $\sum_{x\in \hat \o} \d_x$,  
we define the function 
$\l_k:\O_0 \to [0,+\infty]$  (for $k\in[0,\infty)$) as follows:\\
 \begin{equation}\label{altino15}
 \begin{split}
 & 
 \begin{cases}
  \l_k(\o):=\int _{\bbR^d} d\hat \o  (x) r_{0,x}(\o)|x|^k\\
 \O_0=\{\o\in \O\,:\, n_0(\o)>0\}
 \end{cases}  \qquad
\Large{\substack{ \text{\;\;Case $\bbG=\bbR^d$ and}\\\text{\;\;\;\;\;\;special discrete case}}\,,}
\\
& \begin{cases}
  \l_k(\o,a):=
 \int _{\bbR^d} d\hat \o  (x) r_{a,x}(\o)|x-a|^k   \\
   \O_0:=\{(\o, x)\in \O\times \D\,:\,n_x(\o)>0\}
\end{cases}
 \text{\;\;Case $\bbG=\bbZ^d$\,.}
\end{split}
\end{equation}

%

\medskip

We list our  assumptions (including the ones introduced above).

\smallskip
\noindent
{\bf Assumptions}. \emph{We make the following assumptions \utto{where $\O_*$ is some translation invariant measurable subset of $\O$ with $\cP(\O_*)=1$}:
\begin{itemize}
\item[(A1)] $\cP$ is stationary and ergodic w.r.t. the action $(\theta_g)_{g\in \bbG}$ of the group $\bbG$; 
 \item[(A2)] the  $\cP$--intensity $m$ of the random measure  $\mu_ \o $ 
   is finite and positive;
\item[(A3)]  \utto{for all $\o \in \O_*$  and for all  $g\not =g' $ in $ \bbG $, it holds 
  $ \theta_g\o\not = \theta _{g'} \o$};
\item[(A4)]  \utto{for all $\o \in \O_*$, 
for all   $g\in \bbG$ and  $x,y \in \bbR^d$, it holds  }
\begin{align}
& \mu_{\theta_g \o}= \t_g \mu_\o \label{base}\\
& r_{x,y} (\theta_g\o)= r_{\t_g x, \t_g y} (\o) \,;\label{montagna}
 \end{align}
 \item[(A5)]  \utto{for all $\o \in \O_*$  and} for all  $x,y\in \ovo{\hat \o}$,  it holds \be\label{rubini99}
c_{x,y}(\o):=n_x(\o)  r_{x,y}(\o) = n_y(\o) r_{y,x}(\o)\,;
\en
\item[(A6)]  \utto{for all $\o \in \O_*$ and for all} $x,y \in  \hat \o$, there exists a  path $x=x_0$, $x_1$,$ \dots, x_{n-1}, x_n =y$ such that $x_i \in \hat \o$ and  $r_{x_i, x_{i+1}}(\o) >0$ for all $i=0,1, \dots, n-1$;
\item[(A7)]   $ \l_0, \l_2 \in L^1(\cP_0)$;
  \item[(A8)]  $L^2(\cP_0)$ is separable.
\end{itemize}
}

\ovo{We conclude this section with some comments on the above assumptions}.
We observe that, by  Zero-Infinity Dichotomy (see \cite[Proposition~12.1.VI]{DV2}) and Assumptions (A1), (A2), for $\cP$--a.a.~$\o$ the support $\hat \o$ of  $\mu_\o$ is an infinite set.

\utto{For $k=3,4,5,6$ we  call $\O_k$ the set of environments $\o$ satisfying the properties stated in Assumption (Ak) (for example $\O_3:=\{\o \in \O\,:\, \theta_g \o \not = \theta_{g'} \o \; \forall g\not =g' \text{ in } \bbG\}$). All $\O_k$'s are always translation invariant. If $\bbG=\bbZ^d$,  they are also measurable. Therefore,  for $\bbG=\bbZ^d$, 
we can simply take   $\O_*:=\cap_{k=3}^6\O_k$, which is automatically  a measurable and translation invariant set. 
  When  $\bbG=\bbR^d$  and, as common in applications,  $\O_4=\O_5=\O_6=\O$, we can prove that $\O_3$ is  measurable. Therefore, in the above case, we can simply take   $\O_*:=\O_3$, which is automatically a  measurable and translation invariant set.  The above proof and  the discussion of further cases are provided  in Appendix \ref{misurino78}.}

We point out that (A3) is \ovo{usually} a rather fictitious assumption. Indeed, by free one can   add  some randomness  enlarging $\O$ to assure (A3) (similarly to \cite[Remark 4.2-(i)]{demasi}). For example, if $(\O,\cF, \cP)$ describes  a random simple point process on $\bbR^d$ obtained by periodizing a random simple point process on $[0,1]^d$, then to gain (A3) it would be enough to mark points by i.i.d. random variables with non-degenerate distribution. 

Considering  the random walk $X_t^\o$ introduced in Section \ref{nicolino} below, (A5) and   (A6) correspond   to  reversibility of the measure $\mu_\o$ and to irreducibility \utto{for all $\o\in \O_*$}. 

We observe that (A7) implies
\be\label{maracas}
\bbE_0[\l_1]\leq \bbE_0[\l_0]+ \bbE_0[\l_2]<+\infty\,.
\en

We point out that,    by  \cite[Theorem~4.13]{Br}, (A8) is fulfilled if $(\O_0,\cF_0,\cP_0)$ is a separable  measure space where $\cF_0:=\{A\cap \O_0\,:\, A\in \cF\}$ (i.e. there is a countable family $\cG\subset  \cF_0$ such that  the $\s$--algebra  $\cF_0$ is generated by $\cG$). For example, if $\O_0$ is a separable metric space and 
$\cF_0= \cB(\O_0)$ (which is valid if $\O$ is a separable metric space and 
$\cF= \cB(\O)$) then (cf. \cite[p.~98]{Br}) $(\O_0,\cF_0,\cP_0)$ is a separable  measure space  and (A8) is valid.

\begin{Remark}\label{jolie}  We report some useful identities \ovo{restricting to the   $\o$'s which fulfill the properties in (A4) and (A5) (i.e. $\o\in \O_4\cap \O_5$ with the above notation).}  For all $a,b \in \bbR^d$ it holds $\t_{g(a)} b= a+b$. For all $x \in \bbR^d$ and $g\in \bbG$
it holds $
n_x(\theta _g \o)=n_{\t_g x}(\o)$. Given $a,b \in \hat \o$ we have 
$n_a(\o)= n_0 ( \theta_{g(a)} \o )$, $r_{a,b}(\o) = r_{0,b-a}  ( \theta_{g(a)} \o )$,  $c_{a,b}(\o) = c_{0,b-a} ( \theta_{g(a)} \o)$ and $\widehat{\theta_{g(a)} \o}= \hat \o -a$.
\end{Remark}

\section{Massive Poisson equation, random walk, homogenized matrix, homogenized equation, convergence in $L^2(\mu_\o^\e)$, $L^2(\nu_\o^\e)$}\label{setting2}
Recall that we identify  the support $\hat \o$ of $\mu_\o$ with the measure $\sum_{x\in \hat \o} \d_x$. 
\subsection{Measures $\mu_\o^\e$ and $\nu_\o^\e$} Given  $\e>0$ and $\o \in \O$,  we define  $\mu _\o^\e $ as the \ovo{measure} on $\bbR^d$  (cf. \eqref{cibo})
 \begin{equation}\label{franci}
 \mu_\o^\e :=     \sum _{ x\in \hat \o}  \e^d n_x(\o)  \d_{\e x}\,.\end{equation}
\ovo{We write $\la\cdot, \cdot \ra_{\mu^\e_\o}$ for the scalar product in $L^2(\mu^\e_\o)$.}
By ergodicity, for  
 $\cP$--a.a. $\o $ the measure $ \mu _\o^\e $ converges vaguely to  $m dx $,  where $m$ is the intensity of $\mu_\o$. The above convergence is a special case  of  the following  stronger ergodicity result,    where the interplay   between the  microscale and macroscale   emerges:
 \begin{Proposition}\label{prop_ergodico} 
 Let  $f: \O_0\to \bbR$ be a \rrr{measurable} function with $\|f\|_{L^1(\cP_0)}<\infty$. Then there exists a translation invariant   measurable subset $\cA[f]\subset \O$  such that $\cP(\cA[f])=1$ and such that,  for any $\o\in \cA[f]$ and any  $\varphi \in C_c (\bbR^d)$, it holds
\begin{equation}\label{limitone}
\lim_{\e\da 0} \int  d  \mu_\o^\e  (x)  \varphi (x ) f(\theta_{g( x/\e)} \o )=
\int  dx\,m\varphi (x) \cdot \bbE_0[f]\,.
\end{equation}
\end{Proposition}
Note that  in \eqref{limitone} $x/\e$ and $x$ are respectively at the microscopic and macroscopic scale. The proof of Proposition~\ref{prop_ergodico} can be obtained by standard arguments from \cite{T} (see also \cite[App.~B]{Fhom}).

  We   define $\nu _\o ^\e $ as the \ovo{measure} on $\bbR^d\times \bbR^d$ given by (cf. Remark \ref{jolie})
\begin{equation}\label{nunu}
\begin{split}
 \nu_\o^\e
:&=
 \e ^d \int d \hat \o (a) \int d \hat \o (b) c_{a,b}(\o) \d_{(\e a, b-a)}\\
&=
  \int d \mu^\e_\o (x) \int d(
  \widehat{\theta_{g(x/\e)}\o} )(z)  
 r_{0,z}( \theta_{g(x/\e)}\o)   \d _{(x  ,z)}\,.
\end{split}
 \end{equation}
 \ovo{We write $\la\cdot, \cdot \ra_{\nu^\e_\o}$ for the scalar product in $L^2(\nu^\e_\o)$.}
\subsection{Microscopic gradient and space $H^1_{\o,\e}$}

 Given $\o \in \O$  and a real function $v$ whose domain contains  $\e \hat \o$,  we define the \emph{microscopic gradient} $\nabla_\e  v $  as the function 
 \begin{equation}\label{ricola}
 \nabla_\e v (x,z)= \frac{ v(x+\e z)- v(x)  }{\e}\,, \qquad x \text{ and } x+\e z \in \e \hat \o\,.
 \end{equation}
 By Remark \ref{jolie},  \ovo{given $x\in \e \hat \o$},  $ x+\e z\in \e \hat \o$   if and only if 
   $ z \in \widehat{\theta_{g(x/\e)}\o}$.
 \begin{Definition}\label{bic}
 We say that $v \in  H^1_{\o,\e}  $ if $v \in  L^2(\mu _\o ^\e) $ and $\nabla_\e  v  \in L^2( \nu_\o ^\e)$. Moreover, we endow the space  $H^1_{\o,\e} $ \ovo{with the scalar product  $\la v,w \ra _{H^1_{\o, \e} }:=\la v, w\ra _{ \mu _\o ^\e }+\la \nabla v,\nabla w\ra _{  \nu_\o ^\e}$ and write $ \|\cdot\|_{H^1_{\o, \e} }$ for the norm in $ H^1_{\o,\e} $.}
 \end{Definition}
 It is simple to check that $H^1_{\o,\e}  $ is a Hilbert space.

\subsection{\ovo{Space $H^{1,{\rm f}}_{\o,\e}$},  massive Poisson equation, self-adjoint operator $\bbL^\e _\o$, random walks $X_t^\o$ and $\e X_{\e^{-2}t} ^\o$}\label{nicolino}
We introduce the set
\begin{equation}\label{alba_chiara}
\begin{split}
\O_1:=\{\o\in \O\,:\, &\ovo{r_x(\o):=}\sum_{y \in \hat \o} r_{x,y}(\o) <\infty \; \forall x \in \hat \o\,,\\
& \ovo{ c_{x,y}(\o)=c_{y,x}(\o) \;\forall x,y \in \hat \o}\}\,.
\end{split}
\end{equation}
Then the set  $\O_1\subset \O $ is translation invariant  and it holds $\cP(\O_1=1)$ as can be checked  at cost to  reduce to the case $\bbG=\bbR^d$ by the method outlined in Section \ref{sec_fluidifico} and applying Corollary \ref{eleonora} below  together with the bound $\bbE_0[\l_0]<\infty$. 
We restrict to $\o\in \O_1$.  We call  $\cC (\e \hat \o)$ the set of real functions on $\e\hat \o$ which  are zero outside a finite set. Then $\cC(\e \hat \o)\subset H^1_{\o, \e}$. \ovo{Indeed,  by symmetry of $c_{x,y}(\o)$, for any $v:\e \hat \o\to \bbR$ it holds}
\be\label{aldo85}
\begin{split}
 \ovo{\e^{2} \int d \nu^\e_\o (x,z) \nabla_\e v(x,z)^2}
 & \ovo{\leq 
2 \e^d  \int d \hat \o(x) \int d\hat \o(y) c_{x,y}(\o) [v(\e x )^2 + v(\e y)^2]}\\
& \ovo{ = 4  \e^d \int d\mu_\o (x)  r_x(\o) v(\e x)^2 \,.}
\end{split}
\en
\ovo{When $v\in \cC(\e \hat \o)$, the last term is a finite sum (as $\hat \o$ is locally finite and $\o\in \O_1$)}.
\begin{Definition} \ovo{Given $\o \in \O_1$, the Hilbert space $H^{1,{\rm f}}_{\o, \e}$ is defined as the closure of $\cC(\e \hat \o)$ inside the Hilbert space $H^1_{\o, \e}$.}
\end{Definition}
\ovo{The index ${\rm f}$ in $H^{1,{\rm f}}_{\o, \e}$ refers to finite support, as the functions in $\cC(\e \hat \o)$ are the ones with finite support on $\e \hat \o$}. The symmetric form $(f,h) \mapsto  \frac{1}{2} \la \nabla_\e f, \nabla_\e h \ra _{\nu_\o ^\e}$ with domain  $\ovo{H^{1,{\rm f}}_{\o, \e}}\subset L^2 (\mu^\e_\o)$ is a regular Dirichlet form with core $\cC(\e \hat \o)$ \ovo{(consider \cite[Example~1.2.5]{FOT} with  $E, q_{x,y}, k_x, m^0,m$ there defined as 
$E:=\e \hat{\o}$, $q_{x,y}:= \e^{-2} r_{x/\e,y/\e}$ for $x\not=y$ in $\e \hat \o$, $q_{x,x}:= -\sum_{y:y\not =x} q_{x,y}$, $k_x:=0$, $m^0_x:= \e^d n_{x/\e}(\o)$, $m:=m^0$)}. In particular, there exists a unique nonpositive  \ovo{definite} self-adjoint operator  $\bbL^\e_\o$   in $L^2(\mu^\e_\o)$    such  that $\ovo{H^{1,{\rm f}}_{\o,\e}}$ equals the domain of $\sqrt{-{\bbL}^\e_\o}$ and  $ \frac{1}{2} \la \nabla_\e f, \nabla_\e f \ra _{\nu_\o ^\e}= \|\sqrt{-\bbL^\e_\o} f\|^2_{L^2(\mu^\e_\o)}$ for any $f \in \ovo{H^{1,{\rm f}}_{\o, \e}}$ (see
 \cite[Theorem~1.3.1]{FOT}).
 \ovo{Note that}, if $h \in \cD( \bbL_\o^\e)\ovo{\subset \cD(\sqrt{-{\bbL}^\e_\o})=  H^{1,{\rm f}}_{\o, \e}}$ and  $f\in \ovo{H^{1,{\rm f}}_{\o, \e}}$,  we have 
\begin{align}
\la f, - \bbL ^\e_\o h \ra _{ \mu^\e_\o}  & =
\frac{\ovo{\e^{d-2}}}{2} \int d\hat \o (x) \int d\hat \o (y) c_{x,y}(\o) \bigl( f(\e y)-f(\e x) \bigr) \bigl( h(\e y)-h(\e x)\bigr)\nonumber \\
&= \frac{1}{2} \la \nabla_\e f, \nabla_\e h \ra _{\nu_\o ^\e}\,.\label{dir_form}
\end{align}

%
%
 Identity \eqref{dir_form} suggests a weak formulation of the equation
$-\bbL_\o^\e u+\l u =f$:
\begin{Definition} \ovo{Let $\o \in \O_1$.}
Given $f\in L^2 (\mu_\o^\e)$  and $\l>0$,  a weak solution $u$ of the equation
\begin{equation}\label{strepitio}
-\bbL^\e_\o u+\l u =f
\end{equation}
is a function $u \in \ovo{H^{1,{\rm f}}_{\o,\e}}$ such that 
\begin{equation}\label{salvoep}
\frac{1}{2} \la \nabla_\e v, \nabla_\e u \ra _{\nu_\o ^\e}+\l  \la v, u \ra _{\mu_\o^\e}= \la v, f \ra_{\mu_\o^\e}\qquad \forall v \in \ovo{H^{1,{\rm f}}_{\o,\e}} \,.
\end{equation} 
\end{Definition}
By the  Lax--Milgram theorem \cite{Br}, given  $f\in L^2 ( \mu_\o^\e)$  the weak  solution $u $ of \eqref{strepitio} exists and is unique.
 As any $\l>0$ belongs to the resolvent set of the nonpositive self-adjoint operator $\bbL^\e_\o$, equation \eqref{strepitio} is equivalent to $u=(-\bbL^\e_\o +\l)^{-1} f$.

We point out that  \ovo{(cf. \cite[Lemma 1.3.2 and  Exercise~4.4.1]{FOT})} the self-adjoint operator $\bbL^\e_\o$  is the infinitesimal generator of the \ovo{strongly continuous} Markov semigroup  in $L^2(\mu^\e_\o)$ associated to the diffusively rescaled random walk $(\e X_{\e^{-2}t} ^\o)_{t\geq 0}$, $X^\o _t$  being the random walk on $\hat \o $ with 
 probability rate $r_{x,y}(\o)$  for a jump from $x $ to $ y$ in $ \hat \o$ (possibly with explosion).  
 We can  indeed show (cf.  Appendix \ref{app_no_TNT})  that, for $\cP$--a.a. $\o$, explosion does not take place \ovo{(one needs weaker assumptions for this result)}: 
  \begin{Lemma}\label{lemma_no_TNT} Assume (A1),...,(A6),   $\l_0\in L^1(\cP_0)$ \utto{(for some translation invariant measurable set $\O_*$ with $\cP(\O_*)=1$)}. Then there exists  a translation invariant \ovo{measurable} set $\cA \subset\O$ with $\cP(\cA)=1$ such that, for all $\o \in \cA$, (i) $r_x(\o):=\sum_{y\in \hat \o} r_{x,y}(\o)\in (0,+\infty)$  $\forall x \in \hat \o$, (ii)  the  continuous--time Markov chain on $\hat \o$ starting at any $x_0\in \hat \o$, with waiting time parameter $r_x(\o)$ at $x\in \hat \o$ and  with  probability $r_{x,y}(\o)/r_x(\o) $ for a jump from $x$ to $y$, is non-explosive.
\end{Lemma}


\subsection{Homogenized matrix $D$ and  homogenized equation}\label{diario_schiappa}
\begin{Definition}\label{def_D} We define the \emph{homogenized  matrix}  $D$  as 
the unique $d\times d$ symmetric matrix such that:

\smallskip

$\bullet$ {\bf Case $\bbG=\bbR^d$ and special discrete case\footnote{In the special discrete case, 
$\int_{\bbR^d} d\hat \o (x) $ can be replaced by  a sum among $x\in \bbZ^d$.}}
 \begin{equation}\label{def_D_R}
 a \cdot Da =\inf _{ f\in L^\infty(\cP_0) } \frac{1}{2}\int _{\O_0} d\cP_0(\o)\int_{\bbR^d} d\hat \o (x) r_{0,x}(\o) \left
 (a\cdot x - \nabla f (\o, x) 
\right)^2\,,
 \end{equation}
 for any $a\in \bbR^d$, 
 where $\nabla f (\o, x) := f(\theta_{g(x)} \o) - f(\o)$.
 
 \smallskip
 
$\bullet$   {\bf Case $\bbG=\bbZ^d$ }
\begin{align}
& a \cdot Da= \label{def_D_Z}  \\
& \inf _{ f\in L^\infty(\cP_0) } \frac{1}{2} \int_{\O\times \D}d\cP_0(\o,x)  \int_{\bbR^d} d\hat \o (y) r_{x,y}(\o) \left
 (a\cdot (y-x)  -
 \nabla f (\o, x,y-x)
\right)^2\,,\nonumber
 \end{align}
 for any $a\in \bbR^d$,
 where $\nabla f (\o, x,y-x) := f(\theta_{g(y)} \o, \beta(y) ) - f(\o,x)$.
 \end{Definition}
%
%
%
\ovo{Since $\l_2 \in L^1(\cP_0)$ by (A7),   $a\cdot Da$ is indeed finite for any $a\in \bbR^d$.}

\begin{Remark}\label{virale}
\ovo{Consider  the new random measure $\tilde{\mu}_\o:= \hat \o$ and the new  rates   $\tilde{r}_{x,y}(\o):= c_{x,y}(\o)$, under the assumption that the intensity $\tilde{m} $ of $\tilde{\mu}_\o$ is finite and positive. One can then check that Assumptions (A1),...,(A8) imply  the analogous assumptions in the new setting with   $\tilde{\mu}_\o$ and $\tilde{r}_{x,y}(\o)$. Moreover, writing $\tilde{\cP}_0$ and $\tilde{D}$ for the Palm distribution and the effective homogenized matrix in the new setting  respectively, one easily gets  that $m\, d \cP_0(\o) = \tilde{m} \, n_0(\o ) d\tilde{\cP}_0(\o)$ for $\bbG=\bbR^d$ and in the special discrete case,    $m\,d \cP_0(\o,x) = \tilde{m}\, n_x(\o ) d\tilde{\cP}_0(\o,x)$ for $\bbG=\bbZ^d$ and $m D= \tilde{m} \tilde{D}$. As a consequence we get that   $m/\tilde{m}= \tilde{\bbE}_0[ n_0]$, $\tilde{\bbE}_0[\cdot]$ being the expectation w.r.t. $\tilde{\cP}_0$.}
\end{Remark}

\begin{Definition}\label{divido}
We fix  an orthonormal basis
$\mathfrak{e}_1$,...,$ \mathfrak{e}_d$ of  eigenvectors of $D$ (which  is symmetric) and  we let   $\g_i$ be the eigenvalue of $ \mathfrak{e}_i$.  
At  cost of a relabelling,  
 $ \g_1,\dots, \g_{d_*}$ are all positive  and  $ \g_{d_*+1},\dots, \g_{d}$  are all zero.   In particular, $d_*\in \{0,1,\dots, d\}$ and,  if $D$ is strictly positive, it holds $d_*=d$. Given a \ovo{unit} vector $v$,  we write $\partial_v f $ for the weak derivative of $f$ along the direction  $v$ (if $v=e_i$, then $\partial_v f $ is simply the standard weak derivative $\partial_i f$). \end{Definition}

 \begin{Definition}\label{degenerare}
 We introduce the space  $H^1_*(m dx)$ given by the functions $f \in L^2(m dx)$  such that the weak derivative $\partial_{\mathfrak{e}_i} f$ belongs to  $L^2(m dx)$  $\forall i\in \{1,\dots, d_*\}$.
  \ovo{We endow $H^1_*(m dx)$ with the scalar product $\la f, h \ra  _{H^1_*( m dx)}:= \la f, h \ra_{L^2(mdx)} + \sum_{i=1}^{d_*} \la \partial _{\mathfrak{e}_i} f, \partial _{\mathfrak{e}_i} h \ra_{L^2(m dx)} $}.  
 Moreover, given $f\in H^1_*(m dx)$, we set
 \be
 \nabla_* f:= \sum_{i=1}^{d_*}  \left( \partial _{\mathfrak{e}_i} f \right) \mathfrak{e}_i\in L^2(m dx)^d\,.
 \en 
\end{Definition} 
We stress  that, if $\mathfrak{e}_i = e_i$ for all $i=1,\dots, d_*$ (as in many applications), then 
\be
 \nabla_* f:= (\partial_1 f, \dots, \partial _{d_*} f , 0,\dots, 0)\in L^2(m dx)^d\,.
 \en 
We point out that $H^1_*(m dx)$ is an Hilbert space (adapt the standard proof for $H^1(dx)$). Moreover, $C_c^\infty (\bbR^d)$ is dense in $H^1_*(m dx)$ (adapt the arguments in the proof of \cite[Thm.~9.2]{Br}).  

We now move to the effective  homogenized equation, where $D$ denotes the  homogenized matrix introduced in Definition  \ref{def_D}.
\begin{Definition}
Given  $f\in L^2 (m dx)$  and $\l>0$, a weak solution $u$ of the equation
\begin{equation}\label{strepitioeff}
- \nabla_* \cdot  D \nabla_*  u+\l u =f
\end{equation}
is a function $u \in  H^1_*( m dx) $ such that 
\begin{equation}\label{delizia}
   \int   D  \nabla_* v(x) \cdot   \nabla_* u (x) dx + \l \int   v(x) u(x) dx = \int  v(x) f(x)dx  \,, \qquad \forall v \in H^1_*( m dx )\,.
\end{equation} 
\end{Definition}
 Again, by the   Lax--Milgram theorem, given  $f\in  L^2 (m dx)$  the weak  solution $u $ of \eqref{strepitio} exists and is unique. 


\subsection{Weak and strong convergence for $L^2(\mu_\o^\e)$ and $L^2(\nu_\o^\e)$}
 \begin{Definition}\label{debole_forte} Fix $\o\in \O$ and a 
 family of $\e$--parametrized  functions $v_\e \in L^2( \mu^\e_\o)$. We say  that the family $\{v_\e\}$  \emph{converges weakly} to the function  $v\in L^2( m dx)$, and write  $v_\e \rightharpoonup v$, if
\begin{equation}\label{recinto}
\limsup  _{\e \da 0} \| v_\e\| _{L^2(\mu^\e_\o) } <+\infty
\end{equation}and 
\begin{equation}\label{deboluccio}
\lim _{\e \da 0} \int d  \mu^\e _\o (x)   v_\e (x) \varphi (x)= 
\int dx\,m  v(x) \varphi(x) \end{equation}
for all $\varphi \in C_c(\bbR^d)$. 
We  say that  the  family $\{v_\e\}$  \emph{converges strongly} to  $v\in L^2( m dx)$, and write $v_\e\to v$,  if  in addition to \eqref{recinto}  it holds
\begin{equation}\label{fortezza}
\lim _{\e \da 0} \int  d  \mu^\e _\o (x)   v_\e (x) g_\e (x)= 
\int dx\, m  v(x) g(x) \,,\end{equation}
for any family of functions $g_\e \in  L^2(\mu^\e_\o)$ weakly converging to $g\in L^2( m dx) $.
\end{Definition}

In general, when \eqref{recinto} is satisfied, one simply says that the family $\{v_\e\}$ is bounded.

\begin{Remark}\label{forte} One  can prove (cf.~\cite[Prop.~1.1]{Z}) that $v_\e \to v$ if and only if 
$v_\e \rightharpoonup v$ and $\lim _{\e\downarrow 0} \int_{\bbR^d}
 v_\e(x)^2 d \mu_\o^\e(x)=
\int _{\bbR^d} v(x)^2 m dx $.
\end{Remark}
 
 
 We introduce now a special form of convergence of microscopic  gradients (note that the testing  objects are gradients as \ovo{well}).
\begin{Definition}\label{debole_forte_grad}   Fix $\o\in \O$ and a 
 family of $\e$--parametrized  functions $v_\e \in L^2( \mu^\e_\o)$. We say  that the family $\{\nabla_\e v_\e\}$  \emph{converges weakly} to the vector-valued function  $w$   belonging to the product space $ L^2( m dx)^d$  and with values in $\bbR^{d}$, and write  $\nabla_\e v_\e \rightharpoonup w$, if
\begin{equation}\label{recinto_grad}
\limsup  _{\e \da 0} \| \nabla_\e v_\e\| _{L^2(\nu^\e_\o) } <+\infty
\end{equation}and 
\begin{equation}\label{deboluccio_grad}
\lim _{\e \da 0} \frac{1}{2} \int   d \nu^\e _\o (x,z ) \nabla_\e v_\e (x,z)  \nabla_\e \varphi  (x,z )  = 
\int   dx\,m D w(x) \cdot  \nabla_* \varphi(x)   \end{equation}
for all $\varphi \in C^1_c(\bbR^d)$. 
We  say that   family $\{\nabla _\e v_\e\}$  \emph{converges strongly} to  $w$ as above, and write $\nabla_\e v_\e\to w$,  if  in addition to \eqref{recinto_grad}  it holds
\begin{equation}\label{fortezza_grad}
\lim _{\e \da 0}  \frac{1}{2} \int d  \nu^\e _\o (x,z)    \nabla_\e v_\e (x,z)  \nabla_\e g_\e (x,z )  = 
\int  dx\,m D w(x) \cdot  \nabla_* g(x) 
 \end{equation}
for any family of functions $g_\e \in  L^2(\mu^\e_\o)$ with $g_\e \toup g \in L^2(m dx)$ such that $g_\e \in \ovo{H^{1,{\rm f}}_{\o,\e}}$ and $g\in   H^1_* ( m dx)$.
\end{Definition}
\begin{Remark}\label{sanremo2022}
Denoting by $\varphi_\e$  the restriction of $\varphi$ to $\e \hat \o$, 
\ovo{for all $\o \in \O_1$  any $\varphi\in C_c(\bbR^d)$ has the property that   $\varphi _\e \in \cC(\e \hat \o) \subset  H^{1,{\rm f}}_{\o,\e}$, as $\hat \o$ is locally finite. Moreover, given $\o \in \cA[1]$ (cf. Prop.~\ref{prop_ergodico}), by Remark \ref{forte}  we get that}  $L^2(\mu^\e_\o)\ni \varphi_\e \to \varphi \in L^2( mdx )$. In particular, for \ovo{environments $\o \in \O_1\cap \cA[1]$ (as the ones in $\O_{\rm typ}$ appearing in Theorem \ref{teo1}}),  if $\nabla _\e v_\e \to w$ then $\nabla_\e v_\e\toup w$. 
\end{Remark}

%
%
%
%
%

\section{Main results}\label{sec_MR_hom}

We can now state our  first main result  (recall the definition   of $\O_1$ at the beginning of Section \ref{nicolino}):
\begin{Theorem} \label{teo1} Let Assumptions (A1),...,(A8) be satisfied. Then
there exists a \rrr{measurable}  subset $\O_{\rm typ}\subset \utto{\O_1 \cap \O_*} \subset \O$, of so called \emph{typical environments}, fulfilling the following properties. $\O_{\rm typ}$ is translation invariant  and   $\cP(\O_{\rm typ})=1$. Moreover, given $\o \in \O_{\rm typ}$,  $\l>0$, $f_\e \in L^2(\mu^\e_\o)$ and  \ovo{$f\in L^2( m dx)$},  let  $u_\e$ and $u$ be defined  as the weak solutions,    respectively in $\ovo{H^{1,{\rm f}}_{\o, \e}}$ and  $H^1_*(mdx)$, of the equations 
\begin{align}
&-\bbL^\e_\o u_\e+\l u_\e =f_\e\,,\label{eq1}
\\
&
- \nabla_*\cdot  D \nabla_* u +\l u  =f\,.\label{eq2}
\end{align}
Then we have:\begin{itemize}
\item[(i)] {\bf Convergence of solutions} (cf. Def. \ref{debole_forte}):
\begin{align}
f_\e \toup f \;\Longrightarrow \; u_\e \toup u \,,
\label{limite1}
\\
f_\e \to f  \;\Longrightarrow \; u_\e \to u\,.
\label{limite2}
\end{align}

\item[(ii)]  {\bf Convergence of flows} (cf. Def. \ref{debole_forte_grad}):
\begin{align}
f_\e \toup f \;\Longrightarrow \; \nabla_\e u_\e \toup \nabla_*  u \,,
\label{limite3}
\\
f_\e \to f  \;\Longrightarrow \; \nabla_\e u_\e \to   \nabla_*  u\,.
\label{limite4}
\end{align}

\item[(iii)] {\bf Convergence of energies}:
\begin{equation}\label{limite5}
f_\e \to f  \;\Longrightarrow \; \frac{1}{2}\la \nabla_\e u_\e , \nabla_\e u_\e \ra_{\nu_\o ^\e} \to \int dx\, m  \nabla_* u (x) \cdot D \nabla_* u (x) \,. 
\end{equation}
\end{itemize}
\end{Theorem}

\begin{Remark}\label{gelatino}
Let $\o \in \O_{\rm typ}$.  Then, \ovo{as $\O_{\rm typ}\subset \cA[1]$ (cf. Section \ref{topo})}, \ovo{by Remark \ref{sanremo2022} for any $f \in C_c(\bbR^d)$ it holds}   $L^2(\mu^\e _\o ) \ni f \to f \in L^2(mdx)$. By taking $f_\e:=f$ and using  \eqref{limite2}, we get that   $u_\e \to u$, where $u_\e$ and $u$ are defined as the weak solutions of \eqref{eq1} and \eqref{eq2}, respectively.
 \end{Remark}


 We write $(P^\e _{\o ,t} )_{t \geq 0}$ for the $L^2(\mu^\e _\o)$--Markov semigroup associated to the random walk $(\e X^\o _{ \e^{-2} t} )_{t\geq 0}$ on $\e \hat \o$. In particular, $P^\e _{\o ,t}=e^{t \bbL^\e _\o }$. Similarly we write $( P_t  )_{t \geq 0} $ for the  Markov semigroup on $L^2( m dx)$ associated to the  (possibly degenerate)  Brownian motion on $\bbR^d$  with diffusion matrix $2 D$ given in Definition \ref{def_D}.  Note that  this Brownian motion is not degenerate when projected on $\text{span}(\mathfrak{e}_1, \dots, \mathfrak{e}_{d_*})$, while no motion is present along 
 $\text{span}(\mathfrak{e}_{d_*+1}, \dots, \mathfrak{e}_{d})$.
  In particular, in the case  $\mathfrak{e}_i=e_i$,  writing $p_t(\cdot,\cdot )$ for the probability transition kernel of the Brownian motion on $\bbR^{d_*}$ with non--degenerate diffusion matrix $(2 D_{i,j})_{1\leq i,j \leq d_*}$, it holds 
 \be\label{pesciolini}
 P_t f(x',x'') = \int _{\bbR^{d_*}} p_t ( x', y) f( y, x'' ) dy\,\qquad (x', x'')\in \bbR^{d_*} \times \bbR^{d-d_*}= \bbR^d\,. 
 \en 
Given $\l>0$ we write  $R^\e _{\o ,\l}: L^2(\mu^\e _\o)\to L^2(\mu^\e _\o) $ for the $\l$--resolvent  associated to the random walk $\e X^\o _{ \e^{-2} t} $, i.e.  $R^\e _{\o ,\l}:= (\l -\bbL^\e _\o )^{-1} =\int_0^\infty e^{- \l s} P^\e _{\o ,s} ds $. 
 We write $ R_\l : L^2(m dx) \to L^2(m dx)$ for the $\l$--resolvent associated to the above Brownian motion on $\bbR^d$  with diffusion matrix $2 D$. Note that \eqref{eq1} and \eqref{eq2} can be rewritten as     $u_\e= R^\e _{\o ,\l}  f_\e$  and $u = R_\l f$, respectively. 
   
%
%
%
%


\smallskip
We now show several forms of semigroup and resolvent convergence, whose derivation uses  Theorem \ref{teo1} (they play a fundamental role in  \cite{F_SE_RE}). To this aim we introduce a new Assumption \ovo{(recall definition \eqref{simplesso} of $\D$)}:

\smallskip

\noindent
{\bf Assumption (A9)}: \emph{At least one of the following properties  is fulfilled:
 \begin{itemize}
 \item[(i)] For $\cP$--a.a.~$\o$ 
  $\exists C(\o)>0$ such that $\mu_\o (\ovo{\t_k  \D}) \leq C(\o)$ for all $k\in \bbZ^d$.
  \item[(ii)] Setting $N_k(\o):=\mu_\o (\ovo{\t_k\D})$ for $k\in \bbZ^d$, for some $C_0\geq 0$  it holds     $\bbE[N_0^2]<\infty$ and 
   \be\label{intinto}
| \text{Cov}\,(N_k, N_{k'})| \leq  
  C_0 |k-k'| ^{-1} 
   \en
  for any $k \not = k'$ in $\bbZ^d$. More generally,  we assume that, at cost to enlarge the probability space, one can define random  variables $(N_k) _{k\in \bbZ^d}$ with $\mu_\o (\ovo{\t_k \D})\leq N_k$,  such that $\bbE[N_k],\bbE[N_k^2]$ are bounded uniformly in $k$ and  such that \eqref{intinto} holds for  all $k\not=k'$.
 \end{itemize}
 }
\begin{Remark} \ovo{As follows from the proof of Theorem \ref{teo2} below, when $\bbG=\bbR^d$ one can  replace $\t_k\D$ by $k+[0,1)^d$ as well. In general, for $\bbG=\bbR^d$, one can replace the cells $\{\t_k \D\}_{k \in \bbZ^d}$ by the cells of any  lattice partition of  $\bbR^d$.}
\end{Remark}
\begin{Theorem}\label{teo2}  Let Assumptions (A1),...,(A8)  be satisfied. 
Take  $\o \in \O_{\rm typ}$ and $f \in C_c (\bbR^d)$. Then for any  $t\geq 0$ and $\l>0$, it holds
\begin{align}
&  L^2(\mu^\e _\o ) \ni P^\e_{\o,t} f \to P_t f \in L^2(mdx)\,, \label{marvel0}
 \\
&   L^2(\mu^\e _\o ) \ni R^\e_{\o,\l} f \to R_\l  f \in L^2(mdx)\,. \label{vinello}
 \end{align}
   Suppose in addition that  Assumption (A9) holds. Then there exists a \iva{translation invariant} measurable set $\O_\sharp\subset \O $ with  $\cP( \O_\sharp)=1$  such that for any $ \o \in \O_\sharp \cap \O_{\rm typ}$, any $f \in C_c(\bbR^d)$, $\l>0$, $t\geq 0$ it holds:
\begin{align}
& \lim_{\e \da 0} \int \bigl | P^\e_{\o,t} f(x) - P_t f (x) \bigr|^2 d \mu^\e_\o (x)=0\,,\label{marvel1}\\
& \lim_{\e \da 0} \int \bigl | P^\e_{\o,t} f(x) - P_t f (x) \bigr| d \mu^\e_\o (x)=0\,,\label{marvel2}\\
& \lim_{\e \da 0} \int \bigl | R^\e_{\o,\l} f(x) - R_\l f (x) \bigr|^2 d \mu^\e_\o (x)=0\,,\label{ondinoA}\\
&   \lim_{\e \da 0} \int \bigl | R^\e_{\o,\l} f(x) -R_\l f (x) \bigr| d \mu^\e_\o (x)=0\,.\label{ondinoB}
\end{align}
\end{Theorem}
The proof of Theorem \ref{teo2} is given in Section \ref{dim_teo2}. 
\begin{Remark} Assumption (A9) is used only to derive Lemma \ref{pre_teo2} \iva{in Section \ref{dim_teo2}},  which is applied in the proof of Theorem \ref{teo2} only with $\psi(r):=1/(1+ r^{d+1})$.
\end{Remark}

\section{Some examples}\label{sec_esempi} 
The class of reversible random walks in random environment is very large. We discuss just some popular examples. \utto{Below, when we state that assumptions (A3),...,(A6) are satisfied, we understand that this holds with $\o$ in a suitable translation invariant measurable set $\O_*\subset \O$ with $\cP(\O_*)=1$.} 
\subsection{Nearest-neighbor random conductance model on $\bbZ^d$}\label{nonno}
We take $\bbG:=\bbZ^d$ and $V:=\bbI$. Let $\bbE^d:=\left \{ \{x,y\} \,:\, x,y \in \bbZ^d, \; |x-y|=1\right\}$. We  take  $\O:=(0,+\infty)^{\bbE^d}$ endowed with the product topology. We write $\o=( \o_b\,:\, b \in \bbE^d)$ for a generic element of $\O$ and we write $\o_{x,y}$ instead of $\o_{\{x,y\}}$. \ovo{Note that $\o_{x,y}=\o_{y,x}$}. The action $(\theta _x)_{x\in \bbZ^d}$  is  the standard one: $\theta_x$ shifts the environment along the vector $-x$. 
 We set $\mu_\o:= \sum_{x\in \bbZ^d}\d_x $ and  $r_{x,y}(\o):=\o_{x,y}$ if  $\{x,y\}\in \bbE^d$ and  $r_{x,y}(\o):=0$ otherwise.   Then, given the environment $\o$,  the random walk $X_t^\o$ has state space $\bbZ^d$ and  jumps from $x$ to $y$, with $|x-y|=1$,   with  rate $\o_{x,y}$.
Assumption (A1),...,(A9) are satisfied whenever $\cP$ is stationary and ergodic, it satisfies (A3)  and $\bbE[\o_{x,y}]<+\infty$ for all $|x-y|=1$ \ovo{(i.e., by stationarity, $\bbE[\o_{0,e_i}]<+\infty$ for $i=1,\dots,d$)}. \ovo{Due to \cite[Prop.~4.1]{Bi}, if in addition $\bbE[1/\o_{0,e_i}]<+\infty$ for $i=1,\dots,d$, then the matrix $D$ is non-degenerate}.

If one wants the version with waiting times of parameter $1$, then one has to set $n_x(\o): = \sum_{y:|x-y|=1} \o_{x,y}$, $\mu_\o:= \sum_{x\in \bbZ^d}n_x(\o) \d_x $ and $r_{x,y}(\o):=\o_{x,y}/ n_x(\o)$.   Then 
Assumptions (A1),...,(A8) are satisfied whenever $\cP$ is stationary and ergodic, it satisfies (A3)  and $\bbE[\o_{x,y}]<+\infty$ for all $|x-y|=1$ (use \eqref{zazzera}). \ovo{As in the previous case, $D$ is non-degenerate if $\bbE[1/\o_{0,e_i}]<+\infty$ for $i=1,\dots,d$ (see Remark \ref{virale}).}
  To satisfy Assumption (A9) it is enough e.g. that the conductances $\o_{x,y}$ are uniformly bounded or that  the covariance between $\o_{x,y}$ and $\o_{x',y'}$ decays at least as the inverse of the distance between $\{x,y\}$ and $\{x',y'\}$.
\subsection{Random conductance model on $\bbZ^d$ with long conductances}
We take $\bbG:=\bbZ^d$ and $V:=\bbI$.  We set  $\bbB^d:=\left \{ \{x,y\} \,:\, x,y \in \bbZ^d, \; x\not = y \right\}$ and take  $\O:=(0,+\infty)^{\bbB^d}$ endowed with the product topology.  We set $\o_{x,y}:=\o_{\{x,y\}}$. The action $(\theta _x)_{x\in \bbZ^d}$  is  the standard one.  We take $\mu_\o:= \sum_{x\in \bbZ^d}\d_x $,  $r_{x,y}(\o):=\o_{ x,y}$ if  $\{x,y\}\in \bbB^d$ and  $r_{x,y}(\o):=0$ otherwise.   Then, given the environment $\o$,  the random walk $X_t^\o $ has state space $\bbZ^d$ and  jumps from $x$ to $y$    with probability rate $\o_{x,y}$.  
Assumptions (A1),...,(A9) are satisfied whenever $\cP$ is stationary and ergodic, it satisfies (A3)
and it satisfies   $\bbE[\sum_{z\in \bbZ^d} \o_{0,z}|z|^2]<+\infty$ (which implies $\bbE[\sum_{z\in \bbZ^d} \o_{0,z}]<+\infty$). \ovo{Reasoning as in the proof of \cite[Prop.~4.1]{Bi}, for all $a\in \bbR^d$ one can lower bound the scalar product  $a\cdot Da$  by $C \sum_{x\in \bbZ^d} (a\cdot x)^2 / \bbE[ 1/\o_{0,x}]$ with $C>0$. Hence, $D$ is non-degenerate if the set $\{x\in \bbZ^d\,:\, \bbE[ 1/\o_{0,x}]<+\infty\}$ is not contained in a subspace of $\bbR^d$ with dimension smaller than $d$}.
\subsection{Random walk with random conductances on infinite clusters}

 We take $\bbG:=\bbZ^d$ and $V:=\bbI$. Let $\bbE^d$ be as in Example \ref{nonno}.
  We  take  $\O:=[0,+\infty)^{\bbE^d}$ with the product topology. The action $(\theta _x)_{x\in \bbZ^d}$  is  the standard one.   Let $\cP$ be a probability measure on $\O$ stationary,  ergodic and fulfilling (A3) for the above action. We assume that for $\cP$--a.a. $\o$ there exists a unique infinite connected component $\cC(\o)\subset \bbZ^d$ in the graph given by the edges $\{x,y\}$ in $\bbE^d$  with positive conductivity $\o _{x,y}$.

  We set $\mu_\o:= \sum _{x\in \cC(\o)} \d_x$,  $r_{x,y}(\o):= \o_{x,y}$ if $\{x,y\} \in \bbE^d$ and $r_{x,y}(\o):= 0$ otherwise.
Note that $n_x(\o) = \mathds{1}(x\in \cC(\o))$. Then the random walk $X_t^\o$ has state space $\cC(\o)$ and jumps from $x$ to $y$ in $\cC(\o)$ (where $|x-y|=1$) with probability rate $\o_{x,y}$. Note that $d \cP_0 (\o)=  \mathds{1}(0\in \cC(\o)) d\cP(\o) / \cP( 0\in \cC(\o))$.  If, in addition, $\cP$ satisfies 
$\bbE[ \o_{x,y} ]<+\infty$ for all  $\{x,y\}\in \bbE^d$, then all Assumptions (A1),...,(A9) are satisfied (note that  we  need neither bounded conductances nor the non-degeneracy of the diffusion matrix, differently from    \cite{F1}). This holds for example  if $\cP$ is the Bernoulli product probability  on $\O$ such that $\cP( \o_{x,y}>0)>p_c$,  $p_c$ being  the bond percolation critical probability  on $\bbZ^d$, and  $\bbE[ \o_{x,y}]<+\infty$ for all $\{x,y\}\in \bbE^d$.

If  interested to the modified version with waiting times of parameter 1, then we set $n_x(\o):= \sum _{y:\{x,y\}\in \bbE^d} \o_{x,y} \mathds{1}(x\in \cC(\o))$,
$\mu_\o:= \sum _{x\in \cC(\o)} n_x(\o)  \d_x$ and $r_{x,y}(\o):= \o_{x,y}/n_x(\o)$ if $\{x,y\} \in \bbE^d$ and $x,y\in \cC(\o)$,  otherwise we set $r_{x,y}(\o):= 0$. All assumptions (A1),...,(A8) are satisfied whenever $\bbE[ \o_{x,y} ]<+\infty$ for $\{x,y\}\in \bbE^d$. For (A9) one can argue as in Example \ref{nonno}.

\ovo{We refer e.g. to \cite{BB,DNS,MP} for additional assumptions assuring the non-degeneracy of $D$.}


\subsection{Mott random walk}\label{ex_mott} Mott random walk  (see e.g. \cite{CF2,CFP, FM,FSS}) is a mean-field model for  Mott variable range hopping in amorphous solids \cite{POF}.
We take $\bbG:= \bbR^d$ and $V:=\bbI$. $\O$ is given by the space of marked simple counting measures with marks in $\bbR$. By identifying $\o$ with its support, we have $\o=\{(x_i,E_i)\}$ where $E_i\in \bbR$ and the set $\{x_i\}$ is locally finite.  $\O$ is a metric space, being  a subset 
of the metric space $\cN$ of counting measures  $\mu=\sum _{i} k_i \d_{(x_i,E_i)}$,  where $k_i\in \bbN$  and $\{(x_i,E_i)\} $ is a locally finite subset of $\bbR^d\times \bbR$.  See \ovo{\cite[Eq.~(A2.6.1) in App.~A2.6]{DV}} for the metric  $d$ associated to $\cN$.  One can prove that $(\cN,d)$ is a Polish space having $\O$ as Borel subset \cite[Cor.~7.1.IV, App.~\ovo{A2.6}]{DV}.
The action $\theta_x$ on $\O$ is given by $\theta _x \o:=\{ (x_i-x, E_i)\}$ if $\o=\{(x_i, E_i)\}$.

 Given the environment $\o=\{(x_i,E_i)\}$, to get Mott random walk   we take $\mu_\o := \sum _i \d_{x_i}$ (hence $\hat \o := \{x_i\}$ and  $n_{x_i}(\o) :=1$) and 
\be\label{vento}
r_{x_i,x_j}(\o):= \exp\bigl\{ -|x_i-x_j| - |E_{x_i}|-|E_{x_j}|-|E_{x_i}-E_{x_j}|\bigr\}\,\qquad x_i\not =x_j \,.
\en
Hence, $X_t^\o$ walks on $\{x_i\}$ with jump probability rates given by \eqref{vento}.
\utto{Note that the properties in (A4), (A5), (A6) are automatically satisfied by  all  $\o\in \O$ (i.e. $\O_4=\O_5=\O_6=\O$ with the notation of Section \ref{tassometro}).} 
Suppose that $\cP$ satisfies (A1), (A2) and \utto{$\cP( \theta_g \o \not = \theta_{g'} \o \; \forall g\not =g' \text{ in } \bbG)=1$ (as discussed in Section \ref{tassometro} the set $\{ \o\,:\, \theta_g \o \not = \theta_{g'} \o \; \forall g\not =g' \text{ in } \bbG\}$ is measurable). Then (A3) is satisfied by taking $\O_*=\O_3$ and requiring $\cP(\O_*)=1$.}  $\cP_0$ is simply the  standard  Palm distribution associated to the marked simple point process with law $\cP$ \cite{DV2}.  As the above space $\cN$ is Polish and $\O_0:=\{\o\,:\, 0 \in \hat \o\} $ is a Borel subset of $\cN$, $\O_0$   is separable and therefore (A8) is satisfied \ovo{(see the comment on (A8)  in Section \ref{tassometro})}.

We claim that  the bound  $\bbE \bigl[ |\hat \o \cap [0,1]^d | ^{2} \bigr]<\infty$  implies (A7). To prove   our claim we observe that, by \cite[Lemma 2]{FSS}, given a positive integer $k$ it holds $\l_0 \in L^k(\cP_0)$ if and only if  $\bbE \bigl[ |\hat \o \cap  [0,1]^d | ^{k+1} \bigr]<\infty$. The proof provided there remains true when substituting $\l_0$ by any function $f $  such that  $|f(\o) | \leq C \int d\hat \o (x) e^{-c|x|} $ with $C,c>0$. Hence we can take     $f=\l_2$. Then the above bound on the second moment of  $\hat \o \cap [0,1]^d$ implies that $\l_0,\l_2 \in L^1(\cP_0)$.

For (A9) we observe that $\mu_\o ( k + [0,1)^d)$ equals the number of points $x_i$ in $k + [0,1)^d$. Hence, there are plenty of  simple point processes satisfying (A9).

\ovo{We refer  to \cite{CF2,CFP,FSS} for additional assumptions assuring the non-degeneracy of $D$.}

%
%
%
%
 
 \subsection{Simple random walk on Delaunay triangulation}
We take $\bbG:= \bbR^d$ and $V:=\bbI$. $\O$ is given by the space of simple counting measures on $\bbR^d$. We set $\mu_\o:=\o$. We take $r_{x,y}(\o)=\mathds{1}( x\stackrel{\o}{\sim } y)$, where $x\stackrel{\o}{\sim }y$ means that $x,y$ are adjacent in the  $\o$--Delaunay triangulation. 
Then, given $\o$, the random walk $X^\o_t$ is the simple random walk on the $\o$--Delaunay triangulation. 
By taking $\cP$ stationary with finite intensity,  $\cP_0$ becomes  the standard Palm distribution associated to the stationary simple point process on $\bbR^d$ with law $\cP$ \cite{DV2}. If for example $\cP$ is a  \ovo{homogeneous} Poisson point process, using the \ovo{results} in \cite{Rou} it is simple to conclude  that all  Assumptions (A1),...,(A9) are satisfied (for (A8)  reason as for Mott random walk) \ovo{and that $D$ is non-degenerate}.

 More general cases, also with random conductances, will be discussed in \cite{FT}.

%

\subsection{Nearest--neighbor random conductance models on lattices} To have a concrete example  let us consider the nearest-neighbor  random conductance model on the hexagonal lattice $\cL=(\cV,\cE)$, partially drawn in Fig.~\ref{apetta} (hexagons have  edges \ovo{of length one}). 
\begin{figure}[!ht]
    \begin{center}
     \centering
  \mbox{\hbox{
  \includegraphics[width=0.5\textwidth]{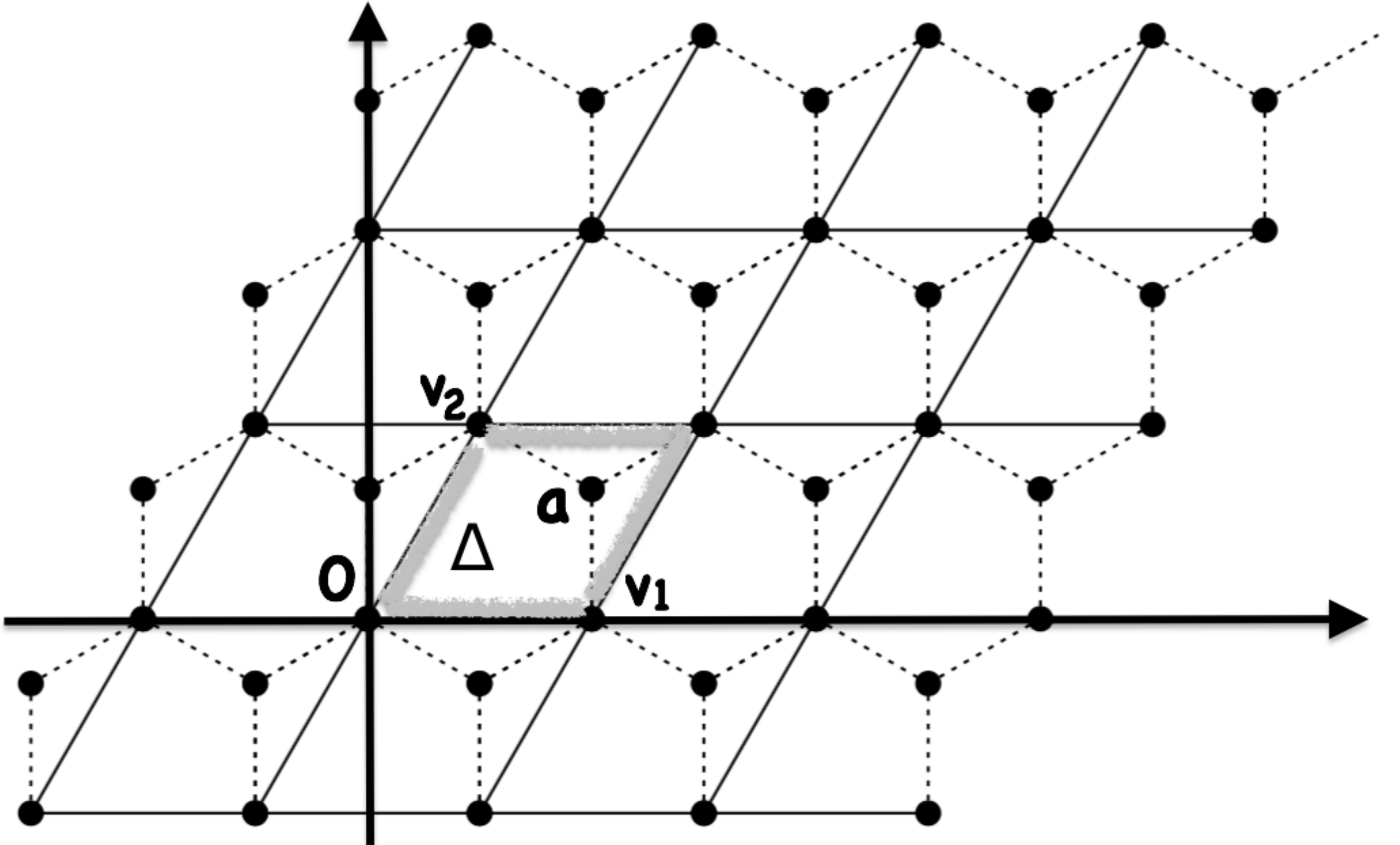}}}
         \end{center}
         \caption{Hexagonal lattice, fundamental cell $\D$, basis $\{v_1,v_2\}$}\label{apetta}
  \end{figure}

Let $v_1,v_2$ be the vectors  $v_1=\bigl( 2 \cos\frac{\pi}{6},0\bigr)$, $v_2=\bigl(2 \cos\frac{\pi}{6} \cos\frac{\pi}{3} ,2 \cos\frac{\pi}{6} \sin\frac{\pi}{3} \bigr)$. 
We take  $\O:= (0,+\infty)^{\cE}$ endowed with the product topology and set $\o_{x,y}:=\o_{\{x,y\}}$.
  Let $\bbG:=\bbZ^d$ and  let $V$ be  the matrix with columns $v_1,v_2$ respectively. 
 The action of $(\theta_z)_{z\in \bbZ^d}$ and $(\t_z)_{z\in \bbZ^d}$ of  $\bbZ^d$ on $\O$ and $\bbR^d$, respectively, are given by
\begin{align*}
&\theta _z \o: =(\o_{x-Vz,y-Vz }\,:\, \{x,y\}\in \cE)  \text{ if } \o=(\o_{x,y}\,:\,\{x,y\}\in \cE)\,, \; z\in \bbZ^d\,,\\
& \t_z x:=x+ Vz  \text{ for } x \in \bbR^d\,, \; z\in \bbZ^d\,.
\end{align*}
Moreover, for any $\o \in \O$,  we set $\mu_\o:= \sum _{x\in \cV} \d_x$.
The fundamental cell $\D$ is given by $\D=\{t_1v_1+ t_2 v_2\,:\, 0\leq t_1,t_2<1\}$. Note that the set  $\O_0$ introduced after  \eqref{Palm_Z} equals $\O\times \{0,a\}$ and that, by \eqref{puffo2}, $m \ell (\D)= 2 $. Hence (see \eqref{Palm_Z})
$\cP_0(d\o,dx)= \cP(d\o)  \otimes  {\rm Av}_{u\in\{0,a\} } \d_u(dx) $.
Setting $r_{x,y}:= \o_{x,y}$ if $\{x,y\}\in \cE$ and $r_{x,y}:=0$ otherwise, the random walk $X^\o_t$  has state space $\cV$ and jumps from $x$ to a nearest-neighbor site $y$ with 
probability rate $\o_{x,y}$. If, for example, $\cP$ is given by a Bernoulli product probability measure with  $\bbE[\o_{x,y} ]<+\infty$  for all  $\{x,y\}\in \cE$, then all assumptions (A1),\dots,(A9) are satisfied.

\section{From $\bbZ^d$-actions to $\bbR^d$--actions}\label{sec_fluidifico}
Suppose $\bbG=\bbZ^d$ and call $\cS[1]$ the setting described by $\bbG$, $(\O,\cF,\cP)$, $(\theta_g)_{g\in \bbG}$, $(\t_g)_{g\in \bbG}$, $\mu_\o$, $r_{x,y}(\o)$, \utto{$\O_*$}. We now introduce a new setting $\cS[2]$ described by new objects $\bar\bbG=\bbR^d$, $(\bar\O,\bar\cF,\bar\cP)$, $(\bar{\theta}_g)_{g\in \bar \bbG}$, $(\bar{\t}_g)_{g\in \bar{\bbG}}$, $\mu_{\bar \o}$, $\bar r_{x,y}(\bar \o)$, $\utto{\bar{\O}_*}$ such that if $\cS[1]$  satisfies the main assumptions (A1),...,(A8), then the same holds for $\cS[2]$, and  if the conclusion of Theorem \ref{teo1}  holds for $\cS[2]$, then the same holds for $\cS[1]$. 
 As a consequence, 
to prove Theorem \ref{teo1} it is enough to consider the case $\bbG=\bbR^d$.   Many identities pointed out below will be proved in Appendix \ref{app_fluidifico}.


We consider the extended  probability space $(\bar \O, \bar \cF, \bar\cP)$ defined as
\[ \bar\O:=\O \times \D\,, \qquad \bar \cP := \ell(\D)^{-1} \cP \otimes \ell\,,\qquad \bar \cF:= \cF \otimes \cB(\D)\,.
\]
We define   the action $(\bar \theta _x)_{x\in\bar \bbG} $ of $\bar \bbG=\bbR^d$ on $\bar \O$ as 
\be \label{cetriolo}
\bar \theta_x \bigl(\o, a\bigr) =\bigl(\theta_{g(x+a)} \o, \beta(x+a)\bigr)\,.
\en
One can easily check that $(\bar \theta _x)_{x\in\bbR^d} $  satisfies the properties analogous to  (P1), ...,(P4) in Section \ref{flauto}, when replacing $\bbG$ by $\bar\bbG$ (for the validity of (P4) concerning the $\bar \bbG$-stationary of $\bar \cP$  see Lemma \ref{LB1} in Appendix  \ref{app_fluidifico}).  Moreover,  $\bar \cP$ is ergodic w.r.t. the action of $(\bar \theta_x)_{x\in \bar \bbG}$ (cf. Lemma \ref{LB5}). Hence, (A1) is fulfilled by $\cS[2]$.

 We define 
\be\label{limone} 
\bar \O \ni \bar \o  \mapsto  \mu _{\bar \o} \in \cM \,, \qquad   \mu_{(\o,a)}(\cdot):= \mu_\o (\cdot +a ) \,.
\en
The intensity  $\bar m $ and $m$   of the random measure $\mu_{\bar \o}$ and 
 $\mu_\o$, respectively, coincide  (cf. Lemma \ref{LB3}). Hence, (A2) is fulfilled by $\cS[2]$.
 
\utto{We set $\bar{\O}_*:= \O_* \times \D$.  It is simple to check that  $\bar{\O}_*$ is a translation invariant measurable set with $\bar \cP(\bar{\O}_*)=1$ and that    (A3) is fulfilled by $\cS[2]$ for any $\bar \o \in \bar{\O}_*$}.
The action $(\bar \t_x)_{x\in \bar \bbG}$ on $\bbR^d$ is given by
\be\label{pinzimonio}
\bar \t_x z:= z +x\,.
\en 
By writing  $n_x(\bar \o)= \mu_{\bar \o}(\{x\})$ for $x\in \bbR^d$, the above definition \ovo{\eqref{limone}} implies that 
\be\label{limoncello}
 \bar\o=(\o, a) \; \Longrightarrow \; 
\hat{\bar{\o}}=\hat \o -a \,, \qquad n_x(\bar \o)=n_{x+a}(\o)\,.
\en
Then we \ovo{have} (cf. Lemma \ref{LB2})
 \be \label{arancia}
 \mu_{\bar \theta _x \bar \o}(\cdot) = \mu_{\bar \o}(\bar \t_x \cdot)
  \qquad \forall \bar \o \in \ovo{\bar{\O}_*}, \;x\in \bar\bbG\,.
  \en
We define the \rrr{measurable} function 
\be 
\bar r: \bar \O \times \bbR^d \times \bbR^d \ni (\bar \o, x, y)\mapsto \bar r_{x,y}(\bar \o) \in [0,+\infty)
\en
as 
\be\label{normale}
\bar r_{x,y}( \o, a):= r_{x+a, y+a}(\o)\,.
\en
 The analogous of \eqref{montagna}  still holds  \utto{ for $\bar \o \in \bar{\O}_*$} (cf. Lemma \ref{lemma_mont}). Note that, by \eqref{limoncello} and \eqref{normale}, we have 
$\bar c_{x,y}(\o, a)= c_{x+a,y+a}(\o)$. Hence, (A4) and (A5) are fulfilled by $\cS[2]$ \utto{for all $\bar \o \in \bar{\O}_*$}

The Palm distributions $\cP_0, \bar\cP_0$ associated  respectively to $\cP,\bar \cP$ coincide   (cf. Lemma \ref{LB4}). Hence, (A8) is trivially satisfied by $\cS[2]$.  
Note moreover that ${\bar \O}_0=\{ (\o,x) \in \bar \O\,:\, n_0(\o,x)>0\}= \{ (\o,x)\in \O\times \D\,:\, n_x(\o)>0\}\ovo{=\O_0}$. 

 Recall \eqref{altino15}. \ovo{We write $ \l_k $ and $\bar \l_k$ for the function corresponding to \eqref{altino15}  in setting $\cS[1]$ and $\cS[2]$, respectively. Note that $\l_k, \bar \l_k$  are defined on the same set ${\bar \O}_0$.
Given $\bar\o= (\o,a) \in {\bar \O}_0=\O_0$, we have  (using \eqref{limoncello} and \eqref{normale})
\be
\begin{split}
\bar\l_k (\bar \o) = \int _{\bbR^d} d\hat{\bar{\o}} (x) \bar{r}_{0,x} (\bar \o) |x|^k & =  \int _{\bbR^d} d\hat{\o} (y) \bar{r}_{0,y-a} (\bar \o) |y-a|^k\\& =\int _{\bbR^d} d\hat{\o} (y) r_{a,y} ( \o) |y-a|^k  =\l_k (\bar \o)\,.
\end{split}
\en}
 In particular  (A7) implies  that   $\bar \l_0,  \bar \l_2 \in L^1(\bar \cP_0)$. In conclusion we have:
If  (A1),...,(A8) are satisfied by $\cS[1]$, then   (A1),...,(A8) are satisfied  by $\cS[2]$. \ovo{Finally}, 
as the integral in the r.h.s. of \eqref{def_D_Z} equals
\[ \frac{1}{2} \int _{\bar\O} d\bar \cP_0(\bar \o) \int _{\bbR^d} d \hat{\bar \o} (z) \bar r_{0,z} (\bar \o) \left( a\cdot z - \bigl[ f(\bar \theta _z \bar \o)- f(\bar \o) \bigr] \right)^2 \,,
\]
one can easily check that  Theorem \ref{teo1} for $\cS[2]$ implies 
Theorem \ref{teo1} for $\cS[1]$.

\begin{Warning}\label{marinaio}
  Due to the above discussion it is enough to prove Theorem  \ref{teo1} only for $\bbG=\bbR^d$.  Due to its relevance in discrete probability,  we will also treat simultaneously the special discrete case. Moreover, to slightly simplify the notation, we will take $V=\bbI$, thus implying that $g(x)=x$. The reader interested to formulas with generic  matrix $V$  has only to replace $\theta_x$ by $\theta_{g(x)}$ in what follows. Indeed, the manipulations behind our formulas rely on Remark \ref{jolie}, which holds in the general case. In conclusion,   from now on (with exception of \ovo{Section \ref{dim_teo2} and} the appendixes)  and without further mention,  we restrict to the case $\bbG=\bbR^d$ and to the special discrete case,  we take  $V=\bbI$ and  we understand that Assumptions (A1),...,(A8) are satisfied.
  \end{Warning}

\section{Some key properties of the Palm distribution $\cP_0$}\label{sec_minecraft}
 For the results of this section  it would be  enough to require (A1) and (A2) and, for Lemma \ref{lemma_siou}, $\bbE_0[\l_0]<\infty$.
Given   \green{a measurable set  $A\subset \O$}, we define
\be\label{tildino}
\tilde A:=\{ \o \in \O\,:\, \theta _{x} \o \in A\; \forall x \in \hat \o \}\,.
\en
Note that $\tilde A$ is  translation invariant \ovo{ and measurable}. 

\begin{Lemma}\label{matteo}
Given \green{$A\subset \O$ measurable}, the following facts are equivalent:
(i) $\cP_0(A)=1$; (ii) $\cP(\tilde A)=1$; (iii) $\cP_0(\tilde A)=1$.
Given  a translation invariant \green{measurable set $A\subset \O$}, it  holds $\cP(A)=1 $ if and only if $ \cP_0(A)=1$ and it holds $\cP(A)=0 $ if and only if $ \cP_0(A)=0$.
\end{Lemma}
\begin{proof} 
We first prove the equivalence between (i), (ii) and (iii).
By \eqref{palm_classica} \ovo{for $\bbG=\bbR^d$ and \eqref{zazzera} for the special discrete case},   (ii) implies (i). If   (i) holds, then we get (ii)  by  \eqref{puffo1} and \eqref{zazzera}  and Campbell's identities   \eqref{campanello}  and \eqref{campanelloZ_sp} with
 $ f(x,\o):=(2\ell)^{-d} \mathds{1}_{[-\ell,\ell]^d}  (x) \mathds{1}_A (\o )$ and $\ell\in\bbN_+$. 
 Note that $\tilde{\tilde A}= \tilde A$. Hence, by applying the equivalence between (i) and (ii) with $A$ replaced by $\tilde A$, we get the equivalence between (ii) and (iii).

 Now let $A$ be  translation invariant. To  prove that   $\cP(A)=1 $ if and only if $ \cP_0(A)=1$, 
 it is enough to observe that $\tilde A=A$  and to apply the above  equivalence between (ii) and (iii). To prove that $\cP(A)=0 $ if and only if $ \cP_0(A)=0$, it is enough to take the complement.
\end{proof}
 As an immediate consequence of Lemma \ref{matteo} we get:
 \begin{Corollary}\label{eleonora}
 Let $f\in L^1(\cP_0)$. Let $B:= \{ \o\in \O\,:\, |f(\theta_x \o  )|<+\infty \; \forall x \in \hat \o \}$. Then $B$ is translation invariant, $\cP(B)=1$ and $\cP_0(B)=1$.
  \end{Corollary}

The  following result generalizes \cite[Lemma 1--(i)]{FSS}:
\begin{Lemma}\label{lemma_siou}  Let $f:\O_0\times \O_0 \to \bbR$ be a \rrr{measurable} function. Suppose that   (i)  at least one of the functions   $\o \mapsto  \int d \hat{\o} (x)  r_{0,x}(\o)  |f(\o, \theta_x \o) | $ and $\o \mapsto \int d\hat{\o}(x)  r_{0,x}(\o)   |f(\theta_x\o,\o)|$ is  in $L^1(\cP_0)$, or (ii) $f\geq 0$. Then it holds
\begin{equation}\label{siou}
\bbE_0\Big[ \int_{\bbR^d} d \hat \o(x) r_{0,x}(\o) f(\o, \theta_x \o)\Big]=\bbE_0\Big[ 
 \int_{\bbR^d} d \hat \o(x) r_{0,x}(\o) f(\theta_x \o,\o)\Big]\,.
\end{equation}
\end{Lemma}
We note that, by \eqref{zazzera}, in the special discrete case  \eqref{siou} reads
 \be \label{siouZ} \sum_{x\in \bbZ^d} \bbE\Big[  c_{0,x}(\o) f(\o, \theta_x \o)\Big]=\sum_{x\in \bbZ^d} 
 \bbE \Big[ 
 c_{0,x}(\o) f(\theta_x \o,\o)\Big]\,.
 \en
\begin{proof}[Proof of Lemma \ref{lemma_siou}] 
We first sketch  the proof in the 
 special discrete case, which is trivial.  Given $\o \in \O_0$,  due to Remark \ref{jolie} and  (A5), $c_{0,x}(\o)=c_{0,-x} (\theta_x \o)$ \ovo{$\cP$--a.s.}
 By the translation invariance of $\cP$ we get $\bbE\big[  c_{0,x}(\o)|f|(\o, \theta_x \o)\big]=\bbE \bigl[ 
 c_{0,{-x}}(\o) |f|(\theta_{-x }\o ,\o)\bigr]$, and  the same then must hold  in Case (i) with $f$ instead of $|f|$. This allows 
  to conclude \ovo{the proof of} \eqref{siouZ}.
 
 We move to the setting $\bbG=\bbR^d$. We start with Case (i)
 supposing first that both functions there 
are  in $L^1(\cP_0)$.
We set $B(n):=[-n,n]^d$.
 Due to \eqref{palm_classica} and using also (A5) for \eqref{cantino} we get 
 \begin{align}
&  \text{l.h.s. of \eqref{siou}}  =\frac{1}{m  (2n)^d  }\bbE\Big[ \sum _{x\in \hat \o \cap B(n) }  \sum_{z \in \hat \o}
c_{x,z}( \o)f\bigl( \theta_x \o, \theta _z \o \bigr) \Big] \,,\label{canto}\\
&  \text{r.h.s. of \eqref{siou}}  =\frac{1}{m  (2n)^d }\bbE\Big[  \sum _{x\in \hat \o  }  \sum_{z \in \hat \o\cap B(n) }
c_{x,z}( \o)f\bigl( \theta_x \o, \theta _z\o \bigr)\Big]\,.\label{cantino}
 \end{align}
 To prove that \eqref{canto} equals \eqref{cantino} it is enough to show that 
 \be\label{titino}
 \lim _{n\to \infty}\frac{1}{ m(2n)^d  }\bbE\Big[  \sum _{x\in \hat \o \cap  B(n)  }  \sum_{z \in \hat \o \setminus B(n) }
c_{x,z}( \o)|f|\bigl( \theta_x \o, \theta _z \o \bigr) \Big]=0
 \en
 and that the same limit  holds when summing among $x \in \hat \o \setminus B(n)$ and $z\in \hat \o \cap B(n)$. We prove \eqref{titino}, the other limit can be treated similarly.
 By \eqref{palm_classica}
  \be
  \begin{split}
&  \frac{1}{ m (2 n)^d  }\bbE\Big[  \sum _{x\in \hat \o \cap  B(n)  }  \sum_{z \in \hat \o \setminus B(n+\sqrt{n}) }
c_{x,z}( \o)|f|\bigl( \theta_x \o, \theta _z \o \bigr)\Big]\leq  \\
&   \frac{1}{ m (2 n)^d  }\bbE\Big[  \sum _{x\in \hat \o \cap  B(n)  }  \sum_{z \in \hat \o: |z-x|_\infty\geq \sqrt{n} }
c_{x,z}( \o)|f|\bigl( \theta_x \o, \theta _z\o \bigr)\Big]=\\
&  \bbE_0\Big[ \int_{\bbR^d} d \hat \o(x) r_{0,x}(\o) |f|(\o, \theta_x \o)\mathds{1}(|x|_\infty \geq \sqrt{n})\Big]\,.
\end{split}
 \en
Due to our $L^1$-hypothesis and  dominated convergence, the last member goes to zero as $n\to \infty$. Hence, it remains to prove  \eqref{titino} with ``$z \in \hat \o \setminus B(n)$'' replaced by ``$z \in \hat \o \cap U(n)$'' where  $U(n):=B(n+\sqrt{n})\setminus B(n)$. 
 To this aim, by \eqref{palm_classica} and  (A5),  we get 
 \be\label{vespa}\begin{split}
&  \bbE_0\Big[ \int_{\bbR^d} d \hat \o(x) r_{0,x}(\o)| f|( \theta_{x}  \o, \o) \Big] =\\
 &
 \frac{1}{m  \ell(U(n))  }\bbE\Big[ \sum_{x \in \hat \o}\sum _{z\in \hat \o \cap U(n)}  
c_{x,z}( \o)|f|\bigl( \theta_{x} \o, \theta _{z} \o \bigr)\Big]\,. 
\end{split}
 \en
 By hypothesis, the  above l.h.s. is finite.  Hence  \eqref{titino} with ``$z \in \hat \o \cap U(n)$''
  follows 
by using \eqref{vespa}  and that $ \ell (U(n) ) /n^d \to 0$. This concludes the proof of \eqref{siou} when both functions in Case (i)  are  in $L^1(\cP_0)$. 
We interrupt with Case (i) and move to Case (ii) with $f\geq 0$. By the above result and since $\bbE_0[\l_0]<+\infty$, identity \eqref{siou} holds with $f$ replaced   by $f\wedge n$.
By taking the limit $n\to \infty$ and using monotone convergence, we get \eqref{siou} for a generic $f\geq 0$. Let us come back to Case (i). Since (by Case (ii)) \eqref{siou} holds with $f$ replaced by $|f|$, we get that in Case (i)  both  functions there  belong to $L^1(\cP_0)$ as soon as at least one does.
The conclusion then follows from our first result.
\end{proof}


%
%
%
\section{Space of square integrable forms}\label{sec_quadro}
We define $\nu$ as the Radon measure on $\O_0 \times \bbG$ such that 
\begin{equation}\label{labirinto}
 \int d \nu (\o, z) g (\o, z) = \int  d \cP_0 (\o)\int d  \hat \o (z) r_{0,z}(\o) g( \o, z) 
 \end{equation} 
 for any nonnegative \rrr{measurable} function $g(\o,z)$.  
 \begin{Remark}
 When considering   the special discrete case  in the r.h.s. of \eqref{labirinto} one can replace the integral $\int d  \hat \o (z)$ with the series $\sum_{z\in \bbZ^d}$ (recall that  $r_{x,x}(\o)=0$ and $r_{x,y}(\o)=0$ if $\{x,y\}\not \subset \hat \o$). A similar rewriting holds in the formulas presented in the rest of the paper. 
 \end{Remark}

  We point out that, by Assumption (A7), $\nu$ has finite total mass:
$ \nu(\O\times \bbR^d )=\bbE_0[\l_0]<+\infty$.
  Elements of  $L^2 ( \nu)$
 are called \emph{square integrable forms}.

 \smallskip
 
 Given a  function $u:\O_0\to \bbR$  we define the   function $\nabla u: \O_0 \times \bbG \to \bbR$ as 
 \begin{equation}\label{cantone}
 \nabla u (\o, z):= u (\theta_z \o)-u (\o)\,.
 \end{equation}
 Note that,  by Lemma \ref{matteo} with $A:=\{\o\in \O_0\,:\, u(\o)=f(\o)\}$, if $u,f :\O_0\to \bbR$  are such that  $u=f$  $\cP_0$--a.s., then $\nabla u= \nabla f $ $\nu$--a.s..
  In particular, 
 if $u$ is defined  $\cP_0$--a.s., then $\nabla u$ is well defined $\nu$--a.s.
 
  If $u $ is bounded and \rrr{measurable}, then $\nabla u \in L^2(\nu)$.
The subspace of \emph{potential forms} $L^2_{\rm pot} (\nu)$ is defined as   the following closure in $L^2(\nu)$:
 \[ L^2_{\rm pot} (\nu) :=\overline{ \{ \nabla u\,:\, u \text{ is  bounded and measurable} \}}\,.
 \]
 The subspace of \emph{solenoidal forms} $L^2_{\rm sol} (\nu)$ is defined as the orthogonal complement of $L^2_{\rm pot} (\nu)$ in $L^2(\nu)$.
  
    \subsection{The subspace $H^1_{\rm env} $}\label{h1omega}
 We define 
 \be \label{anita} H^1 _{\rm env} :=\{   u \in L^2(\cP_0)\,:\,  \nabla u \in L^2(\nu) \}\,.
 \en We endow $ H^1 _{\rm env} $ with  the norm
 $\| u\|_{H^1 _{\rm env}} :=  \|u\|_{L^2(\cP_0)} +\|\nabla u \|_{L^2(\nu)}$. It is simple  to check that  $H^1_{\rm env} $ is a  Hilbert space.

\subsection{Divergence}\label{div_omega}

\begin{Definition}\label{def_div}
Given a square integrable form $v\in L^2(\nu)$  we define its divergence  ${\rm div} \, v \in L^1(\cP_0)$ as 
\begin{equation}\label{emma}
{\rm div}\, v(\o)= \int d \hat{\o} (z)  r_{0,z}(\o) ( v(\o,z)-  v(\theta_z \o, -z) )\,.
\end{equation}
\end{Definition}
By applying Lemma \ref{lemma_siou} with $f$ such that $f(\o, \theta_z\o) =|v( \o, z)| $  for $\cP_0$--a.a. $\o$ (such a function  $f$ exists  by (A3) and Lemma \ref{matteo}) and by  Schwarz inequality, one gets for any $v\in L^2(\nu)$    that 
\begin{equation}\label{fuoco1}
\begin{split}
&\int d\cP_0(\o)
\int  d\hat{\o} (z)  r_{0,z}(\o) \bigl(| v(\o,z)|+|  v(\theta_z \o, -z)| \bigr)\\
& =
2  \|v\|_{L^1(\nu )} \leq 2 \bbE_0[\l_0] ^{1/2} \|v\|_{L^2(\nu)} < +\infty\,.
\end{split}
\end{equation}
In particular, the definition of divergence  is well posed and  the map $L^2(\nu) \ni v\mapsto {\rm div} v \in L^1(\cP_0)$ is continuous.

\smallskip

\begin{Lemma}\label{ponte}  For any  $v \in L^2(\nu)$ and any  bounded and measurable function $u:\O_0 \to \bbR$, it holds
\begin{equation}\label{italia}
\int  d \cP_0(\o)   {\rm div} \,v(\o)  u (\o)= - \int d \nu(\o, z) v( \o, z) \nabla u (\o, z) \,.
\end{equation}
 \end{Lemma}
 \begin{proof} 
 If is enough to apply Lemma  \ref{lemma_siou}  to $f(\o, \o')$ such that  
$f(\o, \theta_z \o) = v( \o, z)  u (\theta_z\o) $  for $\cP_0$--a.a. $\o$ (such a function  $f$ exists  by (A3) and Lemma \ref{matteo}) and observe that $f(\theta_z \o,\o) = v( \theta_z \o,- z)  u (\o) 
$.
\end{proof}

Trivially, the above result implies the following:
\begin{Corollary}\label{grazioso} Given a square integrable form $v\in L^2(\nu)$, we have that  $v\in L^2_{\rm sol}(\nu)$ if and only if ${\rm div}\, v=0$ $\cP_0$--a.s.
\end{Corollary}

%
\begin{Lemma}\label{garibaldi_100}
 If $\nabla u =0$  $\nu$--a.e.,  then   $u={\rm constant} $ $\cP_0$--a.s.
\end{Lemma}
\begin{proof} We define 
$
 A:=\{ \o \in \O_0\,:\, u(\theta_z \o)=u(\o) \;\; \forall z \in \hat \o \text{ with } r_{0,z}(\o)>0\}
 $.
Hence (recall \eqref{tildino}) 
$ \tilde A
  = \{ \o \in \O\,:\, u(\theta_y \o)=u(\theta_z \o) \; \forall y,z \in \hat \o \text{ with } r_{z,y}(\o)>0\}$.
The property that $\nabla u =0$  $\nu$--a.e. is equivalent to $\cP_0(A)=1$. By Lemma \ref{matteo} we get that $\cP( \tilde A)=1$.  \utto{Recall that  the property in  (A6) holds for $\o\in \O_*$,
$\O_*$ is translation invariant (as $\tilde A$) and $\cP(\O_*)=1$}. Moreover, given $\o \in \tilde A\cap \utto{\O_*}$, there exists a constant $c(\o)$ such that $u(\theta_y \o)=c(\o)$ for all $y \in \hat \o$.
Then we define $v(\o):=c(\o)$  if $\o\in \tilde A \cap \utto{\O_*}$  and $v(\o):= 0$ if $\o \not \in \tilde A\cap \utto{\O_*}$. As $v$ is translation invariant and $\cP$ is ergodic, there exists $c\in \bbR $ such that $\cP(v=c)=1$. By Lemma \ref{matteo} we get $\cP_0(v=c)=1$.
\end{proof}

The proof of the following lemma is similar to  the proof of  \cite[Lemma 2.5]{ZP} (see also
\cite[App.~B]{Fhom}). 
 Recall \eqref{anita}. 
\begin{Lemma}\label{minerale} Let $\z\in L^2(\cP_0)$ be orthogonal to  all functions 
$ g \in L^2(\cP_0) $ with $g = {\rm div} (\nabla u)$ for some  $ u \in H^1_{\rm env}$.
 Then $\z\in H^1_{\rm env}$ and 
$\nabla \z=0$ in $L^2(\nu)$.
\end{Lemma}

By combining Lemma \ref{garibaldi_100} and Lemma \ref{minerale} we get:
\begin{Lemma}\label{santanna} 
 The  functions $g\in L^2(\cP_0)$ of the form $g ={\rm div}\, v $ with $v\in L^2(\nu)$ are dense in $\{w \in L^2(\cP_0)\,:\, \bbE_0[w]=0\}$.
\end{Lemma}
\begin{proof}
 Lemma \ref{ponte} implies that  $\bbE_0[g]=0$ if  $g={\rm div}\,v $,  $v\in L^2(\nu)$.  
 Suppose that the claimed  density fails. Then there exists $\z\in  L^2(\cP_0)$ different from zero  with $ \bbE_0[\z]=0$ and such that $\bbE_0 [ \z g]=0$ for any $g \in L^2(\cP_0)$ of the form $g= {\rm div} v$ with $v\in L^2(\nu)$. 
By  Lemma \ref{minerale},  we know that $\z\in H^1_{\rm env}$ and $\nabla \z=0$ $\nu$--a.s. 
By Lemma \ref{garibaldi_100} we get that $\z$ is constant $\cP_0$--a.s.  Since $\bbE_0[\z]=0$ it must be $\z=0$ $\cP_0$--a.s., which is absurd.
\end{proof}
%
%


\section{The   diffusion matrix $D$ and the quadratic form $q$}\label{sec_diff_matrix}
Since $\l_2 \in L^1(\cP_0)$ (see Assumption (A7)), given $a\in \bbR^d$   the form 
\begin{equation}\label{def_ua}
u_a(\o,z):= a\cdot z
\end{equation}
 is square integrable, i.e. it belongs to $L^2(\nu)$. 
 We note  that the symmetric diffusion matrix $D$  defined in \eqref{def_D_R} satisfies, for any $a\in \bbR^d$,
 \begin{equation}\label{giallo}
 \begin{split}
q(a):= a \cdot Da  
&= \inf_{ v\in L^2 _{\rm pot}(\nu) } \frac{1}{2} \int d\nu(\o, x) \left(u_a(x)+v(\o,x) \right)^2\\
& =  \inf_{ v\in L^2 _{\rm pot}(\nu) } \frac{1}{2}\| u_a+v \|^2_{L^2(\nu)}=\frac{1}{2} \| u_a+v ^a \|^2_{L^2(\nu)}\,,
\end{split}
 \end{equation}
 where  $v^a=-\Pi u_a$ and 
 $\Pi: L^2(\nu) \to L^2_{\rm pot}(\nu)$ denotes the orthogonal projection of $L^2(\nu)$ on $L^2_{\rm pot}(\nu)$.   As a consequence, the map $\bbR^d \ni a \mapsto v^a\in L^2_{\rm pot}(\nu)$ is linear. 
Moreover, $v^a$ is characterized by the property
\begin{equation}\label{jung}
v^a \in  L^2_{\rm pot}(\nu)\,, \qquad v^a+u_a\in L^2_{\rm sol}(\nu)\,.
\end{equation}
Hence we can write
$ a\cdot Da=\frac{1}{2} \|  u_a+v^a \|_{L^2(\nu)}^2= \frac{1}{2}
 \la u_a, u_a + v^a \ra _{\nu}$.
As  the two symmetric bilinear forms $(a,b) \mapsto a\cdot Db$ and
$(a,b)\mapsto
  \frac{1}{2}\int d\nu (\o, z) a\cdot z \bigl(b\cdot z + v^{b}( \o, z) \bigr)=
\frac{1}{2}  \int d\nu \bigl (u_a+ v^{a}\bigr)
 \bigl( u_b + v^{b} \bigr) $ coincide on  diagonal terms, we get
 \begin{equation}\label{solare789}
 Da =\frac{1}{2}  \int d\nu (\o, z)  z \bigl(a\cdot z + v^a( \o, z) \bigr)\qquad \forall a \in \bbR^d\,.
\end{equation}

Let us come back to the 
 quadratic form $q$ on $\bbR^d$ defined in \eqref{giallo}. By \eqref{giallo}
 its  kernel  ${\rm Ker}(q)$  is given by 
\begin{equation}\label{rosone}
{\rm Ker}(q):=\{a\in \bbR^d\,:\, q(a)=0\}=\{ a\in \bbR^d\,:\, u_a \in L^2_{\rm pot}(\nu)\}\,.
\end{equation}

\begin{Lemma}\label{rock}
It holds
\begin{equation}\label{jazz}
{\rm Ker}(q)^\perp=\Big\{  \int  d \nu (\o,z) b(\o, z) z  \,:\, b\in L^2_{\rm sol} (\nu) \Big\}\,. 
\end{equation}
\end{Lemma}
Note that, since $\l_2 \in L^1(\cP_0)$ by  (A7),   the integral in the r.h.s. of \eqref{jazz} is well defined. The above lemma corresponds to \cite[Prop.~5.1]{ZP}.
\begin{proof}[Proof of Lemma \ref{rock}]
Let $b\in  L^2_{\rm sol}(\nu) $ and $\eta_b:=\int  d \nu (\o,z)  b(\o, z) z $. Then, given $a\in \bbR^d$, 
$a\cdot \eta_b =\la u_a , b \ra_{\nu}$. By \eqref{rosone},  $a \in  {\rm Ker}(q)$ if and only if $u_a\in L^2_{\rm pot}(\nu)= L^2_{\rm sol }(\nu)^\perp $.
Therefore $a \in  {\rm Ker}(q)$ if and only if  $a\cdot \eta_b =0$ for any $b \in L^2_{\rm sol}(\nu)$. \end{proof}

Due to Lemma \ref{rock} and Definition  \ref{divido} we have:
\begin{Corollary}\label{jack}
$\text{Span}\{ \mathfrak{e}_1, \dots, \mathfrak{e}_{d_*}\}= \bigl\{  \int  d \nu (\o,z) b(\o, z) z  \,:\, b\in L^2_{\rm sol} (\nu) \bigr\}$.
\end{Corollary}


\section{The contraction $b(\o,z)\mapsto \hat b(\o) $ and the set  $\cA_1[b]$}\label{sec_cinghia}

\begin{Definition}\label{artic}  Let $b: \O _0 \times \bbG\to \bbR$ be a \rrr{measurable} function  with $\|b\|_{L^1(\nu)}<+\infty$. We define the \rrr{measurable} function  $r_b:\O_0 \to [0,+\infty]$ as
\begin{equation}\label{magno}
r_b (\o):=  \int d \hat{\o}(z) r_{0,z}(\o)  |b(\o,z)|\,, \end{equation}
 the \rrr{measurable} function $\hat b: \O_0 \to \bbR$ as
\begin{equation}\label{zuppa}
\hat b(\o):= 
\begin{cases}
\int d \hat{\o}(z) r_{0,z}(\o) b(\o,z) & \text{ if } r_b (\o) <+\infty\,,\\
0 & \text{ if } r_b (\o) = +\infty\,,
\end{cases}
\end{equation}
and the \rrr{measurable} set $\cA_1[b]:= \{\utto{ \o\in \O_*}\,:\, r_{b } ( \theta _z \o) <+\infty\; \forall z\in \hat \o\}$.
\end{Definition}

\begin{Lemma}\label{cavallo}   Let $b: \O _0 \times \bbG \to \bbR$ be a \rrr{measurable} function  with $\|b\|_{L^1(\nu)}<+\infty$. 
 Then 
 \begin{itemize}
 \item[(i)] 
$\| \hat b\|_{L^1(\cP_0)} \leq  \| b\|_{L^1(\nu)}= \| r_b \|_{L^1(\cP_0)} $ and  $\bbE_0[\hat b]= \nu\ovo{[b]}$;
\item[(ii)]
given $\o \in  \cA_1[b]$ and 
   $\varphi \in C_c(\bbR^d)$, it holds\begin{equation}\label{rino}
\int     d\mu_\o^\e (x) \varphi ( x) \hat b (\theta_{x/\e} \o) = \int  d \nu_\o^\e (x, z) \varphi(x) b( \theta_{x/\e} \o, z) 
\end{equation}
(the  series in the l.h.s. and in the r.h.s.  are absolutely convergent);
\item[(iii)]    $\cP( \cA_1[b])=\cP_0( \cA_1[b])=1$ and 
$ \cA_1[b]$ is translation invariant. 
\end{itemize}
\end{Lemma}
\begin{proof}  It is trivial to check Item (i) and Item (ii).
%
  We move to Item (iii). We have 
$\bbE_0[ r_b] = \| b\| _{L^1(\nu)} <\infty$. This implies that 
  $\cP_0(\{\o\,:\, r_{b} (\o) <+\infty\})=1$ and  therefore $\cP( \cA_1[b])=\cP_0( \cA_1[b])=1$  by Lemma \ref{matteo}. The last property of  $\cA_1[b]$ follows immediately from  the definition.
 \end{proof}

We point out that, since $\nu $ has finite mass, $L^2(\nu) \subset L^1(\nu)$ and therefore Lemma \ref{cavallo} can be applied to $b$ with $\|b \|_{L^2(\nu) }<+\infty$.

%
%
%
%

\section{The transformation $b(\o,z)\mapsto \tilde b(\o,z) $}\label{hermione}
\begin{Definition}\label{ometto} 
Given  a \rrr{measurable} function  $b :\O_0\times \bbG \to \bbR$ we define  $\tilde b :\O_0\times \bbG\to \bbR$  as
\be
\ovo{\tilde b (\o, z):=
\begin{cases} 
b (\theta_{z} \o, -z)  &\text{ if } z\in \hat \o\,,\\
0 &\text{ otherwise}\,.
\end{cases}
}
\en
\end{Definition}
%
%
%
%
\begin{Lemma}\label{gattonaZZ}  Given  a \rrr{measurable} function  $b :\O_0\times \bbG \to \bbR$,
it  holds  $\tilde{\tilde b}(\o,z)=b(\o,z)$  and 
 $ \| b \| _{L^p(\nu)} =  \| \tilde b \| _{L^p(\nu)}$ for any $p$.
 If 
 $b \in L^2(\nu)$, then 
  ${\rm div} \,\tilde b= - {\rm div}\, b$. \end{Lemma}
 \begin{proof}   To get that  $ \| b \| _{L^p(\nu)} =  \| \tilde b \| _{L^p(\nu)}$ it is enough to 
  apply Lemma \ref{lemma_siou} with  $f$ such that $f(\o, \theta_z \o ) =
| b (\o, z )|^p$ for $\cP_0$--a.a. $\o$ (to define $f$ use  
\utto{(A3)}).
 The  other identities are trivial.
\end{proof}


Recall 
Definition \ref{artic}.
 
 
  \begin{Lemma}\label{gattonaZ} 
(i) 
 Let $b: \O_0 \times \bbG \to [0, +\infty]$ and  $\varphi, \psi:\bbR^d \to [0,+\infty] $  be
  \rrr{measurable} functions. Then, for each $\o \in\utto{ \O_*}$,  it holds
  \begin{equation}\label{micio2}
 \int d \nu^\e_\o (x, z) \varphi ( x) \psi (  x+\e z ) b(\theta_{x/\e}\o, z)=
  \int d \nu^\e_\o (x, z) \psi (x) \varphi (  x+\e z  ) \tilde{b}(\theta_{x/\e}\o, z)\,.
 \end{equation}
 (ii)  Let $b: \O_0 \times \bbG\to \bbR$ be a \rrr{measurable} function
 and take $ \o \in \cA_1[b]\cap \cA_1[\tilde b]$.  
   Given functions  $\varphi, \psi:\bbR^d \to \bbR $   such that  at least one between $\varphi, \psi $ has  compact support  and the other is bounded,  identity \eqref{micio2} is still valid.
   Given now $\varphi $ with compact support and  $\psi$ bounded, it holds
 \begin{align}
&  \int d \nu^\e_\o (x, z) \nabla_\e  \varphi   ( x,z) \psi (   x+\e z ) b(\theta_{x/\e}\o, z) \nonumber\\
 & \qquad \qquad \qquad  =
  - \int d \nu^\e_\o (x, z) 
  \nabla _\e \varphi  (x,z) 
  \psi (x) \tilde{b}(\theta_{x/\e}\o, z)\,.\label{micio3}
   \end{align}
Moreover,  the above integrals  in  \eqref{micio2}, \eqref{micio3} (under the hypothesis of this  Item (ii)) 
 correspond to absolutely convergent  series and are therefore well defined. 
 
\end{Lemma}
\begin{proof}  We check \eqref{micio2} in Item  (i).
  Since $c_{a,a'}(\o)= c_{a',a}(\o)$ and $b(\theta_a \o, a'-a)= \tilde b(\theta_{a'} \o, a-a')$ for all $a,a'\in \hat \o$ \utto{(as $\o\in \O_*$)}, we can write  
 \begin{equation*}
 \begin{split}
 \text{l.h.s. of \eqref{micio2}}& = \e^d \sum_{a\in \hat \o} \sum_{a'\in  \hat \o } c_{a,a'}(\o) \varphi (\e a ) \psi (\e a') b( \theta_a \o, a'-a)
 \\
 &= \e^d  \sum_{a'\in  \hat \o } \sum_{a\in \hat \o} c_{a',a}(\o)\psi (\e a')   \varphi (\e a ) \tilde b(\theta_{a'} \o, a-a')=\text{r.h.s. of \eqref{micio2}}\,.
  \end{split}
 \end{equation*}
 The above identities hold also in Item (ii) since one deals indeed with absolutely convergent sum. For example, when $\varphi$ has compact support and $\psi $ is bounded, it is enough to observe that (as $\o \in \cA_1[b]$) 
 \be\label{tom2}
  \int d \nu^\e_\o (x, z)| \varphi ( x)| | b(\theta_{x/\e}\o, z)| \leq  \int d \mu^\e_\o (x) | \varphi  ( x)| r_b (\theta_{x/\e}\o)<+\infty\,.
  \en
  

 We now prove \eqref{micio3}.  We have 
\begin{multline*}
\text{l.h.s. of \eqref{micio3}}  = \e^d \sum_{a\in \hat \o} \sum_{a' \in  \hat \o } c_{a,a'}(\o) \frac{ \varphi (\e a')-\varphi ( \e a)}{\e} \psi (\e a') b( \theta_a \o, a'-a)
 \\
  =-  \e^d  \sum_{a' \in  \hat \o }  \sum_{a\in \hat \o} c_{a',a }(\o) 
 \frac{ \varphi (\e a)-\varphi ( \e a ')}{\e}  \psi (\e a')   \tilde b(\theta_{a'} \o, a-a')
 =\text{r.h.s. of \eqref{micio3}}\,.
  \end{multline*}
 The above arrangements are indeed legal as one can easily prove that  the above \ovo{series are  absolutely convergent}.
 \end{proof}

\begin{Definition}\label{naviglio} Let $b: \O_0\times \bbG \to \bbR$ be a  \rrr{measurable} function.
 If $\o \in   \cA_1[b]\cap \cA_1[\tilde b]\cap \O_0$,  we set ${\rm div}_* b  (\o) := \hat b (\o) - \hat{\tilde{b}}(\o) \in \bbR$.
\end{Definition}
\begin{Lemma}\label{arranco}
Let  $b: \O_0\times \bbG  \to \bbR$ be a \rrr{measurable} function with $\|b\|_{ L^2(\nu)}<+\infty$. Then $\cP_0 ( \cA_1[b]\cap \cA_1[\tilde b]) =1$ and 
${\rm div}_* b  = {\rm div }\,b $ in $L^1(\cP_0)$.
\end{Lemma}
\begin{proof}
By Lemma \ref{gattonaZZ} we have $  \| \tilde b \| _{L^2(\nu)}<\infty$. Hence, both $b$ and $\tilde b$ are $\nu$--integrable. By Lemma \ref{cavallo}--(iii) we get that $\cP_0 ( \cA_1[b]\cap \cA_1[\tilde b]) =1$. The identity  ${\rm div}_* b   = {\rm div} b $ in $L^1(\cP_0)$ is trivial.
\end{proof}
\begin{Lemma}\label{tav}
Let  $b: \O_0\times \bbG \to \bbR$ be a \rrr{measurable} function with  $\|b\|_{ L^2(\nu)}<+\infty$ and such that  its class of equivalence in $L^2(\nu)$ belongs to $L^2_{\rm sol}(\nu)$. Let 
\be\label{eq_tav}
\cA_d [b]:= \{\o \in \cA_1[b] \cap \cA_1[ \tilde b]\,:\, {\rm div}_* b (\theta_z \o) =0 \; \forall z \in \hat \o\}\,.
\en
Then  $\cP( \cA_d [b])=1$ and  $\cA_d [b]$ is translation invariant.
\end{Lemma}
\begin{proof} By Corollary \ref{grazioso} and  Lemma \ref{arranco} we have $\cP_0(A)=1$ where   $A:= \{\o\in \cA_1[b] \cap \cA_1[ \tilde b]\cap\O_0\,:\, {\rm div}_*b ( \o) =0 \}$. By Lemma \ref{matteo} and \eqref{tildino}, $\cP(\tilde A)=1$. To get that  $\cP( \cA_d [b])=1$ it is enough to observe that $\tilde A=  \cA_d [b]$. The translation invariance  is trivial.\end{proof}
\begin{Lemma}\label{lunetta}
Suppose that $b: \O_0\times \bbG \to \bbR$ is a \rrr{measurable} function with $\|b\|_{ L^2(\nu)}<+\infty$. Take $\o \in   \cA_1[b]\cap \cA_1[\tilde b]$.
 Then for any $\e>0$ and any $u:\bbR^d\to \bbR$ with compact support it holds
\begin{equation}\label{sea}
\int d \mu^\e_\o (x) u(x) {\rm div}_* b (\theta_{x/\e} \o) = - \e \int d \nu ^\e_\o (x,z) \nabla_\e u(x,z) b ( \theta_{x/\e} \o, z) \,.
\end{equation}
\end{Lemma}
\begin{proof} Note that $\theta_{x/\e} \o$ in the l.h.s. of \eqref{sea} belongs to $ \cA_1[b]\cap \cA_1[\tilde b]\cap \O_0$. Hence, by Definition \ref{naviglio},
we can write the l.h.s. of \eqref{sea} as  
\begin{equation}\label{carte}
 \int d \nu ^\e _\o (x,z) u(x) b (\theta_{x/\e} \o, z) - \int d \nu ^\e_\o (x,z) u(x) \tilde b ( \theta_{x/\e} \o, z) \,.
 \end{equation}
Due to our assumptions  we are dealing with absolutely convergent series, hence the above rearrangements are free. By applying \eqref{micio2} to the r.h.s. of \eqref{carte} (see Item (ii) of Lemma \ref{gattonaZ}), 
we can rewrite \eqref{carte} as 
$\int d \nu ^\e_\o (x,z)   b ( \theta_{x/\e} \o, z) [ u(x) - u(x+ \e z)]$ and this allows to conclude.
\end{proof}

\section{Typical environments}\label{topo}

%
%
%
%
%
%
%
%

We can now describe the set $\O_{\rm typ}$ of typical environments appearing in Theorem \ref{teo1} and Theorem \ref{teo2}.
We first fix some basic notation and observations, frequently used below. Given $M>0$ and $a \in \bbR$,  we define $[a]_M$ as 
  \begin{equation}\label{taglio}
  [a]_M= M \mathds{1}_{\{a>M\}} +a \mathds{1}_{\{ |a|\leq M \}} -M \mathds{1} _{\{a<-M\}}\,.
  \end{equation}
 Given $a\geq b$, it holds  $a-b \geq [a]_M - [b]_M\geq 0$. Hence, for any $a, b\in \bbR$, it holds
   \begin{align}
&   \left |  [a]_M - [b]_M   \right |\leq\left |a-b\right |\,,  \label{r101}\\
&  \left | [ a-b]- [[a]_M - [b]_M ]  \right| \leq \left|a-b\right|
   \,.\label{r102}
  \end{align}
Recall that, 
due to Assumption (A8), the space $L^2(\cP_0)$ is separable. 
\begin{Lemma}\label{separati_sole} The space $L^2(\nu)$ is separable. 
\end{Lemma}
\begin{proof}
By the separability of $L^2(\cP_0)$ there exists a countable dense set $\{ f_j\}$  in $L^2(\cP_0)$. At cost  to  approximate, in $L^2(\cP_0)$, $f_j $ by $[f_j]_M$  as $M\to \infty$
(cf. \eqref{taglio}), we can suppose that  $f_j$ is bounded.  Let $\{B_k\}$ be the countable family of closed balls in $\bbR^d$ with rational radius and with center in  $\bbQ^d$.  It is then trivial to check that the zero function is the only  function in  $L^2(\nu)$ orthogonal to all functions $ f_j(\o) \mathds{1}_{B_k}(z)$ (which belong indeed to $L^2(\nu)$).
\end{proof}
%

In the  construction of the functional sets  presented below,  we will use the separability of $L^2(\cP_0)$ and $L^2(\nu)$ without further mention.
The definition of these functional sets and  the typical environments (cf. Definition  \ref{budda}) consists of a list of technical assumptions, which are necessary to justify several steps in the next sections (there, we will indicate explicitly which  technical assumption we are  using).

\smallskip
Recall  the sets $\cA_1[b]$ and $\cA_d[b]$  introduced  respectively in Definition \ref{artic} and \eqref{eq_tav}. Recall Definition \ref{ometto} of $\tilde b$. Recall  \eqref{tildino}. 
Recall the set $\cA[f]$  in Proposition \ref{prop_ergodico} for a  \rrr{measurable} function  $f: \O_0\to \bbR$ with $\|f\|_{L^1(\cP_0)}<+\infty$. 

\begin{Definition}\label{defA}
Given a function $f: \O_0\to [0,+\infty]$ such that  $\|f\|_{L^1(\cP_0)}<+\infty$,  we define $\cA[f]$ as $\cA[f_0]$, where $f_0: \O_0 \to \bbR$ is defined as $f$ on $\{ f <+\infty\}$ and as $0$ on $\{ f =+\infty\}$. 
\end{Definition}

\noindent
$\odot$ {\bf The functional sets $\cG_1,\cH_1$}. 
We fix a countable set $\cH_1$ of \rrr{measurable} functions $b: \O_0\times \bbG\to \bbR$ such that 
$\|b \|_{L^2(\nu)}<+\infty$ for any $b \in \cH_1$ and such that   $\{ {\rm div} \,b\,:\, b \in \cH_1\}$ is a dense subset of $\{ w \in L^2(\cP_0)\,:\, \bbE_0[ w]=0\}$ when thought of as set of $L^2$--functions (recall Lemma \ref{santanna}). 
For each $b \in \cH_1$ we define  the \rrr{measurable} function $g_b : \O_0 \to \bbR$ as 
\be \label{lupetto}
g_b (\o):= 
\begin{cases}
{\rm div}_* b (\o) & \text{ if } \o \in \cA_1[b]\cap \cA_1[\tilde b]\,,\\
0 & \text{ otherwise}\,.
\end{cases}
\en
Note that by Lemma \ref{arranco} $g_b={\rm div}\, b$ $\cP_0$--a.s.
We set $\cG_1:=\{ g_b \,:\, b \in \cH_1\}$.

\smallskip

\noindent
$\odot$ {\bf The functional sets $\cG_2,\cH_2, \cH_3$}.  We fix a countable set $\cG_2$  of  bounded \rrr{measurable} functions 
$g: \O_0\to \bbR$ such that  the set $\{ \nabla g\,:\, g \in \cG_2\}$, thought in $L^2(\nu)$,  is  dense in $L^2_{\rm pot}(\nu)$
 (this is possible by the definition of $L^2_{\rm pot}(\nu)$).  We  define $\cH_2$ as the set of \rrr{measurable} functions $h:\O_0 \times \bbG \to \bbR$ such that $h=\nabla g$ for some $g\in \cG_2$. We define $\cH_3$ as  the set of \rrr{measurable} functions $h: \O_0 \times \bbG \to \bbR$ such that $h(\o,z)= g(\theta_z \o) z_i$ for some $g\in \cG_2$ and some direction $i=1,\dots, d$. Note that, since 
  $\bbE_0[\l_2]<+\infty$ by (A7) and since $g$ is bounded,  $\|h\|_{L^2(\nu)} <+\infty$ for all $h \in \cH_3$.

\smallskip

\noindent
$\odot$ {\bf The functional set $\cW$}. We fix a  countable set $\cW$ of \rrr{measurable} functions $b:\O_0\times \bbG\to \bbR$ such that, thought of as subset of $L^2(\nu)$, $\cW$  is dense in $L^2_{\rm sol} (\nu)$.  By  Corollary \ref{grazioso} and Lemma~\ref{gattonaZZ},  $\tilde b \in L^2_{\rm sol}(\nu)$ for any $b \in L^2_{\rm sol}(\nu)$. Hence, at cost to enlarge $\cW$, we assume that $\tilde b \in \cW$ for any $b \in \cW$.

\smallskip

\noindent
$\odot$ {\bf The functional set $\cG$}. 
We fix a countable set $\cG$ of \rrr{measurable} functions $g:\O_0\to \bbR$ such that: 
\begin{itemize}
\item  $\|g\|_{L^2(\cP_0)}<+\infty$ for any $g\in \cG$ and  $\cG$  is dense in $L^2(\cP_0)$ when thought of as a subset of  $L^2(\cP_0)$;
\item  $ 1\in \cG$, $\cG_1\subset \cG$, $\cG_2\subset \cG$;
\item  $\l_0 \wedge \sqrt{M} \in \cG$ for any $M\in \bbN$;
\item 
for each $b \in \cW$, $M \in \bbN$,  $i\in\{1,\dots,d\}$ and $\ell \in \bbN$,   the function $[f]_\ell: \O_0 \to \bbR$ where (cf. \eqref{taglio})
\be\label{nonfamale}
f(\o):= \begin{cases}
 \int d \hat \o (z) r_{0,z}(\o) z_i \mathds{1}_{\{ |z|\leq \ell\}} [b]_M(\o,z) & \text{if } \int d \hat \o (z) r_{0,z}(\o)|z_i|<+\infty \\
0 &\text{otherwise}
\end{cases}
\en
belongs to $\cG$;
\item at cost to enlarge $\cG$ we assume that $[g]_M\in \cG$ for any $g\in \cG$ and $M\in \bbN$.
\end{itemize}

\smallskip

\noindent
$\odot$ {\bf The functional set $\cH$}. 
We fix a countable set of \rrr{measurable} functions $b : \O_0 \times\bbG \to \bbR$ such that  
\begin{itemize}
\item $\|b\|_{L^2(\nu)}<+\infty$  for any $b\in \cH$ and $\cH$ is dense in $L^2(\nu)$ when  thought  of as a subset of  $L^2(\nu)$;
\item $\cH_1\cup  \cH_2 \cup \cH_3\cup  \cW \subset \cH$ and  $ 1\in \cH$;
\item  $\forall i=1,\dots, d$ the map $(\o,z) \mapsto z_i$ is in $\cH$ (recall: $\l_2 \in L^1(\cP_0) $ by (A7));
\item  at cost to enlarge $\cH$ we assume that $[b]_M\in \cH$ and  that $\tilde b\in \cH$ for any $b\in \cH $ and $M\in \bbN$ ($\tilde b \in L^2(\nu)$ by Lemma \ref{gattonaZZ}).
\end{itemize}


\begin{Definition}\label{budda}
We define $\O_{\rm typ}$ as the intersection of the following sets:
\begin{itemize}
\item $\cA[gg']$ for all  $g,g'\in \cG$ (recall that $1\in \cG$);
\item $\cA_1[bb']\cap \cA [\widehat{b b'}]$ for all  $b,b'\in \cH$ (recall that $1\in \cH$, $\tilde b \in \cH$ $\forall b\in \cH$ and recall Lemma \ref{cavallo}); 
\item $ \tilde A \cap\cA[\l_i]$ where $A:=\{\l_i<+\infty\}$ and $i=0,1,2$ (recall  (A7) and \eqref{tildino});
\item $\tilde A\cap  \cA[\l_0\mathds{1}_{\{\l_0 > \sqrt{M}\}}] $ with $A=\{\l_0<+\infty\}$ and $M\in \bbN$ (recall  (A7));
\item $ \cA_d[b]$ for all $b \in \cW$ (recall \eqref{eq_tav});
\item $ \cA[h_\ell ]$ where $\ell \in \bbN$ and $h_\ell (\o):=\int d\hat \o (z) r_{0,z}(\o) |z|^2 \mathds{1}_{\{|z| \geq \ell \}}$;
\item $ \cA[ f-[f]_\ell]$ where $f$  varies among the functions \eqref{nonfamale} with
 $b \in \cW$, $M \in \bbN$,  $i\in\{1,\dots,d\}$ and $\ell \in \bbN$.
\end{itemize}
\end{Definition}

\begin{Remark} \utto{$\O_{\rm typ} \subset \O_*\cap \O_1$ (see our Assumptions and \eqref{alba_chiara} for the definition of $\O_*$ and $\O_1$).}
\end{Remark}
\begin{Proposition}\label{stemino}
The above set $\O_{\rm typ}$ is a translation invariant \rrr{measurable} subset of $\O$ such that  $\cP(\O_{\rm typ})=1$.
\end{Proposition}
\begin{proof}
The claim  follows from Proposition \ref{prop_ergodico} for all sets  of the form $\cA[\cdot ]$, from   Lemma \ref{cavallo} 
for all sets  of the form $\cA_1[\cdot ]$,  from 
 Lemma   \ref{tav}  for all sets  of the form $\cA_d[\cdot ]$ and from Corollary \ref{eleonora} for  all sets of the form $\tilde A$.   \end{proof}

\section{2-scale convergence of $v_\e \in L^2 (\mu^\e_{\tilde\o})$ and  of $w_\e \in L^2 (\nu^\e_{\tilde \o})$}\label{sec_2scale}
In this section we recall the notion of 2-scale convergence in our context.
\begin{Definition}\label{priscilla}
Fix   $\tilde \o\in \O_{\rm typ}$, an $\e$--parametrized  family $v_\e \in L^2( \mu^\e_{\tilde \o})$ and  a function $v \in L^2 \bigl(m dx \times \cP_0 \bigr)$.
\\
$\bullet$ We say that \emph{$v_\e$ is weakly 2-scale convergent to $v$}, and write 
$v_\e \stackrel{2}{\rightharpoonup} v$, 
if the family $\{v_\e\}$ is bounded, i.e.
$
 \limsup_{\e\downarrow 0}  \|v_\e\|_{L^2( \mu^\e_{\tilde \o})}<+\infty$, 
 and 
\begin{equation}\label{rabarbaro}
\lim _{\e\downarrow 0} \int d \mu _{\tilde \o}^\e (x)  v_\e (x) \varphi (x) g ( \theta _{x/\e} \tilde{\o} ) =\int d\cP_0(\o)\int dx\, m v(x, \o) \varphi (x) g (\o)  \,,
\end{equation}
for any $\varphi \in C_c (\bbR^d)$ and any $g \in \cG$.\\
$\bullet$
 We say that \emph{$v_\e$ is strongly 2-scale convergent to $v$}, and write 
$v_\e \stackrel{2}{\to} v$, 
if the family $\{v_\e\}$ is bounded and 
\begin{equation}\label{gingsen}
\lim_{\e\downarrow 0} \int d\mu^\e_{\tilde \o}(x)   v_\e(x) u_\e(x) = \int d\cP_0(\o)\int  dx\,m   v(x, \o) u(x,\o)  
\end{equation}
whenever $u_\e \stackrel{2}{\rightharpoonup} u$. 
\end{Definition}

\begin{Lemma}\label{medaglia} Let $\tilde \o \in \O_{\rm typ}$. Then,
for any $\varphi\in C_c (\bbR^d) $ and $g \in \cG$, setting $ v_\e(x):= \varphi(x) g(\theta_{x/\e} \tilde \omega)$ it holds 
$L^2(\mu^\e_{\tilde \o} )\ni v_\e \stackrel{2}{\rightharpoonup} \varphi(x) g(\o)\in L^2( m dx \times \cP_0)$.
\end{Lemma}
\begin{proof} Since $\tilde \o \in \O_{\rm typ}\vvv{\subset \cA[g]}$, we get $\lim_{\e\da 0}  \|v_\e\|_{L^2(\mu^\e_{\tilde \o})}^2= \int dx\, m \varphi(x)^2 \ovo{\bbE_0}[g^2]$, hence $\{v_\e\}$ is bounded in $L^2(\mu^\e_{\tilde \o})$.
 Since $g\in \cG\subset  L^2(\cP_0)$, we have $\varphi(x) g(\o) \in L^2(m dx \times \cP_0)$. Take $\varphi_1 \in C_c(\bbR^d)$ and $g_1\in \cG$. Since $\tilde \o \in \O_{\rm typ}\vvv{\subset \cA[g g_1]}$, it holds $\int d \mu _{\tilde \o}^\e (x)  v_\e (x) \varphi _1(x) g_1 ( \theta _{x/\e} \tilde{\o} ) \to\int dx\, m \varphi(x) \varphi_1(x) \bbE_0[g g_1]$.\end{proof}
\begin{Lemma}\label{alba}
Given $\tilde \o \in \O_{\rm typ}$, if $v_\e \stackrel{2}{\toup} v$ then 
\begin{equation}\label{acquario}
\varlimsup _{\e \da 0} \int d\mu^\e_{\tilde \o}(x)\ovo{v_\e(x)^2}   \geq \int d\cP_0(\o)\int dx\,m v(x, \o)^2\,.
\end{equation}
\end{Lemma} 
The proof is similar to the proof of \cite[Item~(iii),~p.~984]{Z} and therefore omitted:
replace    $\Phi$ in \cite{Z} with    a linear combination of functions $\varphi(x) g(\o)$ with $\varphi \in C_c(\bbR^d)$ and $g\in \cG$, use  the density of $\cG$ in $L^2(\cP_0)$ and the property that  \vvv{$\O_{\rm typ}\subset \cA[gg']$} for all  $g,g'\in \cG$ 
  (cf.   \cite[Lemma 10.5]{Fhom}).

Using Lemmas \ref{medaglia} and \ref{alba} one gets the following characterization:
\begin{Lemma}\label{fantasia}
Given  $\tilde \o \in \O_{\rm typ}$, $v_\e \stackrel{2}{\to} v$ if  and only if $v_\e \stackrel{2}{\rightharpoonup} v$ and 
\begin{equation}\label{orchino}
\lim _{\e\da 0} \int_{\bbR^d} d\mu^\e_{\tilde \o}(x) v_\e(x)^2 = \int_\O d \cP_0(\o) \int _{\bbR^d} dx \,m v(x,\o)^2 \,.
\end{equation}
\end{Lemma}
\begin{proof}
 If $v_\e \to v$, then
 $v_\e \stackrel{2}{\rightharpoonup} v$ by 
    Lemma \ref{medaglia}.  By then applying \eqref{gingsen} with $u_\e:=v_\e$, we get \eqref{orchino}. The opposite  implication corresponds to \cite[Item~(iv),~p.~984]{Z} and the proof there can be easily adapted to our setting due to Lemma \ref{alba}.
(cf.   \cite[Lemma 10.4]{Fhom}). 
\end{proof}


\begin{Lemma}\label{compatto1}
 Let $\tilde \o\in \O_{\rm typ}$.    Then, given a bounded family of functions $v_\e\in L^2 ( \mu^\e_{\tilde{\o}})$,  there exists a  sequence $\e_k \downarrow 0 $ such that
 $ v_{\e_k}
  \stackrel{2}{\rightharpoonup}   v $ for some 
 $  v \in L^2(m dx \times \cP_0 )$ with $\|  v\|_{   L^2(m dx \times \cP_0 )}\leq \limsup_{\e\da 0} \|v_\e\|_{L^2( \mu^\e_{\tilde{\o}})}$.
 \end{Lemma}
 The proof of the above lemma  is similar to the proof of \cite[Prop.~2.2]{Z}, but in \cite{Z} some  density in uniform norm is used. Since that density is absent here, we provide the proof in Appendix \ref{app_compatto} to explain how to fill the gap.


 \begin{Definition}\label{foodgang}
Given  $\tilde \o\in \O_{\rm typ}$, a family $w_\e \in L^2( \nu ^\e_{\tilde \o})$ and  a function  $w \in L^2 \bigl( m dx \times d\nu\bigr)$, we say that \emph{$w_\e$ is weakly 2-scale convergent to $w$}, and write $  w_\e \stackrel{2}{\rightharpoonup}w   $,  if   $\{w_\e\}$ is bounded
 in $L^2( \nu ^\e_{\tilde \o})$, i.e. $
 \varlimsup_{\e \da 0} \|  w_\e\|_{L^2( \nu ^\e_{\tilde \o})}<+\infty$,
   and
\begin{multline}\label{yelena}
 \lim _{\e\downarrow 0} \int   d \nu_{\tilde \o}^\e (x,z) w_\e (x,z ) \varphi (x) b ( \theta _{x/\e} \tilde{\o},z )\\
=\int  dx\,m \int d \nu (\o, z) w(x, \o,z ) \varphi (x) b (\o,z )  \,,
\end{multline}
for any $\varphi \in C_c (\bbR^d)$ and  any $b \in \cH $. 
\end{Definition}



\begin{Lemma}\label{compatto2} Let $\tilde \o\in \O_{\rm typ}$.  Then, given a bounded family of functions $w_\e\in L^2 ( \nu^\e_{\tilde{\o}})$,  there exists a \rrr{sequence} $\e_k\downarrow 0$   such that  $w_{\e_k} \stackrel{2}{\rightharpoonup} w$ for some 
 $ w \in L^2(  m\,dx \times \nu )$ with  $\|  w\|_{   L^2(m dx \times \nu )}\leq \limsup_{\e\da 0} \|w_\e\|_{L^2( \nu^\e_{\tilde{\o}})}$.
\end{Lemma}
We postpone  a sketch of the proof of Lemma \ref{compatto2} to Appendix \ref{app_compatto}.


\section{\ovo{Roadmap of the proof of Theorem \ref{teo1}}}\label{mappa}
The main strategy to prove Theorem \ref{teo1} is the same  e.g. of the one  in \cite{ZP} to prove  Theorem 
6.1 there. Below we provide a list of the main steps to arrive to some key structure  results concerning 
 the solutions $u_\e$ of \eqref{eq1} and allowing then to easily conclude the proof. The details are provided in the next sections. 
 
We  fix a typical  environment $\tilde \o\in \O_{\rm typ}$. The first step is to  prove the following structure result: given a bounded  family  of functions $u_\e$ in  $ H^1_{\tilde \o, \e} $,  along a subsequence $\e_k$ it holds 
\begin{align*}
& L^2(\mu^\e_{\tilde \o}) \ni u_\e \stackrel{2}{\toup} u\in  L^2(  m dx\times \cP_0)\,,\\
& L^2  (\nu^\e_{\tilde \o}) \ni \nabla u_\e  (x,z) \stackrel{2}{\toup} w(x,\o,z):=  \nabla_*  u(x) \cdot z + u_1 (x,\o,z)\in L^2( m dx \times \nu)\,,
\end{align*}
for suitable functions $u, u_1$ with  $u=u(x)\in H^1 _*( m dx )$,  $u_1\in L^2\bigl( \bbR^d, L^2_{\rm pot} (\nu)\bigr )$.
We devote Sections \ref{cut-off1}, \ref{sec_risso}, \ref{cut-off2} and \ref{sec_oro} to the above  result.

 Suppose now that  $\{f_\e\}$ is a bounded family in $L^2(\mu^\e_{\tilde \o})$. Let   $u_\e\in H^{1,{\rm f}}_{{\tilde \o}, \e}$ be the  weak solution of \eqref{eq1}, i.e.  
\be\label{blanco}
\frac{1}{2} \la \nabla_\e v, \nabla _\e u_\e \ra _{\nu_{\tilde \o} ^\e}+ \la v, u_\e \ra _{\mu_{\tilde \o}^\e}= \la v, f_\e \ra_{\mu_{\tilde \o}^\e}\,,\qquad \forall v \in  H^{1,{\rm f}}_{{\tilde \o}, \e}\,.
\en It is then standard to get from \eqref{blanco}   that the family $\{u_\e\}$ is bounded in   $H^1_{\tilde \o, \e}$. Hence  we can apply the above structure result. 

The next step is to prove that, for $dx$--a.e.~$x$,   the map $(\o,z) \mapsto w(x,\o,z)$ belong to $L^2_{\rm sol}(\nu)$ (this is the first part of the proof of Claim \ref{vanity} in Section \ref{robot} and relies on a suitable choice of $v$ in \eqref{blanco}).

As $ w(x,\o,z):=  \nabla_*  u(x) \cdot z + u_1 (x,\o,z)$,  $w(x,\cdot,\cdot)\in L^2_{\rm sol}(\nu)$ and $u_1 (x, \cdot,\cdot) \in L^2_{\rm pot} (\nu)$ for $dx$--a.e.~$x$, from \eqref{jung} we get that $u_1(x,\cdot,\cdot)= v^a(\cdot,\cdot)$ with $a= \nabla_* u(x)$ for $dx$--a.e.~$x$ (i.e. $u_1(x,\cdot,\cdot)$ is the orthogonal projection  onto $L^2_{\rm pot}(\nu)$ of the form $(\o,z) \mapsto - \nabla_* u(x) \cdot z$).  As a byproduct with \eqref{solare789} we then get that  $\int d \nu(\o, z)  w(x,\o, z)  z= 2 D  \nabla_* u (x)$ for  $dx$--a.e.~$x$ (see Claim \ref{vanity}).   The effective  homogenized matrix $D$   has finally emerged. 

From  this point the conclusion of the proof of Theorem \ref{teo1} becomes  relatively simple and is detailed in Section \ref{robot} after the proof of Claim \ref{vanity}.


\section{Cut-off for functions  $v_\e\in L^2(\mu^\e_{\tilde \o})$}\label{cut-off1}
 $\bbN_+$ denotes the  set of positive integers. Recall \eqref{taglio}.
\begin{Lemma}\label{lemma1}
Let $\tilde \o \in \O_{\rm typ}$ and let $\{v_\e\}$ be a family of functions such that   $v_\e\in L^2(\mu^\e_{\tilde \o})$ and 
$\limsup_{\e\da 0} \| v_\e\|_{L^2(\mu^\e_{\tilde \o})}<+\infty$. Then  there exist functions $v, v_M\in L^2 (m dx \times \cP_0) $ with $M$ varying in $\bbN_+$ such that 
\begin{itemize}
\item[(i)]  $v_\e \stackrel{2}{ \toup}  v$ and $[v_\e ]_M\stackrel{2}{ \toup}  v_M$ for all $M\in \bbN_+$, along a \rrr{sequence} $\e_k\downarrow 0$;
\item[(ii)] for any $\varphi \in C_c(\bbR^d) $ and $u\in \cG$ it holds
\begin{multline}\label{queen1}
\lim_{M \to \infty}\int dx \,m \int d\cP_0 (\o) v_M(x,\o) \varphi (x) u(\o)\\
=
\int dx\,  m \int d \cP_0(\o) v(x,\o) \varphi(x) u(\o)\,.
\end{multline}
\end{itemize}
\end{Lemma}
\begin{proof}
Without loss, we assume that 
 $\| v_\e\|_{L^2(\mu^\e_{\tilde \o})}\leq C_0 <+\infty$ for all $\e$. 
We set $v^\e_M:= [v_\e]_M$. Since  $\| v^\e_M\|_{L^2(\mu^\e_{\tilde \o})}\leq \|v_\e\|_{L^2(\mu^\e_{\tilde \o})}\leq C_0$, Item (i) follows from Lemma \ref{compatto1} and a diagonal procedure.

Below the convergence $\e\da 0$ is understood along the \rrr{sequence} $\{\e_k\}$.
Let us define $F( \bar v, \bar \varphi, \bar u ):= \int dx\,m \int \cP_0(d\o) \bar v (x,\o) \bar \varphi (x) \bar u (\o)$. Then Item  (ii) corresponds to  the limit
\begin{equation}
\label{pavimento}
\lim_{M\to \infty}  F(v_M , \varphi, u )= F(v, \varphi, u)\qquad \forall \varphi \in C_c (\bbR^d)\,,\; \forall u \in \cG\,.
\end{equation}
We fix  functions $\varphi, u$ as in \eqref{pavimento} and set $u_k:=[u]_k$ for all $k \in \bbN_+$. By definition of $
\cG$, we have $u_k \in \cG$ for all $k$ (see Section \ref{topo}).

\begin{Claim}\label{cipolla1}
For each $k,M\in \bbN_+$ it holds
\begin{align}
& | F(v,\varphi, u) - F(v, \varphi, u_k)| \leq C_0 \|\varphi \|_{L^2(mdx)} \|u-u_k\|_{L^2(\cP_0)}\,,\label{ricola1}\\
& | F(v_M,\varphi, u) - F(v_M, \varphi, u_k)| \leq C_0 \|\varphi \|_{L^2(mdx)} \|u-u_k\|_{L^2(\cP_0)}\,.\label{ricola2}
\end{align}
\end{Claim}
\begin{proof}[Proof of Claim \ref{cipolla1}]
By Schwarz inequality 
\[ | F(v,\varphi, u) - F(v, \varphi, u_k)|
 \leq\|   v\| _{L^2 (m dx \times \cP_0) }   \|   \varphi (u-u_k) \| _{L^2 (m dx \times \cP_0) }  \,.
\]
To get \eqref{ricola1} it is then enough to apply Lemma \ref{alba} (or Lemma \ref{compatto1}) to bound $\|   v\| _{L^2 (m dx \times \cP_0) } $ by $ C_0$. The proof of \eqref{ricola2} is identical.
\end{proof}

\begin{Claim}\label{cipolla2}
For each $k,M\in \bbN_+$ it holds 
\begin{equation}\label{ricola3}
| F(v, \varphi, u_k)- F(v_M, \varphi, u_k)| \leq (k/M) \|\varphi \|_\infty C_0^2\,.
\end{equation}
\end{Claim}
\begin{proof}[Proof of Claim \ref{cipolla2}]
 We note that $(v_\e -v^\e_M)(x)=0$ if $|v_\e(x)|\leq M$. Hence we can bound
\begin{equation}\label{viola} |v_\e - v^\e_M|(x) = |v_\e - v^\e_M|(x) \mathds{1}_{\{ |v_\e(x)|> M\}}\leq |v_\e - v^\e_M|(x)  \frac{ |v_\e(x)|}{M} \leq \frac{ v_\e(x)^2}{M} \,.
\end{equation}
We observe that $F(v, \varphi, u_k)=\lim _{\e\da 0} \int d \mu ^\e_{\tilde \o}(x) v_\e(x) \varphi (x) u_k (\theta_{x/\e}\tilde \o) $,
   since $u_k\in \cG$ and $v_\e \stackrel{2}{\toup} v$.  A similar representation holds for $F(v_M, \varphi, u_k)$. As a consequence, and using \eqref{viola},  we get
\begin{align*}
&| F(v, \varphi, u_k)- F(v_M, \varphi, u_k)| \leq \varlimsup_{\e\da0} \int d \mu ^\e_{\tilde \o}(x)
\bigl|
(v_\e - v^\e_M)(x) \varphi(x) u_k (\theta_{x/\e}\tilde \o) \bigr|
\\
& \leq (k/M) \|\varphi\|_\infty \varlimsup_{\e\da 0}  \int d \mu ^\e_{\tilde \o}(x)  v_\e(x)^2\leq (k/M) \|\varphi \|_\infty C_0^2\,.
\qedhere
\end{align*}
\end{proof}
We can finally conclude the proof of \eqref{pavimento}. Given $\varphi\in C_c(\bbR^d)$ and $u\in \cG$, by applying Claims \ref{cipolla1} and \ref{cipolla2},  we can bound
\begin{multline}
 |F(v_M , \varphi, u )- F(v, \varphi, u)| \leq |F(v_M , \varphi, u )- F(v_M, \varphi, u_k)|\\
+|F(v_M , \varphi, u_k )- F(v, \varphi, u_k)|+|F(v , \varphi, u_k )- F(v, \varphi, u)| \\\leq
2 C_0 \|\varphi \|_{L^2(mdx)} \|u-u_k\|_{L^2(\cP_0)} +(k/M) \|\varphi \|_\infty C_0^2\,.
\end{multline}
\eqref{pavimento} then follows by taking first the limit $M\to \infty$ and afterwards the limit $k\to \infty$, and using   that $\lim_{k\to \infty} \|u-u_k\|_{L^2(\cP_0)}=0$.
\end{proof}


\section{Structure of the 2-scale weak limit  of a  bounded family in $H^1_{\tilde \o,\e}$: part I} \label{sec_risso} 
It is simple to check the following Leibniz rule for the microscopic gradient:
 \begin{equation}\label{leibniz} \nabla _\e  (fg)(x,z)
   =\nabla _\e  f (x, z ) g (x )+ f (x+\e z ) \nabla _\e g  ( x, z)
\end{equation}
where $f,g: \e \hat \o \to \bbR$.

The following  Proposition  \ref{risso}  is \ovo{related to} \cite[Lemma 5.3]{ZP}. In the proof we will use  a cut-off procedure  based on Lemma \ref{lemma1} (see Remark \ref{eccolo1} below).

\begin{Proposition}\label{risso} Let $\tilde \o \in \O_{\rm typ}$. Let $\{v_\e\} $ be a  family of functions $v_\e \in H^1_{\tilde \o, \e} $ satisfying 
\begin{equation}\label{istria}
\limsup _{\e \da 0} \| v_\e\|_{L^2(\mu_{\tilde \o}^\e) } <+\infty\,,\qquad 
\limsup_{\e \da 0}  \|\nabla_\e v_\e\| _{L^2(\nu^\e_{\tilde \o})}  <+\infty \,.
\end{equation} Then, along a \rrr{sequence} $\e_k\downarrow 0$, we have that $v_\e \stackrel{2}{\toup} v$, where $v\in  L^2( m dx\times \cP_0)$ does not depend on $\o$, i.e. for $dx$--a.e. $x\in \bbR^d$  the function $\o \mapsto v(x,\o)$ is constant $\cP_0$--a.s..
\end{Proposition}
\begin{proof} 
Due to Lemma \ref{compatto1} we have that $v_\e \stackrel{2}{\toup} v \in  L^2( m dx\times \cP_0)$ along a sequence $\e_k\downarrow 0$.
Recall the definition of the functional sets $\cG_1, \cH_1$ given in Section \ref{topo}.
We  claim that $\forall \varphi \in C^1_c (\bbR^d)$ and $\forall \psi \in \cG_1$ it holds 
\begin{equation}\label{chiavetta}
\int dx\, m \int \ovo{d}\cP_0(\o) v (x,\o) \varphi (x) \psi(\o)=0\,.
\end{equation}
Before proving our claim, let us explain how it leads to the thesis. Since $\varphi $ varies among $C^1_c(\bbR^d)$ while $\psi$ varies in a countable set,  \eqref{chiavetta} implies that,  $dx$--a.e.,  $ \int \ovo{d}\cP_0(\o) v (x,\o)\psi(\o)
=0$ for any $\psi \in \cG_1$. We conclude that,  $dx$--a.e., $v(x,\cdot)$ is orthogonal in $L^2(\cP_0)$ to $\{ w \in L^2(\cP_0)\,:\, \bbE_0[ w]=0\}$ (due to the density of $\cG_1$), which is equivalent to the fact that $v(x,\o)= \bbE_0[ v(x, \cdot)]$ for $\cP_0$--a.a. $\o$.

It now remains to prove \eqref{chiavetta}. Since $\tilde\o \in \O_{\rm typ}$ and due to \eqref{istria}, at cost to refine the sequence $\{\e_k\}$, Items (i) and (ii) of Lemma \ref{lemma1} hold (we keep the same notation of Lemma \ref{lemma1}). Hence,  in oder to prove \eqref{chiavetta}, it is enough to prove for any $M$ that, given $\varphi \in C^1_c (\bbR^d)$ and $ \psi \in \cG_1$, 
\begin{equation}\label{chiavettaM}
\int dx  \,m \int \ovo{d}\cP_0(\o) v_M (x,\o) \varphi (x) \psi(\o)=0\,.
\end{equation}
We write $v^\e_M:= [v_\e]_M$. Since $|\nabla _\e v^\e_M|\leq |\nabla_\e v_\e| $ (cf. \eqref{r101}),  by Lemma \ref{compatto2} (using \eqref{istria}) and a diagonal procedure,  at cost to refine the \rrr{sequence} $\{\e_k\}$ we have for any $M$  that 
$\nabla_\e v^\e_M \stackrel{2}{\toup} w_M \in L^2( mdx \times \nu)$, along the \rrr{sequence} $\{\e_k\}$. In what follows, we understand that the parameter  $\e$ varies in $\{\e_k\}$. Note in particular that, by \eqref{rabarbaro} and since  $\tilde \o \in \O_{\rm typ}$ and $\psi \in \cG_1\subset \cG$,
\begin{equation}\label{nord1}
\text{l.h.s. of }\eqref{chiavettaM}= \lim_{\e\da 0} \int d \mu^\e_{\tilde \o} (x) v_M^\e(x) \varphi (x) \psi( \theta_{x/\e}\tilde \o) \,.
\end{equation}
Let us write $\psi= g_b$ with $b \in \cH_1$ (recall \eqref{lupetto}).  By Lemma \ref{lunetta},  since $\tilde \o \in \vvv{\O_{\rm typ}\subset \cA_1[b]\cap \cA_1[\tilde b]}$,   the r.h.s. of \eqref{nord1} equals the limit as $\e\da 0$ of 
\be \label{nord2}
 -\e \int d\nu ^\e_{\tilde \o} (x,z) \nabla_\e( v_M^\e \varphi ) (x,z) b (\theta_{x/\e} \tilde{\o}, z)= -\e C_1(\e)\ovo{-} \e C_2(\e)\,,
\en
where (due to \eqref{leibniz})
\begin{align*}
& C_1(\e):= \int d\nu ^\e_{\tilde \o} (x,z) \nabla_\e v_M^\e (x,z) \varphi(\e x) b (\theta_{x/\e} \tilde{\o}, z)\,,\\
&  C_2(\e):=  \int d\nu ^\e_{\tilde \o} (x,z) v_M^\e ( x+\e z  ) \nabla_\e \varphi( x,z) b (\theta_{x/\e} \tilde{\o}, z)   \,.
\end{align*}

Due to \eqref{nord1} and \eqref{nord2}, to get \eqref{chiavettaM} we only need to show that $\lim_{\e \da 0} \e C_1(\e)=0$ and $\lim_{\e \da 0} \e C_2(\e)=0$.
Since $\nabla_\e v^\e_M \stackrel{2}{\toup} w_M$ and $b \in \cH_1$, by \eqref{yelena} we have that 
\begin{equation}\label{nord3}
\lim_{\e \da0} C_1(\e) = \int dx \, m \int d\nu (\o, z) w_M (x, \o,z)  \varphi ( x) b (\o, z)\,,
\en
which is finite, thus implying that $\lim_{\e\da 0} \e C_1(\e)=0$.

 We move to $C_2(\e)$.
Let $\ell$ be such that $\varphi (x)=0$ if $|x| \geq \ell$. Fix   $\phi \in C_c(\bbR^d)$ with values in $[0,1]$, such that $ \phi(x)=1$ for $|x| \leq \ell$ and $\phi(x)=0$ for $|x| \geq \ell+1$. Since $\nabla _\e \varphi(x,z)=0$ if $|x| \geq \ell$ and $|x+\e z|\geq \ell$, by   the mean value theorem we conclude that 
\be\label{paradiso}
\bigl | \nabla _\e \varphi(x,z) \bigr | \leq \| \nabla \varphi \|_\infty |z| \bigl( \phi(x)+ \phi(x+\e z) \bigr) \,.
\en
We apply the above bound and Schwarz inequality  to $C_2(\e)$ getting
\be\label{xar}
\begin{split}
|C_2(\e)| & \leq M \| \nabla \varphi \|_\infty \int d \nu ^\e _{\tilde \o} (x,z) |z|\,|  b (\theta_{x/\e}\tilde \o, z) | \bigl( \phi(x)+ \phi(x+\e z) \bigr) \\
& \leq M \|\nabla \varphi \|_\infty  A_1(\e) ^{1/2} A_2 (\e) ^{1/2}\,,
\end{split}
\en
where  (see below for explanations)  
\begin{align*}
 A_1(\e):& =\int  \nu ^\e _{\tilde \o} (x,z) |z| ^2  \bigl( \phi(x)+ \phi(x+\e z)   \bigr)= 2\int  \nu ^\e _{\tilde \o} (x,z) |z| ^2   \phi(x)^2\,,\\
 A_2(\e):&=\int \nu ^\e _{\tilde \o} (x,z)   b (\theta_{x/\e}\tilde \o, z)^2  \bigl( \phi(x) + \phi(x+\e z)  \bigr)\\& = 2  \int \nu ^\e _{\tilde \o} (x,z)( b^2+ \tilde b^2 ) (\theta_{x/\e}\tilde \o, z) \phi(x)^2\,.
\end{align*}
To get the second identities in the above formulas for $A_1(\e)$ and $A_2(\e)$ we have applied 
 Lemma \ref{gattonaZ}--(i) to the forms $(\o,z)\mapsto |z|^2$ and  $(\o, z)\mapsto  b ^2(\o,z)$.

We now write  
\[ A_1(\e)= 2 \int d \mu ^\e _{\tilde \o}(x) \l_2( \theta_{x/\e}\tilde \o) \phi(x)^2\,, \;A_2(\e)=  2 \int d \mu ^\e _{\tilde \o}(x)
( \widehat{ \tilde b^2}+ \widehat{ b^2} ) (\theta_{x/\e}\tilde \o) \phi(x)^2\,.\]
For the second identity above we used that $\tilde \o \in \vvv{\O_{\rm typ}
\in \cA_1[\tilde b^2]\cap \cA_1[b^2]}
$.
 By using respectively that 
 \vvv{$\o \in  \O_{\rm typ}\subset \tilde A \cap\cA[\l_2]$ with $ A:=\{\l_2<+\infty\}$} (recall Definition \ref{defA}) and that  \vvv{$\o \in  \O_{\rm typ}\subset     \cA[ \widehat{ \tilde b^2} ]\cap  \cA[ \widehat{ b^2} ] \cap  \cA_1[\tilde b^2]\cap \cA_1[b^2]$}, we conclude that $A_1 (\e)$ and $ A_2(\e)$ have finite limits as $\e\da 0$, thus implying (cf. \eqref{xar}) that $\lim_{\e\da 0} \e C_2(\e)=0$. This concludes the proof of \eqref{chiavettaM}.
\end{proof}

\begin{Remark}\label{eccolo1}
We stress that the cut-off  method developed in Lemma \ref{lemma1} has been essential to get \eqref{xar} and go on. If one tries to prove \eqref{chiavettaM} with $v$ instead of $v_M$, then one would be stopped when trying to control $C_2(\e)$ as $\e\downarrow 0$. \end{Remark}


\section{Cut-off for gradients $\nabla_\e v_\e$} \label{cut-off2}


\begin{Lemma}\label{lemma2}
Let $\tilde \o \in \O_{\rm typ}$ and let $\{v_\e\}$ be a family of functions  with    $v_\e\in H^1_{\tilde \o, \e} $, satisfying \eqref{istria}.
 Then  there exist functions $w, w_M\in L^2 (m dx \times \nu) $, with $M$ varying in $\bbN_+$, such that 
\begin{itemize}
\item[(i)] $\nabla_\e  v_\e \stackrel{2}{ \toup}  w$ and $\nabla_\e[ v_\e ]_M\stackrel{2}{ \toup}  w_M$ for all $M\in \bbN_+$, along a  \rrr{sequence} $\e_k\downarrow 0$;
\item[(ii)] for any $\varphi \in C^1_c(\bbR^d) $ and $b\in \cH$ it holds
\begin{multline}\label{queen2}
\lim_{M \to \infty}\int dx \,m \int d\nu  (\o,z) w_M(x,\o,z) \varphi (x) b(\o,z)\\
=
\int dx\,  m \int d \nu (\o,z) w(x,\o,z) \varphi(x) b(\o,z)\,.
\end{multline}
\end{itemize}
\end{Lemma}
\begin{proof}
At cost to restrict to $\e$ small enough, we can assume that $ \| v_\e\|_{L^2(\mu_{\tilde \o}^\e) } \leq C_0$ and 
$ \|\nabla _\e v_\e\| _{L^2(\nu^\e_{\tilde \o})} \leq C_0$ for some $C_0<+\infty$ and all $\e>0$. Due to \eqref{r101}, the same holds respectively for $v^\e_M$ and $\nabla_\e v^\e_M$, for all $M\in \bbN_+$, where we have set $v^\e_M:= [ v_\e]_M$. In particular, by a diagonal procedure, due to Lemmas \ref{compatto1} and \ref{compatto2} along a sequence $\{\e_k\}$  we have  that $v^\e_M \stackrel{2}{\toup} v_M$,  $v_\e  \stackrel{2}{\toup} v$, $\nabla_\e v^\e_M \stackrel{2}{\toup} w_M$ and $\nabla_\e v_\e \stackrel{2}{\toup} w$,  where $v_M,v\in L^2( m dx \times \cP_0)$, $w_M,w\in L^2( m dx \times \nu)$, simultaneously for all $M\in \bbN_+$. This proves in particular Item (i). We point out that we are not claiming that $v_M=[v]_M$, $w_M=[w]_M$. Moreover, from now on we restrict to $\e$ belonging to the above special \rrr{sequence} without further mention.

\smallskip

We prove Item (ii). By extending the diagonal procedure we can assume that along the sequence
$\{\e_k\}$  it holds $|v_\e|\stackrel{2}{\toup}  \tilde v$, as $ \| \,|v_\e|\,\|_{L^2(\mu_{\tilde \o}^\e) } = \| v_\e \|_{L^2(\mu_{\tilde \o}^\e) } \leq C_0$.
We set $H(\bar w, \bar \varphi, \bar b):= \int dx m \int d\nu(\o,z) \bar w (x, \o, z) \bar \varphi (x) \bar b (\o,z)$. Then \eqref{queen2} corresponds to the limit $\lim _{M\to \infty} H(w_M, \varphi, b)=  H(w, \varphi, b)$. Here and below $b \in \cH$ and $\varphi \in C^1_c(\bbR^d)$. Recall that $b_k:=[b]_k \in \cH$ for any $k \in \bbN_+$ (see Section \ref{topo}).

\smallskip

\begin{Claim}\label{utti1}
For each $k,M\in \bbN_+$ it holds
\begin{align}
& | H(w,\varphi, b) - H(w, \varphi, b_k)| \leq C_0 \|\varphi \|_{L^2(mdx)} \|b-b_k\|_{L^2(\nu )}\,,\label{leoncino1}\\
& | H(w_M,\varphi, b) - H(w_M, \varphi, b_k)| \leq C_0 \|\varphi \|_{L^2(mdx)} \|b-b_k\|_{L^2(\nu)}\,.\label{leoncino2}
\end{align}
\end{Claim}
We omit the proof of the above claim since it can be obtained by reasoning exactly as in the proof of Claim \ref{cipolla1}.

\begin{Claim}\label{utti2} For any $k \in \bbN_+$, it holds
\be
\lim_{M\uparrow \infty} | H(w, \varphi, b_k)- H( w_M, \varphi, b_k) | 
\leq k C_* \bigl( M^{-1/2} + \ovo{\bbE_0}[\l_0 \mathds{1}_{\{\l_0\geq \sqrt{M}\}}]\bigr)^{1/2}
\,,
\en
where $C_*$ is a constant depending only on $C_0$ and $\varphi$.
\end{Claim}
\begin{proof}[Proof of Claim \ref{utti2}]
We note that $\nabla_\e v_\e (x,z) = \nabla_\e v_M^\e(x,z)$ if $|v_\e(x)|\leq M$ and $|v_\e(x+\e z)|\leq M$. Moreover, by \eqref{r102}, we have  $ | \nabla _\e v_\e - \nabla_\e v_M^\e|\leq | \nabla _\e v_\e|$. Hence we can bound
\begin{equation}
\bigl| \nabla_\e v_\e - \nabla_\e v^\e_M\bigr| (x,z) \leq | \nabla_\e v_\e|(x,z)\bigl( \mathds{1} _{\{ |v_\e(x)|\geq M\}} + \mathds{1} _{\{ |v_\e(x+\e z) |\geq M\} }\bigr)\,.
\end{equation}
Due to the above bound  we can estimate 
 \begin{equation}\label{mat}
\begin{split}
& | H(w, \varphi, b_k)- H(w_M, \varphi, b_k) | \\
&= \bigl| 
\lim_{\e \da 0} \int d \nu^\e_{\tilde \o} (x,z) \bigl( \nabla_\e v_\e - \nabla_\e v^\e_M \bigr) (x,z) \varphi(x) b_k( \theta_{x/\e} \tilde \o, z) \bigr|\\
&\leq k\, \varlimsup_{\e \da 0} \int d \nu^\e_{\tilde \o} (x,z) | \nabla_\e v_\e|(x,z)\bigl( \mathds{1} _{\{ |v_\e(x)|\geq M\}}+ \mathds{1} _{\{ |v_\e(x+\e z) |\geq M\} }\bigr) |\varphi(x) |\,.
\end{split}
\end{equation}
Note that the identity in \eqref{mat}  follows  from \eqref{yelena} since $b_k \in \cH$ (recall that  $\tilde \o \in \O_{\rm typ}$,  $\nabla_\e v^\e_M \stackrel{2}{\toup} w_M$,  $\nabla_\e v_\e \stackrel{2}{\toup} w$).
By Schwarz inequality  we have 
\be\label{mat1}
 \int d \nu^\e_{\tilde \o} (x,z) | \nabla_\e v_\e|(x,z) \mathds{1} _{\{ |v_\e(x)|\geq M\}}  |\varphi(x) |\leq C_0 A(\e)^{1/2}\,,
\en
where 
\[A(\e): = \int  d \nu^\e_{\tilde \o} (x,z)  \mathds{1} _{\{ |v_\e(x)|\geq M\}} \varphi(x) ^2=
\int d \mu ^\e _{\tilde \o} (x)\mathds{1} _{\{ |v_\e(x)|\geq M\}} \varphi (x) ^2 \l_0 (\theta_{x/\e} \tilde \o)
\,.
\]
As $\mathds{1} _{\{ |v_\e(x)|\geq M \} } \leq |v_\e(x)|/M$ we can bound $A(\e)\leq A_1(\e)+A_2(\e)$ where 
\begin{align*}
& A_1(\e ):=
\frac{1}{M}\int d \mu ^\e _{\tilde \o} (x) |v_\e(x)|   \varphi (x) ^2 (\l_0\wedge \sqrt{M}) (\theta_{x/\e} \tilde \o) \,,\\
& A_2(\e):=
\int d \mu ^\e _{\tilde \o} (x)\varphi (x) ^2 (\l_0\mathds{1}_{\{\l_0 > \sqrt{M}\}}) (\theta_{x/\e} \tilde \o)\,.
\end{align*}
As $\l_0 \wedge \sqrt{M} \in \cG$ and $|v_\e| \stackrel{2}{\toup} \tilde v \in L^2( mdx \otimes \cP_0)$, we have 
$\lim_{\e \da 0} A_1(\e)= M^{-1} \int dx \, m \int d \cP_0(\o)\tilde v(x, \o )\varphi(x)^2 (\l_0\wedge \sqrt{M}) (\o)$. 
We recall (see Lemma \ref{compatto1}) that  $\| \tilde v \|_{L^2( m dx \otimes \cP_0)}\leq C_0$.  Hence by Schwarz inequality we have
\be
 \Big|\int dx \, m \int d \cP_0(\o)\tilde v(x, \o )\varphi(x)^2 (\l_0\wedge \sqrt{M})(\o)\Big |\leq \ovo{C_0 M^{1/2}}    \| \varphi^2 \|_{L^2(m dx) }\,.
\en
As a consequence  we have  $\lim_{\e \da 0}  A_1(\e )\leq  \ovo{(C_0 /M^{1/2})}    \| \varphi^2 \|_{L^2(m dx) }$.
As $\tilde \o \in \O_{\rm typ}\subset \cA[\l_0\mathds{1}_{\{\l_0 > \sqrt{M}\}}]\cap \tilde A $ with $A=\{\l_0<+\infty\}$, we have $\lim_{\e \da 0} A_2(\e)=\int dx\, m \varphi(x)^2 \bbE_0[ \l_0\mathds{1}_{\{\l_0 > \sqrt{M}\}}]
$.
Since $A(\e)\leq A_1(\e)+A_2(\e)$, we then get 
\be\label{rachel}
\limsup_{\e \da 0} A(\e)\leq \ovo{ (C_0 /M^{1/2}) }   \| \varphi ^2\|_{L^2(m dx) }+\int dx\, m \varphi(x)^2 \bbE_0[ \l_0\mathds{1}_{\{\l_0 > \sqrt{M}\}}]\,.
\en
Reasoning as above we have 
\be\label{mat2}
 \int d \nu^\e_{\tilde \o} (x,z) | \nabla_\e v_\e|(x,z) \mathds{1} _{\{ |v_\e(x+ \e z )|\geq M\}}  |\varphi(x) |\leq C_0 B(\e)^{1/2}\,,
\en
where (applying also  \eqref{micio2} to the form  $(\o,z) \mapsto 1$) 
\begin{equation*}
B(\e): = \int  d \nu^\e_{\tilde \o} (x,z)  \mathds{1} _{\{ |v_\e(x+ \e z )|\geq M\}} \varphi(x) ^2 = \int  d \nu^\e_{\tilde \o} (x,z)  \mathds{1} _{\{ |v_\e(x)|\geq M\}} \varphi(x+\e z) ^2  \,.
\end{equation*}
We want to replace in the last expression the term  $\varphi(x+\e z) ^2$ with $\varphi(x)^2$ (note that this replacement would produce $A(\e)$). To estimate the error we observe that 
by  \eqref{paradiso} (we use the same notation here) we can bound 
\be
| \varphi(x+\e z) ^2 - \varphi (x)^2| \leq  C \e  |z| \bigl( \phi(x)+ \phi(x+\e z) \bigr) \,,
\en
where $C=C(\varphi)$. Hence we have  
$B(\e) \leq A(\e)+C \e B_1(\e)+C\e B_2(\e)$, where
\begin{align*}
& B_1(\e):=  \int  d \nu^\e_{\tilde \o} (x,z)  \mathds{1} _{\{ |v_\e(x)|\geq M\}} \phi(x) |z|\leq 
 \int  d \mu^\e_{\tilde \o} (x )   \phi (x) \l_1( \theta _{x/\e} \tilde\o) \,,\\
& B_2(\e):= \int  d \nu^\e_{\tilde \o} (x,z)  \mathds{1} _{\{ |v_\e(x)|\geq M\}}  |z|
\phi(x+\e z) \leq  
\int  d \nu^\e_{\tilde \o} (x,z)   |z| \phi(x+\e z)\,.
\end{align*}
Note that by \eqref{micio2} we have 
\[
\int  d \nu^\e_{\tilde \o} (x,z)   |z| \phi(x+\e z)= \int  d \nu^\e_{\tilde \o} (x,z)   |z| \phi(x)= \int  d \mu^\e_{\tilde \o} (x )   \phi (x) \l_1( \theta _{x/\e} \tilde\o)\,.\] 
Hence $B_1(\e)+B_2(\e) \leq 2\int  d \mu^\e_{\tilde \o} (x )   \phi (x) \l_1( \theta _{x/\e} \tilde\o)$.
As $\phi\in C_c(\bbR^d)$ and $\tilde \o \in \O_{\rm typ}\subset \cA[\l_1] \cap \tilde A$ with $A:=\{\l_1<+\infty\}$, we conclude that $\limsup_{\e \da 0} (B_1(\e)+B_2(\e))<+\infty$. By using that $B(\e) \leq A(\e)+C \e B_1(\e)+C\e B_2(\e)$, we get  that $\limsup_{\e \da 0} B(\e) \leq \limsup_{\e \da 0} A(\e)$ and the latter has been estimated in \eqref{rachel}. At this point, using also  \eqref{mat}, \eqref{mat1} and \eqref{mat2}, we get the claim.
\end{proof}
We can finally derive \eqref{queen2}, i.e. that $\lim _{M\to \infty} H(w_M, \varphi, b)=  H(w, \varphi, b)$.
By using Claims \ref{utti1} and \ref{utti2} we have 
\begin{equation*}
\begin{split}
| H(w_M, \varphi , b)- H(w, \varphi, b)| \leq |H( w_M, \varphi, b)- H( w_M , \varphi, b_k)| +\\
| H(w_M, \varphi, b_k)- H( w, \varphi, b_k)| + | H(w, \varphi, b_k)- H( w, \varphi, b)|\\
\leq C_0 C(\varphi) \| b-b_k\|_{L^2(\nu) }+ 
k C_* \big( M^{-1/2} + \ovo{\bbE_0}[\l_0 \mathds{1}_{\{\l_0\geq \sqrt{M}\}}]\big)^{1/2}\,.
\end{split}
\end{equation*}
At this point it is enough to take first the limit $M\to \infty$ and afterwards the limit $k\to \infty$ and to use that  $\lim_{k \to \infty}\| b-b_k\|_{L^2(\nu) }=0$.
\end{proof}


\section{Structure of the 2-scale weak limit  of a  bounded family in $H^1_{\tilde \o,\e}$: part II} \label{sec_oro}
We point out that the next result is the analogous of \cite[Lemma 5.4]{ZP}. The proof relies also on a cut-off procedure based on Lemma \ref{lemma1} and   Lemma \ref{lemma2} (see also Remark \ref{eccolo2}  below). Recall Definition \ref{degenerare}.
\begin{Proposition}\label{oro}
Let $\tilde \o \in \O_{\rm typ}$
 and let $\{v_\e\} $ be a  family of functions $v_\e \in H^1_{\tilde \o, \e} $  satisfying \eqref{istria}.
Then, along a \rrr{sequence} $ \e_k\downarrow 0$, we have:
\begin{itemize}
\item[(i)]
$\ovo{L^2(\mu^\e_{\tilde \o})\ni}v_\e \stackrel{2}{\toup} v\ovo{\in  L^2(  m dx\times \cP_0)}$, where \ovo{$v$} does not depend on $\o$.  Writing $v$ simply as $v(x)$ we have that   $v\in H^1 _*( m dx )$;
\item[(ii)] $\ovo{L^2(\nu^\e_{\tilde \o})\ni} \nabla v_\e  (x,z) \stackrel{2}{\toup}  \nabla_*  v (x) \cdot z + v_1 (x,\o,z) \ovo{\in  L^2(  m dx\times \nu)}$,
where $v_1\in L^2\bigl( \bbR^d, L^2_{\rm pot} (\nu)\bigr )$.
\end{itemize}
\end{Proposition}


The property $v_1\in L^2\bigl( \bbR^d, L^2_{\rm pot} (\nu)\bigr )$ means that 
for $dx$--almost every $x$ in $\bbR^d$ the map $(\o,z)\mapsto v_1(x, \o, z) $ is a potential form, hence in  $L^2_{\rm pot} (\nu)$, moreover the map  $\bbR^d \ni x \to v_1(x, \cdot, \cdot) \in L^2_{\rm pot} (\nu)$ is measurable and 
\begin{equation}
\int  dx \| v_1(x, \cdot, \cdot)\|_{L_2(\nu) }^2=
\int dx  \int  \nu (\o,z) v_1 (x, \o, z)^2 <+\infty\,.
\end{equation}

\begin{proof}[Proof of Prop.~\ref{oro}] 
At cost to restrict to $\e$ small enough, we can assume that $ \| v_\e\|_{L^2(\mu_{\tilde \o}^\e) } \leq C_0$ and 
$ \|\nabla _\e v_\e\| _{L^2(\nu^\e_{\tilde \o})} \leq C_0$ for some $C_0<+\infty$ and all $\e>0$.  We can assume the same bounds for $v^\e_M:=[v_\e]_M$.  
Along a sequence  $\e_k \downarrow 0$ the $2$-scale convergences in Item (i) of Lemma \ref{lemma1} and in Item (i) of Lemma \ref{lemma2} take place.   We keep here the same notation. In particular, $v_\e \stackrel{2}{ \toup}  v$,  $v^\e _M\stackrel{2}{ \toup}  v_M$, 
 $\nabla_\e  v_\e \stackrel{2}{ \toup}  w$ and $\nabla_\e  v^\e _M\stackrel{2}{ \toup}  w_M$.
 By Lemmas \ref{compatto1} and \ref{compatto2}  the norms $\| v_M\|_{L^2(m dx \times \cP_0)}$, $\| v\|_{L^2(mdx \times \cP_0)}$, $\| w_M\|_{L^2(m dx \times \nu)}$ and $\| w\|_{L^2(m dx \times \nu)}$ are upper bounded by $C_0$. 

Due to Prop.~\ref{risso} $v=v(x)$ and $v_M=v_M(x)$. We claim that 
 for   each  solenoidal form $b  \in L^2_{\rm sol}(\nu) $ and each  function $\varphi \in C^2_c(\bbR^d)$, it holds 
 \begin{equation}\label{kokeshi}
 \int  dx   \varphi(x) \int  d \nu (\o ,z) w(x,\o,z) b(\o, z) =  
  -\int   dx  v(x) \nabla \varphi(x) \cdot \eta_b\,, 
\end{equation}
where  $\eta_b := \int d\nu (\o,z) z b(\o,z)$.
Note that $\eta_b$ is well defined since both $b$ and the map $(\o,z) \mapsto z$ are in $ L^2(\nu)$. Moreover, by  applying Lemma \ref{lemma_siou} with $f$ such that $f(\o, \theta_z \o)=  z b(\o,z) $  \ovo{$\forall z\in \hat \o$} for $\cP_0$--a.a. $\o$ \utto{(use  (A3) to define $f$)}, we get that $\eta_b = -\eta_{\tilde b}$ (cf. Def.~\ref{ometto}).

Before proving \eqref{kokeshi} we show how to conclude the proof of Prop.~\ref{oro} starting with Item (i).
Due to Corollary \ref{jack} for each $i=1,\dots, d_*$ there exists  $b_i \in L^2_{\rm sol}(\nu)$ such that $\eta_{b_i}=\mathfrak{e}_i$.
 Consider the measurable function 
  \begin{equation}
  g_i (x):=  \int   d \nu (\o ,z) w(x,\o,z) b_i(\o, z)\,, \qquad  1\leq i \leq d_*\,.
  \end{equation}
We have that $g_i \in L^2(dx) $ since, by Schwarz inequality, 
\begin{multline}
\int g_i(x) ^2 dx =
   \int dx \left[ \int  d \nu (\o ,z) w(x,\o,z) b_i(\o, z)\right]^2\\
\leq  \|b _i \|^2_{L^2(\nu)}\int dx   \ \int  d \nu (\o ,z) w(x,\o,z) ^2 <\infty\,.
\end{multline}
   Moreover, by \eqref{kokeshi} we have that $\int dx \varphi(x) g_i(x) =  -\int   dx\, v(x) \partial_{\mathfrak{e}_i} \varphi(x) $ for $1\leq i \leq d_*$.  This proves that $v(x)\in H^1_*(mdx)$ and   $\partial_{\mathfrak{e}_i} v(x)= 
 g_i(x)$  for $1\leq i \leq d_*$.
  This concludes the proof of Item (i).
  
  We move to Item (ii) (always assuming 
  \eqref{kokeshi}). By Item (i) and Corollary \ref{jack} implying that $\eta_b \in \text{span}\{\mathfrak{e}_1, \dots, \mathfrak{e}_{d_*}\} $ for all $b\in L^2_{\rm sol}(\nu)$,  we can  replace  the r.h.s. of \eqref{kokeshi}  by $ \int  dx  (\nabla_* v(x) \cdot \eta_b ) \varphi(x)$. Hence   \eqref{kokeshi} can be rewritten as 
 \begin{equation}\label{cocco}
 \int  dx  \varphi(x) \int   d \nu (\o ,z)
 \left[ w(x,\o,z) -\nabla_* v(x)  \cdot z
 \right] b(\o, z) =  
  0\,.
  \end{equation}
By the arbitrariness of $\varphi$ we conclude that $dx$--a.s. 
\begin{equation}\label{criceto13}
\int  d \nu (\o ,z)
 \left[ w(x,\o,z) -\nabla_* v(x)  \cdot z
 \right] b(\o, z) =  
  0\,, \qquad \forall b \in L^2 _{\rm sol}(\nu)\,.
\end{equation}
Let us now show that the map 
$w(x,\o,z) -\nabla_* v(x) \cdot z $ belongs to $L^2( dx, L^2(\nu) )$. Indeed,  we have 
$
\int dx \|w(x, \cdot, \cdot)\|_{L^2(\nu) }^2 =
 \| w\|^2_{ L^2( mdx \times d\nu)}<+\infty$ and also
 \begin{equation}
\int dx \| \nabla_* v(x) \cdot z \|_{L^2(\nu) }^2 \leq \int dx | \nabla_* v(x) |^2  \int d\nu (\o, z) |z|^2 
<\infty\,,
\end{equation}
  by Schwarz inequality
 and since $ \nabla_* v\in L^2(dx)$ and $\bbE_0[\l_2]<\infty$.

As the map 
$w(x,\o,z) -\nabla_*v(x) \cdot z $ belongs to $L^2( dx, L^2(\nu) )$, for $dx$--a.e. $x$ we have that the map 
  $(\o,z) \mapsto w(x,\o,z) -\nabla_* v(x) \cdot z$ belongs to $ L^2(\nu)$ and therefore, by \eqref{criceto13}, to   $ L^2_{\rm pot} (\nu)$.
This concludes the proof of Item (ii).

\smallskip
It remains to prove \eqref{kokeshi}. Here is a roadmap: (i) we reduce  \eqref{kokeshi}  to \eqref{kokeshiM}; (ii) we prove  \eqref{provence}; (iii) 
by \eqref{provence} we reduce 
\eqref{kokeshiM} to \eqref{tramonto}; (iv) we prove \eqref{cuoricino1}; (v) by \eqref{cuoricino1} we reduce \eqref{tramonto} to \eqref{tramontoA}; (vi) we prove \eqref{tramontoA}. 

Since both sides of \eqref{kokeshi} are continuous as functions of $b \in L^2_{\rm sol}(\nu)$, it is enough to prove it for $b\in \cW$ (see Section~\ref{topo}). We apply Lemma \ref{lemma2}--(ii) (recall that $b\in \cW\subset \cH$) to approximate the l.h.s. of \eqref{kokeshi} and Lemma \ref{lemma1}--(ii) with $u:=1\in \cG$ to approximate the r.h.s. of \eqref{kokeshi}. Then to prove \eqref{kokeshi}
it is enough to show that 
  \begin{equation}\label{kokeshiM}
 \int  dx \,m  \varphi(x) \int  d \nu (\o ,z) w_M(x,\o,z) b(\o, z) =  
  -\int   dx \,m v_M(x) \nabla \varphi(x) \cdot \eta_b\,, 
\end{equation}
 for any $\varphi \in C_c^2(\bbR^d)$, $b \in \cW$ and $M\in \bbN_+$. From now on $M$ is fixed.

 \smallskip
  Since $\tilde \o \in \O_{\rm typ}$, $\nabla_\e v^\e_M\stackrel{2}{\toup} w_M$ and $b \in \cW \subset \cH$ we can write (cf.~\eqref{yelena}) 
  \be\label{bruna1}
  \text{l.h.s. of }\eqref{kokeshiM}= \lim _{\e \da 0} \int d \nu ^\e _{\tilde \o}(x,z) \nabla_\e v^\e _M(x,z) \varphi (x) b ( \theta_{x/\e} \tilde \o, z)\,.
  \en
  Since $b \in L^2_{\rm sol}(\nu)$ and \vvv{$\tilde \o \in \O_{\rm typ}\subset \cA_d[b]$} (cf.  Lemmata \ref{tav} and    \ref{lunetta}), we get
  \[
  \int d \nu ^\e _{\tilde \o}(x,z) \nabla_\e ( v^\e  _M \varphi ) (x,z)  b ( \theta_{x/\e} \tilde \o, z)=0\,.
  \]
Using the above identity, \eqref{leibniz}  and finally \eqref{micio3} in Lemma \ref{gattonaZ} as \vvv{$\tilde \o \in \O_{\rm typ}\subset \cA_1[b]\cap \cA_1[\tilde b]
$}, we conclude that 
\be\label{bruna2}
\begin{split}
\int d \nu ^\e _{\tilde \o}(x,z) &\nabla_\e v^\e _M(x,z) \varphi (x) b ( \theta_{x/\e} \tilde \o, z)\\
&=- \int d \nu ^\e _{\tilde \o}(x,z)  v^\e _M(x+ \e z) \nabla_\e \varphi (x,z) b ( \theta_{x/\e} \tilde \o, z)
\\
&  =\int d \nu ^\e _{\tilde \o}(x,z)  v^\e _M(x) \nabla_\e \varphi (x,z) \tilde b ( \theta_{x/\e} \tilde \o, z)\,.
\end{split}
\en
Up to now we have obtained that 
\be\label{alba1}
  \text{l.h.s. of }\eqref{kokeshiM}= \lim _{\e \da 0} \int d \nu ^\e _{\tilde \o}(x,z)  v^\e _M(x) \nabla_\e \varphi (x,z) \tilde b ( \theta_{x/\e} \tilde \o, z)\,.
    \en
We now set $b_k:=[b]_k$,
 $\tilde b_k:= [\,\tilde b \,]_k= \widetilde{[b]_k}= \widetilde{\,b_k\,}$. 
We  want to prove that 
\be\label{provence}
\varlimsup_{k \uparrow \infty} \varlimsup  _{\e\da 0}
 \bigl| \int d \nu ^\e _{\tilde \o}(x,z)  v^\e _M(x) \nabla_\e \varphi (x,z)( \tilde b- \tilde b_k) ( \theta_{x/\e} \tilde \o, z)\bigr|
 =0\,.
\en
To this aim let $\ell$ be such that $\varphi (x)=0$ if $|x| \geq \ell$. Fix   $\phi \in C_c(\bbR^d)$ with values in $[0,1]$, such that $ \phi(x)=1$ for $|x| \leq \ell$ and $\phi(x)=0$ for $|x| \geq \ell+1$. 
  Using \eqref{paradiso} and Schwarz inequality  we can bound
\be \label{acqua}
\begin{split}
& \bigl| \int d \nu ^\e _{\tilde \o}(x,z)  v^\e _M(x) \nabla_\e \varphi (x,z)( \tilde b- \tilde b_k) ( \theta_{x/\e} \tilde \o, z)\bigr|
\\ 
&\leq  M \| \nabla \varphi \|_\infty \int d \nu ^\e _{\tilde \o}(x,z) |z| \bigl( \phi (x) + \phi (x+ \e z) \bigr) | \tilde b- \tilde b_k| ( \theta_{x/\e} \tilde \o, z)\\
& \leq M \| \nabla \varphi \|_\infty [ 2 A(\e) ] ^{1/2} [ B(\e,k)+ C(\e,k)]^{1/2}\,
\end{split}
\en
where (using  \eqref{micio2} in Lemma \ref{gattonaZ}  for $A(\e)$ and $C(\e)$)
\begin{align*}
 A(\e): &= \int d \nu ^\e _{\tilde \o}(x,z) |z|^2 \phi(x)= \int d \nu ^\e _{\tilde \o}(x,z) |z|^2 \phi(x+ \e z)\,,\\
 B(\e,k): & =  \int d \nu ^\e _{\tilde \o}(x,z)  ( \tilde b- \tilde b_k)^2 ( \theta _{x/\e} \tilde \o, z) \phi (x) \,,\\
 C(\e,k):& =  \int d \nu ^\e _{\tilde \o}(x,z)  ( \tilde b- \tilde b_k)^2 ( \theta _{x/\e} \tilde \o, z) \phi (x+\e z)\\
& =  \int d \nu ^\e _{\tilde \o}(x,z)  ( b-   b_k)^2 ( \theta  _{x/\e} \tilde \o, z) \phi (x)\,.
\end{align*}
\vvv{As $\O_{\rm typ}\subset  \tilde A \cap\cA[\l_2]$ with $ A:=\{\l_2<+\infty\}$} (cf. Def.~\ref{defA}),
    $A(\e)= \int d \mu^\e_{\tilde \o} (x) \phi(x) \l_2 ( \theta_{x/\e} \tilde \o) $ has finite limit as $\e\da 0$. 
Hence to get \eqref{provence} we only need to show that $\lim_{k \uparrow  \infty, \e \da 0 } B(\e,k)= \lim _{k \uparrow  \infty, \e \da 0 }C(\e, k)=0$. Setting $ d:=   | \tilde b- \tilde b_k|
$ we can write
$
B(\e,k)= 
\int d \mu^\e _{\tilde \o} (x) \phi(x)
\widehat{d^2}  (\theta_{x/\e} \tilde \o )$ as \vvv{$\o\in  \O_{\rm typ}\subset \cA_1[d^2]$ (recall that $\tilde b, \tilde b_k \in \cH$)}.  As in addition \vvv{$\o\in  \O_{\rm typ}\subset \cA[\widehat{d^2}]$}, we conclude that
$
\lim_{\e \da 0} B(\e,k)= \int dx \, m \phi (x)  \|  \tilde b- \tilde b_k\|_{L^2(\nu)}^2 $.
Similarly we get that $
\lim_{\e \da 0} C(\e,k)= \int dx \, m \phi (x)  \|   b-  b_k\|_{L^2(\nu)}^2 $.  As the above limits go to zero as $k\to \infty$, we get  \eqref{provence}.

Due to \eqref{alba1}, \eqref{provence} and since, by Schwarz inequality, $\lim _{k \to \infty} \eta_{\tilde b_k}= \eta _{\tilde b}= - \eta_b$, to prove \eqref{kokeshiM} we only need to show, for fixed $M,k$,  that 
 \begin{equation}\label{tramonto}
 \lim _{\e \da 0} \int d \nu ^\e _{\tilde \o}(x,z)  v^\e _M(x) \nabla_\e \varphi (x,z) \tilde b _k( \theta_{x/\e} \tilde \o, z)=
 \int   dx \, m v_M(x) \nabla \varphi(x) \cdot \eta_{\tilde b_k}\,.
 \end{equation}
 To prove \eqref{tramonto} we first show that 
 \begin{equation}\label{cuoricino1} \lim_{\e\da 0} \Big |
  \int d \nu ^\e _{\tilde \o}(x,z)  v^\e _M(x) \bigl[ \nabla_\e \varphi (x,z) -\nabla \varphi(x) \cdot z\bigr]\tilde b _k( \theta_{x/\e} \tilde \o, z)
  \Big |=0\,.
 \end{equation}
 Since $\|v^\e _M\|_\infty \leq M $ and $\| \tilde b_k \| _\infty \leq k$, it is enough to show that 
 \be \label{cuoricino2}
 \lim _{\e\da 0}  \int d \nu ^\e _{\tilde \o}(x,z) \bigl| \nabla_\e \varphi (x,z) -\nabla \varphi(x) \cdot z\bigr| =0\,.
 \en
Since $\varphi \in C_c^2(\bbR^d)$,  by Taylor expansion we have $\nabla_\e \varphi (x,z) -\nabla \varphi(x) \cdot z=\frac{1}{2} \sum_{i,j} \partial^2 _{ij} \varphi(\, \z_\e(x,z)\,) z_i z_j \e$, where $\z_\e(x,z)$ is a point between $x$ and $x+\e z$. Moreover we note that 
$\nabla_\e \varphi (x,z) -\nabla \varphi(x) \cdot z=0$ if $|x| \geq \ell$ and $|x+\e z| \geq \ell$. All these observations imply that 
\be \label{mirra}
\bigl| \nabla _\e \varphi (x,z) - \nabla \varphi (x) \cdot z\bigr | \leq \e C(\varphi) |z|^2 \bigl( \phi(x) + \phi(x+\e z) \bigr)\,.
\en
Due to \eqref{micio2}  we can write
\[ \int d \nu ^\e _{\tilde \o}(x,z)  |z|^2  \phi(x+\e z)=
 \int d \nu ^\e _{\tilde \o}(x,z)  |z|^2  \phi(x)= \int d \mu^\e_{\tilde \o} (x) \phi(x) \l_2 ( \theta_{x/\e} \tilde \o)\,.
 \]
\vvv{As $\o \in \O_{\rm typ}\subset  \tilde A \cap\cA[\l_2]$ with $ A:=\{\l_2<+\infty\}$,} we conclude that the  above r.h.s. has a finite limit as $\e \da 0$.  Due to \eqref{mirra}, we get \eqref{cuoricino2} and hence
 \eqref{cuoricino1}.
 
 Having \eqref{cuoricino1}, to get \eqref{tramonto} it is enough to show that 
  \begin{equation}\label{tramontoA}
 \lim _{\e \da 0} \int d \nu ^\e _{\tilde \o}(x,z)  v^\e _M(x) \nabla \varphi (x) \cdot z \tilde b _k( \theta_{x/\e} \tilde \o, z)=
 \int   dx \,m v_M(x) \nabla \varphi(x) \cdot \eta_{\tilde b_k}\,.
 \end{equation}
 To this aim we observe that
 \be\label{favorita1}
\int d \nu^\e_{\tilde \o} (x,z) v_M^\e (x) \partial _i \varphi (x) z_i \tilde{b}_k (\theta_{x/\e} \tilde\o, z)=
\int d \mu^\e_{\tilde \o}(x)  v^\e_M(x) \partial _i \varphi (x) u_k( \theta_{x/\e} \tilde \o)\,,
\en
where $u_k(\o):= 
\int d \hat{\o} (z) r_{0,z}(\o) z_i \tilde{b}_k ( \o, z)$. Note that $u_k \in L^1(\cP_0)$ as $\l_1 \in L^1(\cP_0)$ by \eqref{maracas}. Given $\ell\in \bbN$ we consider $u_k -[u_k]_\ell$ (cf. \eqref{taglio}).
  As  $\tilde \o \in  \O_{\rm typ}\subset \tilde A\cap \cA[ |u_k - [u_k ]_\ell|  ]$ where $A:=\{\l_1<+\infty\}$,  we have 
\be
\begin{split}
 & \Big | \int d \mu^\e_{\tilde \o}(x)  v^\e_M(x) \partial _i \varphi (x) (u_k-[u_k]_\ell)( \theta_{x/\e} \tilde \o)\Big | \\
 &  \leq M  \int d \mu^\e_{\tilde \o}(x)  |\partial _i \varphi (x)| \,|u_k-[u_k]_\ell|( \theta_{x/\e} \tilde \o) \\
 &
 \stackrel{\e \da 0}{\to} M \int dx m |\partial _i \varphi (x)| \bbE_0[ |u_k-[u_k]_\ell|  ]\,.
\end{split}
\en
Note that the last  expectation goes  to zero when $\ell \uparrow \infty$ by dominated convergence as $\l_1\in L^1(\cP_0)$ by \eqref{maracas}.
Since $\tilde \o \in \O_{\rm typ}$, $v^\e_M\stackrel{2}{\toup} v_M$ and   $[u_k]_\ell \in \cG$ (cf. \eqref{nonfamale} and recall that $\tilde b \in \cW$  $\forall b\in \cW$), by \eqref{rabarbaro} we conclude that 
\begin{equation*}
 \begin{split}
& \lim_{\e \da 0}
  \int d \mu^\e_{\tilde \o}(x)  v^\e_M(x) \partial _i \varphi (x) u_k( \theta_{x/\e} \tilde \o) =
  \lim_{\ell \uparrow \infty}\lim_{\e \downarrow 0}\int d \mu^\e_{\tilde \o}(x)  v^\e_M(x) \partial _i \varphi (x)  [u_k]_\ell( \theta_{x/\e} \tilde \o)\\
 &=\lim_{\ell \uparrow \infty} \int dx \, m v_M(x) \partial_i \varphi (x) 
 \bbE_0[ [u_k]_\ell ]
 =\int dx \, m v_M(x) \partial_i \varphi (x) 
 \bbE_0[ u_{k} ]\\
&  =  \int dx \, m v_M(x) \partial_i \varphi (x) 
 (\eta_{\tilde b_k} \cdot e_i)\,.
 \end{split}
 \end{equation*}
  Our target \eqref{tramontoA} then follows from the above equation  and \eqref{favorita1}.
 \end{proof}
 
 \begin{Remark}\label{eccolo2} We stress that, without the cut-off trick allowing to move from \eqref{kokeshi} to \eqref{kokeshiM}, one would be blocked when trying to study the limit in the r.h.s. of \eqref{alba1} with $v^\e_M$ replaced by $v_\e$. 
    \end{Remark}
\section{Proof of Theorem \ref{teo1}}\label{robot}

Without loss of generality, we prove Theorem \ref{teo1}
 with $\l=1$ to simplify the notation.
Due to Prop.~\ref{stemino} we only need to prove Items (i), (ii) and (iii). 
  Some arguments below are taken from \cite{ZP}, others are intrinsic to the possible presence of long jumps.
 We start with two results  (Lemmas \ref{ascoli} and \ref{blocco}) concerning the microscopic gradient $\nabla_\e\varphi$ for $\varphi \in C_c(\bbR^d)$.
 
\begin{Lemma}\label{ascoli} Let $\o \in \O_{\rm typ}$.
 Then 
 $\varlimsup_{\e\da 0}\|\nabla_\e \varphi \|_{L^2(\nu_\o ^\e) }<\infty$ for any $\varphi \in C^1_c (\bbR^d)$.
\end{Lemma}

\begin{proof} 
Let $\phi$ be as in \eqref{paradiso}. By \eqref{paradiso} and since $ \o \in \O_{\rm typ}$ 
 (apply \eqref{micio2} with $b(\o,z):=|z|^2$), we get
\begin{equation*}
\begin{split}
\| \nabla_\e \varphi \|_{L^2(\nu^\e_\o)}^2 &  \leq C(\varphi) \int d \nu^\e_\o (x,z)  |z|^2\bigl( \phi(x)+ \phi(x+\e z) \bigr) \\
&=
2C(\varphi) \int d \nu^\e_\o (x,z)  |z|^2 \phi(x) = 2 C(\varphi) \int d \mu^\e_\o(x)\phi(x) \l_2( \theta_{x/\e} \o)\,.
\end{split}
\end{equation*}
The thesis then follows from Prop. \ref{prop_ergodico}  \vvv{as $\o\in \O_{\rm typ}\subset  \tilde A \cap\cA[\l_2]$ with  $ A:=\{\l_2<+\infty\}$} (recall Def.~\ref{defA}).
\end{proof}

\begin{Lemma}\label{blocco} Given $ \o \in \O_{\rm typ}$ and $\varphi \in C_c^2(\bbR^d)$  it holds 
\begin{equation}\label{football}
\lim  _{\e \da 0}  \int d \nu^\e_{\o}(x,z)  \bigl[\nabla_\e \varphi (x,z) - \nabla  \varphi (x) \cdot z \bigr]^2   =0\,.
\end{equation}
\end{Lemma}
\begin{proof}
 Let $\ell$ be  such that $\varphi (x)=0$ if $|x| \geq \ell$. Let    $\phi \in C_c(\bbR^d)$ be as in   \eqref{paradiso}. 
The  upper bound given by   \eqref{paradiso}  with $\nabla_\e\varphi (x,z)$ replaced by  $\nabla \varphi (x) \cdot z $ is also true. We will apply the above bounds for $|z| \geq  \ell$.  On the other hand, we apply \eqref{mirra} for $|z|<\ell$.
As a result,  we  can bound
\be \label{piano}
 \int d \nu^\e_{ \o}(x,z)  \bigl[\nabla_\e \varphi (x,z) - \nabla  \varphi (x) \cdot z \bigr]^2   \leq 
C(\varphi) [A(\e, \ell )+B(\e, \ell)]\,,
\en
 where (cf. \eqref{micio2} and the Definition \ref{budda} for $h_\ell$) 
\begin{align*}
 A(\e, \ell ):& =   \int d \nu^\e_{ \o}(x,z) |z|^2  (\phi(x)  + \phi (x+ \e z)  )  \mathds{1}_{\{ |z| \geq  \ell\}} \\
& = 2 \int d \nu^\e_{ \o}(x,z) |z|^2  \phi(x) \mathds{1}_{\{ |z| \geq  \ell\}}  =2 \int d \mu ^\e_{ \o} (x)  \phi(x) h_\ell  (\theta_{x/\e} \o) \,,  \\
 B(\e, \ell):& = \e ^2 \ell^4 \int d \nu^\e_{ \o}(x,z) (\phi(x)  + \phi (x+ \e z)  )  \\
& =2 \e^2  \ell^4 \int d \nu^\e_{ \o}(x,z)  \phi(x)  = 2  \e^2 \ell^4 \int d \mu ^\e_{ \o} (x)  \phi(x) \l_0 (\theta_{x/\e} \o)
  \,.
\end{align*}
 Since  $\o\in \O_{\rm typ}\subset \left(\cap _{\ell\in \bbN} \cA[h_\ell ]\right) \cap \tilde A$ with $A:=\{\l_2<+\infty\}$ (recall Def.~\ref{defA}), we get that 
     $\lim_{\e \da 0}  \int d \mu ^\e_{ \o} (x)  \phi(x) h_\ell  (\theta_{x/\e} \o) = \int dx\,m \phi(x) \bbE_0[h_\ell]$.   By dominated convergence
    and (A7) we then get  that $\lim_{\ell\uparrow  \infty,\e \da 0} A(\e, \ell)=0$.
As \vvv{$\o \in \O_{\rm typ} \subset \tilde A  \cap \cA[\l_0]$ with $A=\{ \l_0<+\infty\}$},  the integral $ \int d \mu ^\e_{ \o} (x)  \phi(x) \l_0 (\theta_{x/\e} \o)$ converges to $ \int dx\,m \phi(x) \bbE_0[\l_0]$ as $\e \da 0$. As a consequence,  $\lim_{\e \da 0} B(\e, \ell)=0$. Coming back to \eqref{piano} we finally get \eqref{football}.
\end{proof}

\smallskip

From now on we denote by $\tilde \o$ the environment in $\O_{\rm typ}$ for which we want to prove Items (i), (ii) and (iii) of Theorem \ref{teo1}. 

\smallskip

\noindent
$\bullet$ {\bf Convergence of solutions}. We start by proving Item (i).

\smallskip

We consider \eqref{limite1}.  We recall that the  weak solution $u_\e\ovo{\in H^{1,{\rm f}}_{{\tilde \o}, \e}}$ satisfies (cf. \eqref{salvoep})
\begin{equation}\label{salvoepino}
\frac{1}{2} \la \nabla_\e v, \nabla _\e u_\e \ra _{\nu_{\tilde \o} ^\e}+ \la v, u_\e \ra _{\mu_{\tilde \o}^\e}= \la v, \ovo{f_\e} \ra_{\mu_{\tilde \o}^\e}\qquad \forall v \in \ovo{H^{1,{\rm f}}_{{\tilde \o}, \e}}\,.
\end{equation} 
 Due to \eqref{salvoepino} with $v:=u_\e$  we get that $\|u_\e  \|_{L^2(\mu_{\tilde \o}^\e)}^2 \leq \la u_\e , f_\e\ra_{\mu_{\tilde \o}^\e}$ and therefore $ \|u_\e\| _{L^2(\mu^\e_{\tilde \o})}\leq \|f_\e\| _{L^2(\mu^\e_{\tilde \o})}$ by Schwarz inequality. Hence, by  \eqref{salvoepino}  it holds   
$ 
\frac{1}{2} \| \nabla_\e u_\e  \| _{L^2(\nu^\e_{\tilde \o})} ^2 \leq \| f_\e \| _{ L^2(\mu^\e_{\tilde \o})}^2$.
Since $f_\e \toup f$, the family $\{f_\e\}$ is bounded and therefore there exists $C>0$ such that, for $\e$ small enough as we assume below, 
\begin{equation}\label{salvezza}
\|u_\e  \|_{L^2(\mu_{\tilde \o}^\e)} \leq C \,, \qquad  \| \nabla_\e u_\e  \| _{L^2(\nu^\e_{\tilde \o})}  \leq C \,.
\end{equation}
Due to \eqref{salvezza} and  by Proposition \ref{oro}, 
 along a \rrr{sequence} $\e_k\downarrow 0$ we have:
\begin{itemize}
\item[(i)]
$u_\e \stackrel{2}{\toup} u$, where $u$  is of the form $u=u(x)$ and  $u\in H^1 _*(m dx )$;
\item[(ii)] $\nabla _\e u_\e  (x,z) \stackrel{2}{\toup} w(x,\o,z):=  \nabla_*  u  (x)\cdot z + u_1(x,\o,z)$,  
$u_1\in  L^2\bigl( \bbR^d, L^2_{\rm pot} (\nu)\bigr )$.
\end{itemize}
Below,  convergence  for  $\e\downarrow 0$ is understood along the \rrr{sequence} $\{\e_k\}$.
 
 \begin{Claim}\label{vanity} For $dx$--a.e. $x\in \bbR^d$ it holds 
\begin{equation}\label{mattacchione}
\int d \nu(\o, z)  w(x,\o, z)  z= 2 D  \nabla_* u (x)\,.
\end{equation}
 \end{Claim}
 \begin{proof}[Proof of Claim \ref{vanity}] 
We apply \eqref{salvoepino} to the test function $v(x):= \e \varphi (x) g(\theta_{x/\e}\tilde \o)$, where 
$\varphi \in C^2_c(\bbR^d)$ and $g \in \cG_2$ (cf. Section \ref{topo}). Recall that $\cG_2$ is given by bounded functions.   \ovo{Note that   $v\in \ovo{\cC(\e \widehat{\tilde\o}) \subset H^{1,{\rm f}}_{\tilde \o, \e}}$.}

Due to \eqref{leibniz} we have
 \begin{equation}\label{aquilotto}
\nabla_\e  v (x,z)
= \e \nabla _\e \varphi (x,z) g(\theta_{z+x/\e}\tilde \o)+  \varphi(x) \nabla  g ( \theta_{x/\e} \tilde \o, z)\,.
\end{equation}
In the above formula, 
the gradient $\nabla g$ is the one defined in \eqref{cantone}.
\ovo{Note that both the expressions in the r.h.s. belongs to $ L^2( \nu ^\e_{\tilde \o})$ (as this holds for the l.h.s., it is enough to check it for the first expression by using \eqref{aldo85}).}
 Due to \eqref{aquilotto},
\eqref{salvoepino} can be rewritten as 
\begin{equation}\label{cioccolata}
\begin{split}
& \frac{\e}{2}\int d  \nu ^\e_{\tilde \o} (x,z)  \nabla _\e \varphi (x,z) g(\theta_{z+x/\e}\tilde \o) \nabla _\e u_\e (x,z) +\\
&  \frac{1}{2}\int d  \nu ^\e_{\tilde \o}  (x,z)  \varphi(x) \nabla g ( \theta_{x/\e} \tilde \o, z)\nabla_\e  u_\e (x,z) +\\
&\e \int d\mu^\e_{\tilde \o} (x)  \varphi (x) g(\theta_{x/\e}\tilde \o)u_\e(x) =\e \int d\mu^\e_{\tilde \o} (x)  \varphi (x) g(\theta_{x/\e}\tilde \o)f_\e(x)\,.
\end{split}
\end{equation}
Since the families of functions $\{u_\e(x) \}$, $\{f_\e(x)\}$,  $\{ \varphi (x) g(\theta_{x/\e}\tilde \o)\}$ are bounded families in    $L^2(\mu^\e_{\tilde \o})$,  the expressions in the third line of \eqref{cioccolata} go to zero as $\e \da 0$.

We now claim that 
\be \label{ghiaccio}
\lim  _{\e \da 0}  \int d \nu^\e_{\tilde \o}(x,z)   \nabla _\e u_\e (x,z)  \bigl[\nabla_\e \varphi (x,z) - \nabla  \varphi (x) \cdot z \bigr] g (\theta_{z+x/\e}\tilde \o) =0\,.
\en
This follows by  using that $\|g\|_\infty<+\infty$, applying Schwarz inequality and afterwards Lemma \ref{blocco} (recall that  $\|  \nabla _\e u_\e \| _{ L^2 ( \nu^\e_{\tilde \o})}\leq C$ \ovo{for $\e$ small}).
The above limit \eqref{ghiaccio},  the 2-scale convergence $\nabla_\e u_\e\stackrel{2}{ \toup }w$ and the fact that \eqref{yelena} holds for all functions in $\cH_3\subset \cH$ (cf. Section \ref{topo}), imply that 
\begin{multline}\label{bisonte}
\lim  _{\e \da 0}  \int d \nu^\e_{\tilde \o}(x,z)   \nabla_\e u_\e (x,z)  \nabla _\e \varphi (x,z)   g (\theta_{z+x/\e}\tilde \o)=\\
\lim  _{\e \da 0}  \int d \nu^\e_{\tilde \o}(x,z)   \nabla_\e u_\e (x,z)  \nabla  \varphi (x)\cdot z   g (\theta_{z+ x/\e}\tilde \o)=\\
\int dx \, m \int d \nu( \o,z) w(x,\o,z) \nabla \varphi(x) \cdot z g(\theta_z \o) \,.
\end{multline}
Due to \eqref{bisonte} also the expression in the first line of \eqref{cioccolata} goes to zero as $\e\downarrow 0$. We conclude therefore that also the   expression in the second  line of \eqref{cioccolata} goes to zero as $\e\downarrow 0$. Hence 
\[\lim_{\e \da 0}\int d  \nu ^\e_{\tilde \o}  (x,z) \nabla_\e  u_\e (x,z)  \varphi(x) \nabla g ( \theta_{x/\e} \tilde \o, z)=0\,.\]
Due to the 2-scale convergence $\nabla_\e  u_\e\stackrel{2}{ \toup} w$  and since \eqref{yelena} holds for all gradients $\nabla g$, $g\in \cG_2$ (since $\cH_2\subset \cH$), we conclude that 
\begin{equation*}
\int dx \, m \varphi(x) \int d\nu (\o, z) w(x,\o,z) \nabla g (\o,z) =0\,.
\end{equation*}
Since $\{\nabla g \,:\, g \in \cG_2\}$ is dense in $L^2_{\rm pot}(\nu)$,  the above identity implies that, for $dx$--a.e. $x$, the map  $(\o,z) \mapsto w(x,\o,z)$ belongs to $L^2_{\rm sol}(\nu)$.
 On the other hand, we know that 
$w(x,\o,z)= \nabla_*  u (x) \cdot z + u_1(x,\o,z)$, 
where $u_1\in L^2\bigl( \bbR^d, L^2_{\rm pot} (\nu)\bigr )$. Hence, by \eqref{jung}, for $dx$--a.e. $x$ we have that 
\[ u_1 (x,\cdot,\cdot ) = v^a(\cdot,\cdot) \,, \qquad a:= \nabla _* u (x)\,. \]
As a consequence (using also \eqref{solare789}),  for $dx$--a.e. $x$, we have
\begin{equation*}
\int d \nu(\o, z)  w(x,\o, z)  z= \int d \nu(\o,z) z [  \nabla_* u (x)\cdot z + v ^{  \nabla _* u (x)}(\o,z)]= 2 D \nabla_* u (x)\,.
\end{equation*}
This concludes the proof of  Claim \ref{vanity}.
\end{proof}
We now reapply \eqref{salvoepino} but with $v(x):= \varphi(x) \ovo{\in C^2_c(\bbR^d)}$ \ovo{(note that  $v \in \cC(\e \widehat{\tilde\o}) \subset  H^{1,{\rm f}}_{\tilde \o, \e}$).}
We get
\begin{equation}\label{cioccorana}
 \frac{1}{2}\int d  \nu ^\e_{\tilde \o} (x,z) \nabla _\e \varphi (x,z) \nabla _\e u_\e (x,z) +
  \int d\mu^\e_{\tilde \o} (x)  \varphi (x) u_\e(x) =\int d\mu^\e_{\tilde \o} (x)  \varphi (x) f_\e(x)\,.
\end{equation}
Let us analyze the first term in \eqref{cioccorana}.
By  \eqref{ghiaccio} which holds also with $g\equiv 1$,   the expression $ \int d  \nu ^\e_{\tilde \o} (x,z) \nabla_\e \varphi (x,z) \nabla_\e  u_\e (x,z) $ equals  $ \int d \nu^\e_{\tilde \o}(x,z)   \nabla _\e u_\e (x,z) \nabla  \varphi (x) \cdot z +o(1)$ as $\e\da 0$. Since the function  $(\o,z) \mapsto z_i$  is in $\cH$ and since $\tilde \o \in \O_{\rm typ}$, by  the 2-scale convergence  $\nabla _\e u_\e
 \stackrel{2}{\toup} w$ we obtain that 
\begin{equation*}
\lim_{\e \da 0} \int d  \nu ^\e_{\tilde \o} (x,z) \nabla_\e \varphi (x,z) \nabla _\e u_\e (x,z)  =
\int dx \, m \int d \nu(\o, z)  w(x,\o, z) \nabla  \varphi (x)\cdot z\,.
\end{equation*}
Since   $u_\e \stackrel{2}{\toup} u$ with  $u=u(x)$, $1\in \cG$    and  $f_\e {\toup} f$,  by taking the limit $\e \da 0$ in \eqref{cioccorana} we get 
\begin{multline}
\frac{1}{2}\int dx \, m \nabla \varphi (x)\cdot \int d \nu(\o, z)  w(x,\o, z)  z+\\
\int dx\, m\varphi(x) u(x) =\int dx\, m  \varphi(x) f(x)
\,.\end{multline}
Due to  \eqref{mattacchione} the above identity reads 
\begin{equation}\label{basta180}
 \int dx \nabla  \varphi (x) \cdot D  \nabla_* u (x)
+\\
\int dx \varphi(x) u(x) =\int dx  \varphi(x) f(x)\,,
\end{equation}
i.e. 
 $u$ is a weak solution of \eqref{eq2} (recall  that $C_c^\infty (\bbR^d) $ is dense in $H^1_*(m dx)$ and note  that $\nabla \varphi$ in \eqref{basta180} can be replaced by $\nabla_* \varphi$ due to \ovo{Definitions \ref{divido}, \ref{degenerare}}). This concludes the proof of limit  \eqref{limite1} \vvv{as $\e \downarrow 0$ along the sequence $\{\e_k\}$. Since for each sequence $\{\e_n\}$ converging to zero  we can extract a subsequence $\{\e_{n_k}\}$  for which Items (i) and (ii) stated after \eqref{salvezza} hold and for which the limit  \eqref{limite1}  holds along $\{\e_{n_k}\}$, we get \eqref{limite1}    as $\e \downarrow 0$}.

It remains to prove \eqref{limite2}. It is enough to apply the same arguments of \cite[Proof of Thm.~6.1]{ZP}.   Since $f_\e\to f$ we have $f_\e\toup f$ and therefore, by \eqref{limite1}, we have $u_\e \toup u$. This implies that $v_\e \toup v$ (again by \eqref{limite1}), where $v_\e$ and $v$ are respectively the weak solution \ovo{in $H^{1,{\rm f}}_{\tilde \o, \e}$ and $H^1_*(m dx)$}  of $-\bbL_{\ovo{\tilde\o}}^\e v_\e+  v_\e =u_\e$ and $-\nabla_*\cdot  D  \nabla_* v+  v  =u$. By taking the scalar product \ovo{in} the weak version of \eqref{eq1} with $v_\e$ (as in \eqref{salvoep}), the scalar product \ovo{in} the weak version of \eqref{eq2} with $v$ (as in \eqref{delizia}), the scalar product \ovo{in}  the weak version of  $-\bbL^\e_{\ovo{\tilde{\o}}} v_\e+  v_\e =u_\e$ with $u_\e\ovo{\in H^{1,{\rm f}}_{\tilde\o, \e}}$ and the scalar product \ovo{in} the weak version of  $-\nabla_* \cdot D \nabla_* v+  v  =u$ with $u$ and comparing the resulting expressions, we get
\begin{equation}\label{adria}
\la u_\e, u_\e \ra _{\mu^\e_\o} = \la v_\e, f_\e \ra _{\mu^\e_\o} \,, \qquad \int  u(x)  ^2   dx=  \int  f(x)v(x)  dx\,.
\end{equation}
Since $f_\e \to f$ and $v_\e \toup v$ we get that $\la v_\e, f_\e \ra _{\mu^\e_\o}\to \int v(x) f(x)  m dx$. Hence, by  \eqref{adria}, we conclude that
$\lim_{\e\da 0}  \la u_\e, u_\e \ra _{\mu^\e_\o}=  \int  u(x)^2 m  dx$. The last limit and the weak convergence $u_\e \toup u$ imply the strong convergence $u_\e \to u$ by Remark \ref{forte}. This concludes the proof of \eqref{limite2} and therefore of Theorem \ref{teo1}--(i).

\smallskip

\noindent
$\bullet$ {\bf Convergence of flows}. We \ovo{now} prove  \eqref{limite3} in   Item (ii), i.e. $\nabla_\e u_\e \toup \nabla_* u$. By \eqref{salvezza} the analogous of bound \eqref{recinto_grad}  \ovo{with $\nabla_\e u_\e$} is satisfied. Suppose that $f_\e \toup f$. Take $\varphi \in C_c^1(\bbR^d)$, then $\la \varphi, f_\e \ra_{ \mu_{\tilde\o}^\e }\to \la \varphi, f \ra_{m dx }$. By Item (i) we know that $u_\e \toup u$ and therefore $\la \varphi, u_\e \ra_{\mu_{\tilde \o}^\e }\to \la \varphi, u \ra_{m dx }$. The above convergences and  \eqref{salvoepino} with $v$ given by  $\varphi $ restricted to $\e \widehat{ \tilde \o}$ (\ovo{note that   $v\in\cC(\e\widehat{ \tilde \o})\subset H^{1,{\rm f}}_{\tilde \o,\e}$}), we conclude that \[
\lim _{\e\da 0} \frac{1}{2} \la \nabla_\e \varphi, \nabla _\e u_\e\ra_{ \nu^\e _{\tilde \o} } = \lim_{\e \da 0} \Big[
\la \varphi, f_\e \ra_{ \mu_{\tilde \o}^\e }-\la \varphi, u_\e \ra_{\mu_{\tilde \o}^\e }\Big]= \la  \varphi, f-u  \ra_{ m dx }\,.\]
Due to \eqref{eq2} and \eqref{delizia}, the r.h.s. equals $\int dx \, m D(x) \nabla_* \varphi(x) \cdot\nabla_* u(x) $.
This proves the analogous of \eqref{deboluccio_grad} and therefore  \eqref{limite3}.

Take now $f_\e \to f$. Then, by \eqref{limite2}, $u_\e \to u$.  Reasoning as above we get  that, given $\ovo{g_\e \in H^{1,{\rm f}}_{ \tilde \o,\e}}$ and $g\in H^1_*(m dx)$ with  $L^2(\mu^\e_{\tilde \o}) \ni g_\e \toup g\in L^2(m dx)$, it holds 
\[
\lim _{\e\da 0} \frac{1}{2} \la \nabla_\e g_\e, \nabla _\e u_\e\ra_{\nu^\e _{\tilde \o} } = \lim_{\e \da 0} \Big[
\la g_\e , f_\e \ra_{\mu_{\tilde \o}^\e}-\la g_\e, u_\e \ra_{\mu_{\tilde\o}^\e}\Big]=
 \la  g, f-u   \ra_{m dx }\,.\]
 Since $g\in H^1_*(m dx)$, due to \eqref{eq2}, the r.h.s. equals $ \int dx \, m D(x)  \nabla_*g(x) \cdot \nabla_* u(x) $.
This proves \eqref{limite4}.

\smallskip

\noindent
$\bullet$ {\bf Convergence of energies}. We prove  Item (iii). Since $f_\e \to f$, we have $u_\e \to u$ by \eqref{limite2} and  $\nabla_\e u_\e \to  \nabla_* u $ by \eqref{limite4}.  It is enough to apply \eqref{fortezza_grad} with \ovo{$v_\e$ replaced by $u_\e$, $w$ replaced by $\nabla_* u$}, $g_\e := u_\e$ and $g:= u$ and one gets \eqref{limite5}.



\section{Proof of Theorem \ref{teo2}}\label{dim_teo2}
{Limit \eqref{vinello} corresponds to Remark \ref{gelatino}.
Limit \eqref{marvel0} follows from 
\eqref{vinello} and 
\cite[Thm.~9.2]{ZP}.
To treat \eqref{marvel1}, \eqref{marvel2}, \eqref{ondinoA}, \eqref{ondinoB} we claim  that it is enough to prove them \iva{for a  fixed $f$ and  for all $t\geq 0$, $\l>0$ as $\o$ varies in a translation invariant set of full $\cP$--probability}.  To prove this  claim, given $n \in \bbN$ we set $B_n:=\{ x\in \bbR^d\,:\, |x| \leq n\}$ and \iva{we restrict to $\o\in \cA[1]$ (cf. Prop. \ref{prop_ergodico}), thus implying  that $\lim _{\e \da 0} \mu_\o^\e (B_n ) = m \ell ( B_n) $ (recall that $\cA[1]$ is translation invariant and measurable  and that $\cP(\cA[1])=1$)}. Then, given $f\in C_c(\bbR^d)$ with support in $B_n$,  we have (since $\mu^\e_\o$ is reversible for $  P^\e_{\o,t} $)
\be
\begin{split}
 \| P^\e_{\o,t} f \|^2_{L^2(\mu^\e_\o) } & = \int d \mu^\e_\o(x) \bigl( P^\e_{\o,t} f\bigr)^2 (x) \leq \int d \mu^\e_\o(x)  P^\e_{\o,t} f^2  (x)\\
 &  =\mu^\e_\o( f^2) \leq \| f\|_\infty^2 \mu^\e_\o(B_n)\,.
\end{split}
\en
Hence  $\| P^\e_{\o,t} f \|_{L^2(\mu^\e_\o) }\leq \| f\|_\infty \mu^\e_\o(B_n)^{1/2}$.
Similarly, $ \| P^\e_{\o,t} f \|_{L^1(\mu^\e_\o) }\leq  \| f\|_\infty \mu^\e_\o(B_n)$. 
As $R^\e _{\o ,\l} =\int_0^\infty e^{- \l s} P^\e _{\o ,s} ds $, we also have 
$\| R^\e_{\o,\l} f \|_{L^2(\mu^\e_\o) }\leq \l^{-1} \| f\|_\infty \mu^\e_\o(B_n)^{1/2}$ and
 $\| R^\e_{\o,\l} f \|_{L^1(\mu^\e_\o) }\leq \l^{-1} \| f\|_\infty \mu^\e_\o(B_n)$.
 The same bounds hold for the Brownian motion with diffusion matrix $2D$ and for the measure $m dx$. Hence, by a density argument with functions in $C_c(\bbR^d)$,  one gets our claim.

To prove \eqref{marvel1},\dots,\eqref{ondinoB} for fixed $f$
we need the following fact (proved at the end of the section):
\begin{Lemma}\label{pre_teo2}
Suppose that Assumption (A9) is  satisfied.
 Fix 
a weakly decreasing  function $\psi: [0,+\infty) \to [0,+\infty)$  such that $\bbR^d \ni x \mapsto \psi(|x|) \in [0,\infty)$ is  Riemann integrable.
Then
$\cP$--a.s.  it holds 
\be \label{claudio2}
 \lim _{\ell \uparrow \infty} \varlimsup_{\e \da 0} \int d \mu^\e _\o (x)\psi(|x|)  \mathds{1}_{\{ |x| \geq \ell\}}=0\,. \en
\end{Lemma}
\iva{We will apply the above lemma only  with $\psi(r):=1/(1+r^{d+1})$. By this choice it is then simple to check that \eqref{claudio2} holds for $\o$ varying in a translation invariant set, as $\int d \mu^\e_{\theta_g \o} (x) f(x)= \int d\mu^\e_\o(x) f(\t_{-g}x)$ for all $g\in \bbG$ and $\o\in \O_*$.}

By the same arguments used to prove \cite[Lemma 6.1]{F1}, the above lemma implies the following for 
 \iva{$\o$ varying in a translation invariant set with full $\cP$--probability}: given $h\in C(\bbR^d)$ with $|h(x)| \leq C/(1+ |x|^{d+1})$ for all $x\in \bbR^d$ and given $L^2(\mu^\e_\o)\ni h^\e_\o \to h\in L^2(mdx )$, it holds $\lim_{\e \da 0} \int_{\bbR^d} | h^\e_\o (x) - h(x)|^2 d\mu^\e_\o(x) =0$. Then, as $P_t f$ and $R_\l f$ decay exponentially fast, \eqref{marvel0} and \eqref{vinello} imply, respectively, \eqref{marvel1} and \eqref{ondinoA}  by Lemma \ref{pre_teo2}.
It remains to derive \eqref{marvel2} from \eqref{marvel1} and to derive \eqref{ondinoB} from
 \eqref{ondinoA}. We use some manipulations as in the proof of \cite[Corollary 2.5]{F1}. We show the derivation of \eqref{ondinoB} (which is absent in \cite{F1}), since the derivation of \eqref{marvel2} is similar.
To this aim, without loss of generality, we  restrict to $f\geq 0$. For any $n \in \bbN$ we can  bound 
\begin{equation}\label{panna}
\begin{split}
&\| R^\e_{\o,\l} f  - R_\l  f  \|_{L^1(\mu^\e_\o )} \leq \| (R^\e_{\o,\l} f )  
\mathds{1}_{B_n^c}   \|_{L^1(\mu^\e_\o  )}
\\
& +  \|   (R_\l f)\mathds{1}_{B_n^c}\|_{L^1(\mu^\e_\o )}+ \mu^\e_\o ( B_n) ^\frac{1}{2} \| R^\e_{\o,\l} f  - R_\l  f  \|_{L^2(\mu^\e_\o )} \,.
\end{split}
\end{equation} 
We restrict to  $\o \in \cA[1]$ (\iva{$\cA[1]$ is a translation invariant measurable set with $\cP( \cA[1])=1$}) and $\o$ satisfying \eqref{claudio2} with \iva{$\psi(r):=1/(1+r^{d+1})$}. Hence 
it holds  $\lim _{\e \da 0} \mu^\e_\o ( B_n)= m \ell (B_n) $. Then,  by \eqref{ondinoA},  the last addendum in the r.h.s. of   \eqref{panna} goes to zero as $\e\da 0$.
Let us move to the second addendum in the r.h.s. of \eqref{panna}. As $R_\l f$ decays exponentially \ovo{(hence $R_\l f \leq C \psi$)},
by Lemma  \ref{pre_teo2}  it holds $ \lim_{n\uparrow \infty}\lim_{\e \da 0 }\|   (R_\l f)\mathds{1}_{B_n^c}\|_{L^1(\mu^\e_\o )}= 0$.

Let us finally move to the first addendum in the r.h.s. of \eqref{panna}.
Since $R^\e_{\o,\l} f\geq 0$ we can write
\be\label{giostra0} \| (R^\e_{\o,\l} f )  
\mathds{1}_{B_n^c} \|_{L^1(\mu^\e_\o  )}=\| R^\e_{\o,\l} f \|_{L^1(\mu^\e_\o  )} -  \| (R^\e_{\o,\l} f )  
\mathds{1}_{B_n} \|_{L^1(\mu^\e_\o  )}\,.\en
 We   claim that 
\be\label{pig100}
\lim_{\e \da 0} \|R^\e_{\o,\l} f \|_{L^1(\mu^\e_\o )}=\| R_\l f \| _{L^1(m dx)}\,.
\en
To prove our claim 
 we observe  that, as 
 $R^\e _{\o ,\l} =\int_0^\infty e^{- \l s} P^\e _{\o ,s} ds $ and $f\geq 0$,  it holds $R^\e_{\o,\l} f \geq 0$ and 
therefore 
\[  \|  R^\e_{\o,\l} f \|_{L^1(\mu^\e_\o )} =\int_0^\infty ds e^{- \l s}  \int d\mu^\e_\o (x) P^\e _{\o ,s} f (x) =\int_0^\infty ds e^{- \l s}  \mu^\e_\o (f) =\frac{ \mu^\e_\o(f)}{\l}\,.
\] 
As $\o \in \cA[1]$ we have 
 $ \mu^\e_\o(f)   \to \int dx \, m f(x)  $. 
On the other hand, arguing as above, we get $\| R_\l f \|_{L^1(m dx  )} = \l^{-1}\int dx \, m f(x)   $.  By combining the above observations we get  \eqref{pig100}.

Now we claim that 
\be \label{lyon}
\lim_{\e \da 0}  \| (R^\e_{\o,\l} f )  
\mathds{1}_{B_n} \|_{L^1(\mu^\e_\o  )}= 
\| (R_\l f  ) \mathds{1}_{B_n}\| _{L^1(m dx)}\,.
\en Indeed, 
by Schwarz inequality,  \eqref{ondinoA} and since $\o \in \cA[1]$, 
 $\| (R^\e_{\o,\l} f    )\mathds{1}_{B_n} \|_{L^1(\mu^\e_\o  )}-    \| (R_{\l} f  )
\mathds{1}_{B_n} \|_{L^1(\mu^\e_\o  )}$ $ = \la R^\e_{\o,\l} f   -    R_{\l} f  ,  \mathds{1}_{B_n} \ra _{L^2(\mu^\e_\o  )}  $ goes to $0$ as $\e \da 0$.
On the other hand, as $\o \in \cA[1]$ and  by  upper and lower bounding  $(R_{\l} f  )
\mathds{1}_{B_n}$ with nonnegative functions $\varphi \in C_c(\bbR^d)$, we get that $\| (R_{\l} f  )
\mathds{1}_{B_n} \|_{L^1(\mu^\e_\o  )} \to \| (R_\l f  ) \mathds{1}_{B_n}\| _{L^1(m dx)}$ as $\e \da 0$, thus allowing to prove our claim \eqref{lyon}.

By combining \eqref{giostra0}, \eqref{pig100} and \eqref{lyon} we get that
$ \| (R^\e_{\o,\l} f )  
\mathds{1}_{B_n^c} \|_{L^1(\mu^\e_\o  )}$ 
goes to  $ \| (R_\l f  ) \mathds{1}_{B_n^c}\| _{L^1(m dx)}$ as $\e \da 0$. 
By taking then the limit $n\uparrow \infty$ we get that $\lim _{n\uparrow \infty} \lim _{\e \da 0}\| (R^\e_{\o,\l} f )   \mathds{1}_{B_n^c} \|_{L^1(\mu^\e_\o  )}=0$.

\begin{proof}[Proof of Lemma \ref{pre_teo2}] 
\ovo{To simplify the notation, we take $V=\bbI$ in \eqref{trasferta1}, thus implying that $\t_k \D= k+[0,1)^d$ (the arguments below can be easily adapted to the general case). Always to}
simplify the notation\ovo{,} we prove a slightly different version of \eqref{claudio2}, the method can be easily adapted to \eqref{claudio2}. In particular, we now  prove that $\cP$--a.s. it holds 
\be\label{claudio3}
 \lim _{\ell \uparrow \infty} \varlimsup_{\e \da 0}X_{\e,\ell}=0\; \text{ where }\; X_{\e, \ell} (\o):=  \e^d \sum_{
\substack{ k\in \bbZ^d:|k| \geq  \ell/\e } }\psi( | \e k| ) N_k\,.
\en
Trivially  Item (i) in Assumption (A9)  implies \eqref{claudio3}. Let us suppose that Item (ii) is satisfied (we restrict  to  the first case in Item (ii), the second more  general case can be treated similarly).
Given $\e\in(0,1)$ let $r=r(\e) $ be  the positive integer of the form $2^a$, $a 
\in \bbN$, such that $r^{-1} \leq \e <2 r^{-1}$.  Then, since $\psi$ is weakly decreasing,  
\be \label{cenere}
X_{\e, \ell} (\o) \leq2^d    Y_{r, \ell} (\o) \;\text{ where } \;   Y_{r, \ell} (\o):=r^{-d} \sum_{
\substack{ k\in \bbZ^d: |k| \geq  r \ell/2 } } \psi(| k/r|)N_k\,.
\en
In particular, to get \eqref{claudio3} it is enough to show that,  $\cP$--a.s,   $\lim _{\ell \uparrow \infty} \varlimsup_{r  \uparrow \infty }Y_{r,\ell}=0$, where $r$ varies in $\G:=\{2^0,2^1,2^2,\dots\}$.  From now on we understand that $r \in \G$. Since $\bbE[N_k]=m$ and since $\psi(|x|)$ is Riemann integrable,  we have
\be\label{fabietto1}
\lim_{r\uparrow \infty } \bbE[Y_{r, \ell}] = z_\ell:=m \int  \psi( |x|)   \mathds{1}_{\{|x| \geq \ell/2\}}dx <\infty\,.
\en
We now estimate the variance of $Y_{r, \ell}$. Due to the stationarity of $\cP$ and since $\bbE[N_0^2]<+\infty$, it holds $\sup_{k \in\bbZ^d} \text{Var} (N_k)<+\infty$.
 By Condition (ii) we have, for some fixed constant $C_1>0$,  
\begin{align*}
\text{Var} (Y_{r, \ell})& \leq C_1  r^{-2d} \sum_{
\substack{ k\in \bbZ^d:\\|k| \geq  r \ell/2 } }\;
\sum_{
\substack{ k'\in \bbZ^d:\\|k'| \geq  r \ell/2 } }
 \left[|k-k'|^{-1}\mathds{1}_{k\not =k'}+ \mathds{1}_{k=k'}  \right] \psi (| k/r|) \psi( | k'/r|) \\
 &=: I_0(r,\ell)+ I_1(r,\ell)+I_2(r,\ell) \,,
\end{align*}
where $I_0(r, \ell)$, $I_1(r,\ell)$ and $I_2(r,\ell)$  denote   the contribution from addenda as above  respectively with  (a) $k=k'$,  (b)    $|k-k'|\geq r$ and (c) $1\leq |k-k'|<r$.
Then we have 
\begin{align}
&    \lim _{r \uparrow \infty}   r ^{d} I_0(r,\ell)=C_1  \int _{|x| \geq \ell/2} \psi(|x|)^2 dx <+\infty\,, \label{vino0} \\
&  \lim _{r \uparrow \infty}   r  I_1(r,\ell)= C_1 \int _{|x| \geq \ell/2}dx \int _{|y| \geq \ell/2} dy \frac{ \mathds{1}_{\{ |x-y|\geq 1\} } }{  |x-y|}\psi(| x| )\psi (|y|)  <+\infty \,. \label{vino1}
\end{align}
To control $I_2(r, \ell)$  we observe that, for $r\geq 2$,  
\[
 \sum _{\substack{v \in \bbZ^d\,:\\ 1\leq \|v\|_\infty \leq c r} }\|v\|_\infty ^{-1 } \leq C' \sum _{n=1}^{ c r } n^{d-2}\leq \begin{cases}
 C'' r^{d-1} & \text{ if } d\geq 2\\
 C'' \ln r & \text{ if } d=1 
 \end{cases}
  \,.
\]
The above bound implies for $r$ large   that 
\be\label{vino20}
\begin{split}
I_2(r, \ell) & \leq   C_1 \|\psi\|_\infty   r^{-2d} \sum_{
\substack{ k\in \bbZ^d:\\|k| \geq  r \ell/2 } }\psi( | k/r| )
\sum_{
\substack{ k'\in \bbZ^d:\\ 1 \leq |k-k'| 	\leq r  }}
 |k-k'|^{-1}\\
& \leq
 C_2 r^{-1 } \ln r \int_{|x| \geq \ell/2} \psi( | x| ) dx \,.
 \end{split}
\en
Due to \eqref{vino0}, \eqref{vino1} and \eqref{vino20},  $\text{Var} (Y_{r, \ell})\leq C_3(\ell) r^{-1} \ln r $ for $ r\geq C_4(\ell)$. 
Now we write explicitly $r= 2^j$. 
 By Markov's inequality, we have  for $j\geq C_5(\ell)$ that \[   \cP( |Y_{2^j, \ell}-\bbE[Y_{2^j, \ell}]| \geq 1/j ) \leq j^2 \text{Var} (Y_{2^j, \ell})\leq C_3(\ell) j^2 \ln (2^j)  2^{-j}\,.
\]
Since the last term is summable among $j$, by Borel--Cantelli lemma we conclude that, for $\cP$--a.a. $\o$, 
$
|Y_{2^j, \ell}(\o) -\bbE[Y_{2^j, \ell}]| \leq  1/j $ for  all $ \ell \geq 1$ and $ j\geq C_6(\ell, \o)$.
This proves that,  $\cP$--a.s., $\lim _{r\uparrow \infty, r \in \G}Y_{r, \ell}= z_\ell$ (cf. \eqref{fabietto1}). Since $\lim_{\ell \uparrow \infty}z_\ell=0$, we get that $\lim _{\ell \uparrow \infty} \lim_{r  \uparrow \infty,r \in \G }Y_{r,\ell}=0$, $\cP$--a.s.
\end{proof}



\appendix


\section{\utto{Further comments on assumptions (A3),...,(A6)}}\label{misurino78}
  In this appendix we extend our comments concerning the choice of the set $\O_*$ in our main assumptions when $\bbG=\bbR^d$. We recall that all sets $\O_k$ are translation invariant.
  
   We first point out that  $\O_5$ and $\O_6$ are always measurable. Indeed    the points of   the simple point process $\hat \o$ can be  enumerated as  $x_1(\hat\o), x_2 (\hat\o),\dots $   in a measurable way by ordering the points according to their distance from the origin and,  in case of more points at the same distance, ordering these points 
  in lexicographic order   (see \cite[p.~480]{DV} for details). 
 By using the above measurable functions   $x_1(\hat\o), x_2 (\hat\o),\dots $, one can easily express $\O_5$ and $\O_6$ as countable intersection
of measurable sets,  thus leading to  the measurability of $\O_5$ and $\O_6$. As a consequence, in (A5) and (A6), one could replace ``$\forall \o \in \O_*$" by ``for $\cP$--a.a. $\o$". 
 
     We now claim  that $\O_3$ is measurable if \eqref{base} holds for all $\o\in \O$ and $g\in \bbG$. To prove our claim we observe that 
      $\theta_g \o  = \theta_{g'}\o$  for some $ g\not= g'$ in $\bbG$
      if and only if  $\theta_g \o  = \o$  for some $ g\in \bbG\setminus\{0\}$. The last property implies that 
 $\mu_{\theta _g \o}  =\mu_\o$  for some $ g \in \bbG \setminus\{0\}$ and therefore, by \eqref{base},  that   $\t_{g} \mu_{ \o }= \mu_\o$  for some $ g \in \bbG \setminus\{0\}$. If $\t _g\mu_{\o}=\mu_\o$, then $\t_{-g} \hat \o   =\hat \o$ and therefore $g= V^{-1}(x-y)$ for some  $x,y \in \hat \o$. 
  Hence, when \eqref{base} holds for all $\o \in \O$, then  $\O_3$ is given by  the measurable set $\cap _{m\not =n} \{\o\in\O\,:\, \theta_{V^{-1}(x_m(\hat \o)-x_n(\hat \o)) }\o\not= \o \}= \{\o\in \O\,:\, \theta_{V^{-1} x} \o \not = \theta_{ V^{-1} y} \o \; \forall x\not= y \text{ in } \hat\o\}$. This concludes the proof of our claim. As a consequence, if \eqref{base} holds for all $\o\in \O$ and $g\in \bbG$,  in (A3) we can  replace ``$\forall \o \in \O_*$" by ``for $\cP$--a.a. $\o$".

\section{Campbell's identity}\label{app_bell}
In this appendix we consider the general context (instead of  the context of Warning \ref{marinaio}) and we 
recall    Campell's identity for the Palm distribution $\cP_0$.  For what follows, it is enough to require only (A1) and (A2). Recall definition  \eqref{simplesso} of $\D$ and that $ \O_0:=\{\o\in \O\,:\, n_0(\o)>0\}$ when  $\bbG=\bbR^d$ and in the special discrete case, and that $\O_0:=\{(\o, x)\in \O\times \D\,:\,n_x(\o)>0\}$ when $\bbG=\bbZ^d$.

$\bullet$ {\sl Case $\bbG=\bbR^d$}. For any   \rrr{measurable} function $f: \bbR^d\times \O_0 \to[0,+\infty) $ it holds
  \begin{equation}\label{campanello}
 \int_{\bbR^d}dx  \int _{\O_0} d\cP_0 ( \o) f(x, \o) =\frac{1}{m \ell(\D) } \int _{\O}d\cP(\o)\int_{\bbR^d}  d\mu_\o (x) f(g(x) , \theta_{g(x)} \o) 
 \end{equation}
 (cf.  \cite[Eq. (4.11)]{Ge} together with  Appendix \ref{gennaro} \ovo{below},  cf. \cite[Thm.~13.2.III]{DV2}).
 \ovo{As $g(x)=V^{-1} x$ by \eqref{attimino}, by taking $f(x, \o):= \mathds{1}_{V^{-1}U} (x) \mathds{1}_A(\o)$ Campbell's identity \eqref{campanello}  reduces to \eqref{palm_classica}. Moreover, note}
 that, when $V=\bbI$, \ovo{from \eqref{campanello}} we recover the more common Campbell\ovo{'s}  formula 
\begin{equation}
 \int_{\bbR^d}dx  \int _{\O_0} d\cP_0 ( \o) f(x, \o) =\frac{1}{m  } \int _{\O}d\cP(\o)\int_{\bbR^d}  d\mu_\o (x) f(x  , \theta_{x} \o)\;\text{ if } V=\bbI.\label{ciampino}
 \end{equation}

$\bullet$ {\sl Case $\bbG=\bbZ^d$}.
For any  \rrr{measurable} function $f:\bbZ^d \times \O_0 \to [0,+\infty)$ it holds 
\begin{equation}\label{campanelloZ}
\begin{split}
&\sum_{g\in \bbG}\int _{\O_0}d\cP_0 (\o, x) f(g,\o,  x) \\
&= \frac{1}{m\, \ell(\D)}
\int _\O  d\cP( \o) \int _{\bbR^d} d\mu_\o (x) f\bigl(g(x),  \theta_{g(x)} \o, \b(x)\bigr )\\
& = \frac{1}{m\, \ell(\D)}\sum_{g\in \bbG}
\int _\O  d\cP(\o) \int _{\t_g \D} d\mu_\o (x) f\bigl(g,\theta_{g} \o,  \t_{-g} x\bigr )
\end{split}
\end{equation}
(cf.  \cite[Eq. (4.11)]{Ge} together with  Appendix \ref{gennaro} \ovo{below}). \ovo{Note that,  Campbell's identity \eqref{campanelloZ} with   $f(g,\o,x):=\d_{0,g} \mathds{1}_A(\o,x)$ reduces to \eqref{Palm_Z}}.

$\bullet$ {\sl Special discrete case}. As discussed in Section \ref{subsec_palm} we think of $\cP_0$ as a probability measure on $\O_0=\{\o\in \O\,:\, n_0(\o)>0\}$.
Then, due to \eqref{campanelloZ}, 
\be  \label{campanelloZ_sp} 
 \sum_{g\in \bbZ^d} \int _{\O_0}d\cP_0 (\o ) f(g, \o  )=\frac{1}{\bbE[n_0]} \sum_{g\in \bbG} \int _\O  d\cP( \o)  n_g(\o)  f\bigl(g, \theta_g\o   \bigr )\,,
 \en
for any   \rrr{measurable} function $f: \bbZ^d\times \O_0 \to [0,+\infty)$.


\section{Sign choices and  cumulative Palm  measure}\label{gennaro}
In this appendix we consider the general context (instead of  the context of Warning \ref{marinaio}) and   explain  how to derive  our formulas concerning the Palm distribution $\cP_0$ from the present literature. Our main reference is given by \cite{Ge}. When the probability space $(\O,\cF, \cP)$ is such that  $\O$ is a Borel space and $\cF= \cB(\O)$, then one can refer as well to the theory of Palm pairs to get our formulas (cf. \cite[Theorem~4.10]{Ge}, \cite{GL,Km}).  For what follows, it is enough to require (A1) and (A2).

 Our action $(\theta_g)_{g\in \bbG}$ on $\O$  is related to the action $(\theta^{\rm Ge}_g)_{g\in \bbG}$ on $\O$ in \cite{Ge} by the identity $\theta^{\rm Ge}_g=\theta_{-g}$ for all $g\in \bbG$ (we have added here the supfix ``Ge'' in order to distinguish the two actions).  As a consequence, when applying some formulas from \cite{Ge}  sign changes are necessary.
By setting $ \t_g \mathfrak{m} (A):= \mathfrak{m} (\t_{g} A)$ in 
 \eqref{mignolo},  we have followed the convention used in \cite[Section 12.1]{DV2}  and \cite[p.~23]{ZP}. 
By setting  $\mu_{\theta_g\o}= \t_g \mu_\o$
 in \eqref{base}, we have followed the convention of \cite[p.~23]{ZP}.  We stress that in e.g.  \cite{Ge}, \cite{GL}, $ \t_g \mathfrak{m} (A)$ is defined as $\mathfrak{m} (\t_{-g} A)$. On the other hand,  the fundamental relation \eqref{base} is valid also in \cite{Ge} (see Eq.~(2.27) in \cite{Ge}). Indeed, as already observed,  $\theta^{\rm Ge}_g=\theta_{-g}$.

The Palm distribution $\cP_0$ introduced in Section \ref{subsec_palm}
is the normalized version of the cumulative Palm measure   $\bbQ$ \cite[Thm.~4.1, Def.~4.2]{Ge}. We clarify the notation there. As Haar measure  $\l$ of $\bbG$ we take the Lebesgue measure $dx$ if $\bbG=\bbR^d$ and the counting measure on $\bbG$ if $\bbG=\bbZ^d$. As \ovo{a} set of orbit representatives $\cO$ \ovo{for the action of $\bbG$ on $\bbR^d$},  we take  $\cO=\{0\}$ if $\bbG=\bbR^d$ and $\cO=\D$ if $\bbG=\bbZ^d$ (cf. \eqref{simplesso}). Given $x\in \bbR^d$,  we define $\b(x):=a$ and $g(x) := g$ if 
$x= \t_g a$ with $a \in \cO$. Note that this definition incorporates  both \eqref{attimino} and \eqref{attimo}  and that $\b(x)\equiv 0$ when $\bbG=\bbR^d$. Given $x\in \bbR^d$, as in \cite{Ge} 
we define $\pi_x : \bbG \to \bbR^d$ as $\pi_x (g):= \t_g x$ (cf. \cite[Sect.~2.2.2]{Ge}) and we introduce the   measure $\mu_x$ on $\bbR^d$ as 
 $\mu_x:= \l \circ \pi_x ^{-1}$  (cf. \cite[Sect.~2.2.3]{Ge}). In particular, $\mu_x (dy)=  \ell (\D) ^{-1}  dy $  if  $\bbG=\bbR^d$, and $\mu_x$ is the counting measure on $\{\t_g x\,:\, g \in \bbZ^d\}$ if $\bbG=\bbZ^d$.  
 Given  $x,y\in \bbR^d$,  we define  the measure  $\k_{x,y}$  
on $\bbG$ as follows:
if $y$ does not belong to the $\bbG$--orbit of $x$, then $\k _{x,y}$ is the zero measure; if 
$y$  belongs to the $\bbG$--orbit of $x$, then $\k _{x,y}$ is 
 the Dirac atomic measure $\d_g$,  
 $g$ being the unique element of $ \bbG$ such that  $\t_g x=y$, \ovo{i.e.~$g=V^{-1} (y-x)$}.
 Given a \rrr{measurable} set $A \subset \bbG$, the map  $\ovo{\bbR^d\times\bbR^d\ni}(x,y)\mapsto \k_{x,y}(A)= \mathds{1}(  V^{-1}(y-x)\in A)\ovo{\in \{0,1\}}$  is \rrr{measurable}. Hence $\k_{x,y} $ is a kernel from $\bbR^d \times \bbR^d$ to $\bbG$ (cf. \cite[Sect.~2.1.2]{Ge}).
 Note moreover that for any  \rrr{measurable} function $f: \bbR^d \times \bbG \to [0,+\infty)$ and for any $x\in \bbR^d$   it holds
 \be
 \int _{\bbG} f(\t_g x, g) d\l (g)=\int_{\bbR^d}  d \mu_x (y) \int _{\bbG} d\k _{x,y} (g)  f(y,g)\,.
 \en
The above identity is  equivalent  to \cite[Eq.~(3.1)]{Ge} with the notation there (take $S:=\bbR^d$ and $g x:= \t_g x$ there). Due to these observations, it is simple to check that the kernel $\k_{x,y}$ satisfies the properties listed in  \cite[Thm.~3.1]{Ge}
\ovo{(note that the map $\theta_g$ appearing in Item (i) of  \cite[Thm.~3.1]{Ge} refers to the action of $\bbG$ on $\bbG$ itself, hence it is the map $\bbG\ni g' \mapsto g+g'\in \bbG$)}.
 Hence, $\k_{x,y}$ is the kernel entering in  \cite[Thm.~4.1]{Ge}.  To apply now \cite[Eq.~(4.4)]{Ge}  we observe that   $\k_{\b(x),x}=\d_{g(x)}$ for  all $x\in \bbR^d$. It remains to fix the function $w:\bbR^d \to [0,+\infty)$ appearing there.  
Defining $w:=\mathds{1}_{  \D}$, due to  the above  description of $\mu_x$ and $\cO$ we have   that $\mu_x (w)=1$ for any $x\in \cO$. We can finally reformulate  \cite[Eq.~(4.4)]{Ge} in our context: for each \rrr{measurable} $B\subset \O\times \cO$ the  measure $\bbQ$ is given by 
\be
\bbQ(B)= \int _{\O} d \cP (\o)\int_{\bbR^d} \mu_\o(dx)\int _{\bbG} \k _{\b(x),x}(dg) \mathds{1} _B  \left (\theta_g \o, \b(x)\right ) w(x)\,,
\en
which reads 
\be\label{anguria}
\bbQ(B)= \int _{\O} d \cP (\o)\int_{\D} \mu_\o(dx)\mathds{1} _B \left(  \theta_{g(x)} \o, \b(x)\right)\,.
\en
Since $\bbQ( \O \times \cO)=\bbE\left[ \mu_\o(\D)\right]= m \ell(\D)$, we get that  $\cP_0:=\bbQ( \O \times \cO)^{-1} \bbQ= \left( m \ell(\D) \right)^{-1} \bbQ$.
When $\bbG=\bbR^d$, as $\cO=\{0\}$ we identify $\O\times \cO$ with $\O$. Using   the stationarity of $\cP$ and that $\b(x)=0$, \eqref{anguria}  becomes  \eqref{palm_classica} \ovo{for $U=\D$}. When $\bbG=\bbZ^d$, as $\cO=\D$ we have $\b(x)=x$ and $g(x)=0$ for any $x\in \D$, hence \eqref{anguria} becomes  \eqref{Palm_Z}.
\section{Proof of Lemma \ref{lemma_no_TNT}}\label{app_no_TNT}
In this appendix we consider the general context (instead of  the context of Warning \ref{marinaio}). 
Due to the discussion in Section \ref{sec_fluidifico}, to prove Lemma \ref{lemma_no_TNT} when  $\bbG=\bbZ^d$ it is enough to prove the analogous claim   for the random walk $\bar X_t$ with rates \eqref{normale} in the setting $\cS[2]$ introduced in Section \ref{sec_fluidifico}. As a consequence, from now on we restrict to the case $\bbG=\bbR^d$ and to the special discrete case. Moreover, to simplify the notation we take $V=\bbI$ (in the general case it would be enough to use Remark \ref{jolie}). \ovo{Recall the notation introduced in \eqref{tildino}}.

We set  $ A_1:=\{ \o\in  \O_0 \,:\, 0<\l_0(\o) <\infty\}$ and define $\tilde A_1$ according to \eqref{tildino}.  We point out that $\cP_0(\l_0>0)=1$ due (A6) and the property that $|\hat \o|=\infty  $ for $\cP$--a.a. $\o$ and therefore for $\cP_0$--a.a. $\o$ by Lemma \ref{matteo} (see the comments 
\ovo{on the main assumptions} 
before Remark \ref{jolie}).   Using that $\bbE_0[\l_0]<\infty$, as done  for Corollary \ref{eleonora},  we get that  $\cP_0(\tilde A_1)=1$.  Since $r_x(\o)=\l_0(\theta _x \o)$, for each $\o \in \tilde A_1$  it holds $r_x(\o)\in(0,+\infty)$ for all $x\in \hat \o$. 

\ovo{Consider now the translation invariant measurable set $\O_*$. By assumption $\cP(\O_*)=1$ and $ \theta _x \o \not = \theta _{y} \o $ for all $ x\not = y$ in $\bbG$ and $\o\in \O_*$. By Lemma \ref{matteo},  we get $\cP_0(\O_*)=1$}.
Note that, given $\o,\o'\in \ovo{\O_*}\cap \O_0$, if $\o'=\theta_x\o$ for some $x\in \bbG$, then $x$ is unique and in this case we  define $r(\o,\o'):= r_{0,x}(\o)$, otherwise we define $r(\o,\o'):=0$.

We set $\ovo{A_2}:=\tilde A_1\cap \ovo{\O_*} \cap \O_0$\ovo{. Due} to the above properties, we have $\cP_0(\ovo{A_2})=1$. Moreover, given $\o_*\in \ovo{A_2}$
we can introduce    the discrete time  Markov chain $(\o_n)_{n \in \bbN}$ with state space $\ovo{A_2}
$, initial configuration $\o_*$ and jump probabilities $p(\o,\o'):= r(\o,\o')/\l_0(\o)$.  We write $P_{\o_*}$ for its law and $E_{\o_*}$ for the associated expectation.
 $P_{\o_*}$ is a probability measure on $\ovo{A_2^{\bbN}}$. 
 
 We now introduce the distribution $d\cQ_0(\o):=\bbE_0[\l_0]^{-1} \l_0(\o)  d\cP_0(\o) $  on $\ovo{A_2}$.
We claim that $\cQ_0$ is  reversible (hence stationary)  and ergodic for the discrete-time Markov chain $(\o_n)_{n\geq 0}$.
To get reversibility  we observe that, by Lemma \ref{lemma_siou}  and since $\bbE_0[\l_0]<\infty$,
\[ 
\bbE_0\Big[  \int_{\bbR^d} d \hat \o(x)  r (\o, \theta _{x}\o)   f(\o) h (\theta_x \o)\Big] =\bbE_0\Big[  
 \int_{\bbR^d} d \hat \o(x)   r (\o, \theta _{x}\o)   h(\o) f(\theta_{x} \o)\Big]\,,
\]
 for any bounded \rrr{measurable} functions $f,h:\O_0\to \bbR$.  Let us now prove the ergodicity of $\cQ_0$. To this aim let $B\subset \ovo{A_2}$ be a set left invariant by the Markov chain. Due to (A6) we get   $B=\O_0 \cap \tilde B$ (cf. \eqref{tildino}). Since  $\tilde B$ is translation invariant and   $\cP$ is ergodic by  (A1), we get that $\cP(\tilde B)\in\{0,1\}$. By Lemma \ref{matteo} we conclude that $\cP_0( B) = \cP_0( \tilde B \cap \O_0)=\cP_0(\tilde B)=\cP(\tilde B)\in \{0,1\}$.

 By the ergodicity of $\cQ_0$  and  since  $\int _{\ovo{A_2}}d\cQ_0(\o) \l^{-1}_0( \o)=\bbE_0[\l_0]^{-1}$, we have 
 \be\label{sardegna}
 \int_{\ovo{A_2}} d\cQ_0(\o) P_\o \Big(  \lim _{n \to \infty} \frac{1}{n}\sum_{k=0}^n\frac{1}{\l_0(\o_n)}= \frac{1}{\bbE_0[\l_0]}\Big)=1\,.
 \en
 Let $\ovo{A_3}:= \{\o \in \ovo{A_2}\,:\, \lim _{n \to \infty} \sum_{k=0}^n\frac{1}{\l_0(\o_n)}=+\infty\; P_\o\text{--a.s.} \}$.  By \eqref{sardegna} and since $\bbE_0[\l_0]<\infty$, we have $\cQ_0(\ovo{A_3})=1$.   As $\l_0>0$ $\cP_0$--a.s.,  
  $\cP_0$ and $\cQ_0$ are mutually absolutely continuous,  hence we conclude that $\cP_0(\ovo{A_3})=1$.
 
 Then for any $\o_*\in  \ovo{A_3} $ we can define the continuous-time Markov chain $(\o_t)_{t\geq 0}$ starting at $\o_*$ obtained by a random time--change from the Markov chain $(\o_n)_{n \in \bbN}$ starting at $\o_*$, imposing that the waiting time at $\o$ is an exponential random variable with parameter $\l_0(\o)$. Note in particular that $\o_t$ is defined for all $t\geq 0$. 
Given $\o \in \ovo{ \tilde{A}_3}$ and $x_0\in \hat \o$ let $(\o_t)_{t\geq 0}$ be the above  continuous-time Markov chain
 starting now at $\theta_{x_0}\o$.
 For $t\geq 0$ we set 
$X_t^{\o}:=x_0+ 
\sum _{
 s\in [0,t]: 
  \o_{s-}\not= \o_s}
  F(\o_{s-},\o_s)\,,
$
 where  $F(\o,\o'):=x$ if  $\o'=\theta_x \o$. Then $X_t^{\o}
$ coincides with the random walk described in Lemma \ref{lemma_no_TNT}. Setting $\ovo{\cA:= \tilde{A}_3}$ and using Lemma \ref{matteo}, the above construction  implies the content of Lemma \ref{lemma_no_TNT}.

\section{Technical facts concerning  Section \ref{sec_fluidifico}}\label{app_fluidifico}
In this appendix we consider the general context (instead of  the context of Warning \ref{marinaio}). \ovo{Below $\bar{\bbE}[\cdot]$ denotes the expectation w.r.t. $\bar{\cP}$.}
\begin{Lemma}\label{LB1} $\bar \cP$ is $\bar \bbG$--stationary.
\end{Lemma}
\begin{proof} 
Let $f: \bar \O \to [0,+\infty)$ be a \rrr{measurable} function and let $x\in \bbR^d$. Then we get (see below for some comments)
\begin{equation*}
\begin{split}
& \ell(\D) \bar\bbE[f\circ \bar \theta_x] =\bbE\Big[ \int_\D da f\bigl( \bar \theta_x( \o,a)\bigr)\Big]=\bbE\Big[ \int_\D da f\bigl( \theta_{g(x+a)} \o, \beta(x+a)
\bigr)\Big]
\\&=\bbE\Big[ \int_{\D+x} da f\bigl( \theta_{g(a)} \o, \beta(a)\bigr)\Big]
=\sum_{g\in \bbG}  \int_{(\D+x)\cap \t_g \D} da\bbE\Big[  f\bigl( \theta_{g} \o, a-Vg\bigr)\Big]
\\&=\sum_{g\in \bbG} \int_{(\D+x)\cap \t_g \D} da \bbE\Big[  f\bigl(  \o, a-Vg\bigr)\Big]
=\sum_{g\in \bbG} \int_{(\t_{-g}(\D+x))\cap  \D} du \bbE\Big[   f\bigl(  \o, u\bigr)\Big]\\
&= \int_{  \D} du \bbE\Big[   f\bigl(  \o, u\bigr)\Big]=\ell(\D) \bar\bbE[f]\,.
\end{split}
\end{equation*}
For the forth identity, we point out that if 
 $a=\t_g y=y+Vg$ with $y\in \D$, then $\b(a)= y= a-Vg$. The firth identity follows from the $\bbG$-stationarity of $\cP$. For the seventh identity  observe that 
 $\bbR^d=\sqcup_{g\in \bbG} (\t_{-g}(\D+x))$.
\end{proof}

    \begin{Lemma}\label{LB5} 
  $\bar \cP$ is ergodic w.r.t. the action $(\bar \theta_g)_{g\in \bar \bbG}$.
  \end{Lemma}
  \begin{proof} Let $f: \bar \O \to \bbR$ be a \rrr{measurable} function such that $f( \bar \theta _x \bar \o) =f (\bar \o)$ for any $x \in \bar \bbG$ and any $\bar \o \in \bar \O$. We need to prove that $f$ is constant $\bar \cP$--a.s.. 
 By \eqref{cetriolo} the invariance of $f$ reads
  $f\bigl( \theta _{g(x+a)} \o, \b (x+a) \bigr)=f(\o,a)$ for all $(\o, a)\in \bar \O$ and all $x \in \bbR^d$. Fixed $(\o,a) \in \bar \O$, given $a'\in \D$ we take $x:= a'-a$. As $x+a=a'\in \D$ we have $\beta (x+a)=a'$ and $g(x+a)=0$. In particular, the invariance of $f$ implies that $f(\o,a')=f(\o,a) $ for any $\o \in \O$ and $a,a'\in \D$. 
  When   $x= Vg $ for some $g\in \bbG$ and $a \in \D$, we have  $g(x+a)=g$ and $\b(x+a)=a$. The invariance of $f$ then implies that $f( \theta _g \o, a)= f(\o,a)$ for all $g\in \bbG$, $\o\in \O$ and $a\in \D$. As $\cP$ is ergodic w.r.t. the action $(\theta_g)_{g\in \bbG}$, given $a\in \D$ we conclude that $\exists  c\in \bbR$ such that for  $f(\o,a)= c$  for $\cP$--a.a. $\o$.   By combining this identity with the fact that $f(\o,a')=f(\o,a) $ for any $\o \in \O$ and $a,a'\in \D$, we conclude that $f\equiv c$ $\bar \cP$--a.s..
  \end{proof}

\begin{Lemma}\label{LB3}
The intensities   $\bar m $ and $m$  of the random measure $\mu_{\bar \o}$ 
and  $\mu_\o$, respectively, coincide.
\end{Lemma}
\begin{proof}
We have (see comments below) 
\begin{equation*}
\begin{split}
&\bar m \, \ell(\D)^2  =  \bar\bbE \bigl[ \mu_{\bar \o}(\D)\bigr] \ell(\D)=
\bbE\Big[ \int _\D da\, \mu_{\o} ( \D+a)\Big]\\
& = \sum_{g\in \bbG}  \int _\D da  \bbE\Big[ \mu_{\o}\bigl ( (\D +a)\cap \t_g \D \bigr) \Big]= \sum_{g\in \bbG}  \int _\D da  \bbE\Big[\mu_{\theta_{-g} \o}\bigl ( (\D +a)\cap \t_g \D \bigr) \Big] \\
&= \sum_{g\in \bbG}  \int _\D da  \bbE\Big[\mu_{\o}\bigl ( \t_{-g} (\D +a)\cap  \D \bigr)  \Big]=\int _\D da  \bbE\Big[ \mu_{\o} (   \D )\Big]=  m \ell(\D)^2\,.
\end{split}
\end{equation*}
In the third identity we used that $\bbR^d =\sqcup _{g\in \bbG}\t_{g}\D$. In the  forth identity we used the $\bbG$--stationarity of $\cP$. In the fifth identity we used \eqref{base}.
In the sixth identity we used that $\bbR^d =\sqcup _{g\in \bbG} \t_{-g} (\D +a)$.
%
\end{proof}

\begin{Lemma}\label{LB2} \utto{For all $\bar \o \in \bar{\O}_*$}
     it holds 
$\mu_{\bar \theta _x \bar \o}(\cdot) = \mu_{\bar \o}(\bar \t_x \cdot)$  for any $x\in \bar \bbG$.
\end{Lemma}
\begin{proof} Let $\bar \o=(\o,a)$ and  $A\in \cB(\bbR^d)$. We \utto{have} \eqref{base}.
We apply (in order) \eqref{cetriolo}, \eqref{limone}, \eqref{base}, \eqref{trasferta1} to get:
$\mu_{\bar \theta _x \bar \o}(A) = \mu_{(\theta_{g(x+a)} \o, \beta(x+a))}(A)= \mu_{\theta_{g(x+a)} \o}(A+ \beta(x+a))
=\mu_\o\bigl( \t_{ g(x+a)} (A+ \beta(x+a)\bigr)=\mu_\o( A+ \beta(x+a)+Vg(x+a))$.
Since  $\forall u\in\bbR^d$ we have $u= \beta(u)+ V g(u)$,  $\mu_\o( A+ \beta(x+a)+Vg(x+a))$ equals $\mu_\o ( A+x+a)= \mu_{(\o,a) }(A+x)=\mu_{\bar \o}(A+x)$.  \eqref{pinzimonio} then allows to conclude.
\end{proof}
\begin{Lemma}\label{lemma_mont} \utto{For all $\bar \o \in \bar{\O}_*$}
 it holds  $\bar r_{x,y} (\bar \theta_z \bar \o)= \bar r_{\bar \t_z x, \bar \t_z y} (\bar \o)$ for any $x,y\in \bbR^d$, $z\in \bar \bbG$.
\end{Lemma}
\begin{proof} Let $\bar \o=(\o,a)$. We \utto{have} \eqref{montagna}.
By  \eqref{cetriolo} and \eqref{normale},  we have
$\bar r_{x,y} (\bar \theta_z \bar \o)= \bar r_{x,y}
\bigl(\theta_{g(z+a)} \o, \beta(z+a)\bigr)=   r_{x+\beta(z+a),y+\beta(z+a)}
\bigl(\theta_{g(z+a)} \o\bigr)$.
Note that, for any $u \in \bbR^d$,  $\t_ {g(z+a)} (u+\beta(z+a))=u+\beta(z+a)+ V g(z+a)=u+z+a$. As a byproduct of this observation and  \eqref{montagna}, 
$  r_{x+\beta(z+a),y+\beta(z+a)}
\bigl(\theta_{g(z+a)} \o\bigr)$
can be rewritten as  
  $ r_{x+z+a, y+z+a}(\o)$. Hence, $\bar r_{x,y} (\bar \theta_z \bar \o)=r_{x+z+a, y+z+a}(\o)$. On the other hand, by  \eqref{pinzimonio} and \eqref{normale},
 $
  \bar r_{\bar \t_z x, \bar \t_z y} (\bar \o)= \bar r_{\bar \t_z x, \bar \t_z y} (\o,a)=
   r_{\bar \t_z x+a, \bar \t_z y+a} (\o)=r_{x+z+a, y+z+a}(\o)$.
\end{proof}
\begin{Lemma}\label{LB4}
The Palm distributions $\cP_0$ and $\bar\cP_0$ associated respectively to $\cP$ and $\bar \cP$ coincide.
\end{Lemma}
\begin{proof} 
Let $A\in \bar \cF$. Due to \eqref{pinzimonio} the function $\bar g$ analogous to \eqref{attimino}  for $\cS[2]$  is the identity map. Hence,   by \eqref{palm_classica}  and since $\bar m=m$, we have 
\be\label{tricolore}
\begin{split}
\bar \cP_0(A)
&=\frac{1}{m \ell(\D) ^2}\bbE\Big[ \int_{\D} da \, \int _{\D }  d\mu_{(\o,a)} (x) \mathds{1}_A\bigl(\theta_{g(x+a)} \o, \beta(x+a)\bigr)\Big]\\
&=\frac{1}{m \ell(\D)^2 }\bbE\Big[ \int_{\D} da \, \int _{\D+a }  d\mu_{\o} (x) \mathds{1}_A\bigl(\theta_{g(x)} \o, \beta(x)\bigr)\Big]\\
&=\frac{1}{m \ell(\D)^2 } \sum_{g\in \bbG}   \int_{\D} da\bbE\Big[  \int _{\D+a }  d\mu_{\o} (x)  \mathds{1} _{ \{g(x)=g\}} \mathds{1}_A\bigl(\theta_{g} \o, \beta(x)\bigr)\Big]\\
& =\frac{1}{m \ell(\D) ^2 } \sum_{g\in \bbG}   \int_{\D} da \bbE\Big[  \int _{\D+a }  d\mu_{\theta_{-g} \o} (x)  \mathds{1} _{ \{g(x)=g\}}  \mathds{1}_A\bigl( \o, \beta(x)\bigr)\Big]\,.
\end{split}
\en
Note that in the last identity we have  used  the $\bbG$--stationarity of $\cP$.
\utto{For all $\o\in \O_*$ we} can write  (see comments below)
\be\label{kindle}
\begin{split}
& \int _{\D+a }  d\mu_{\theta_{-g} \o} (x) \mathds{1} _{ \{g(x)=g\}}   \mathds{1}_A\bigl( \o, \beta(x)\bigr)=\int _{\D+a }  d\bigl( \t_{-g} \mu_{\o} \bigr)(x)  \mathds{1} _{ \{g(x)=g\}} \mathds{1}_A\bigl( \o, \beta(x)\bigr)\\
 & =\int _{\t_{-g}(\D+a) }  d \mu_{\o} (x) \mathds{1} _{\{x\in \D \}}  \mathds{1}_A\bigl( \o, \beta(\t_g x)\bigr)
 =\int _{(\D+a- V g )\cap \D}  d \mu_{\o} (x)    \mathds{1}_A\bigl( \o, x \bigr)\,.
 \end{split}
\en
Above, 
for the first identity we used \eqref{base}. For the second one   we used that $\t_{-g} \mathfrak{m}[ f]= \int  f( \t_{g} x )d \mathfrak{m}(x)$  observing that $g( \t_{g} x)=g $ if and only $x\in \D$. For the third one, we used that $\beta(\t_g x)=x$ for any $x\in \D$.
As a byproduct of \eqref{tricolore} and \eqref{kindle} we get (see comments below)
 \be
\begin{split}
\bar \cP_0(A)& =\frac{1}{m \ell(\D) ^2} \sum_{g\in \bbG}   \int_{\D} da \bbE\Big[ 
\int _{(\D+a- V g )\cap \D}  d \mu_{\o} (x)    \mathds{1}_A\bigl( \o, x \bigr)\Big]\\
&= \frac{1}{m \ell(\D)^2 }   \int_{\D} da \bbE\Big[
\int _{ \D}  d \mu_{\o} (x)    \mathds{1}_A\bigl( \o, x \bigr)\Big]\\
&=\frac{1}{m \ell(\D) }  \bbE\Big[ 
\int _{ \D}  d \mu_{\o} (x)    \mathds{1}_A\bigl( \o, x \bigr) \Big]=\cP_0(A)\,.
\end{split}
\en
Above, in the second identity we have used that 
 $\bbR^d= \sqcup_{g\in \bbG}(\D +a- Vg )$, while the last identity  follows from \eqref{Palm_Z}.
  \end{proof}
  
 \section{Proof of Lemma   \ref{compatto1} and \ref{compatto2}}\label{app_compatto}
\subsection{Proof of Lemma \ref{compatto1}}
Since $\{v_\e\}$ is bounded in $L^2(\mu_{\tilde \o}^\e)$, there exist $C, \e_0$ such that $\| v_\e\|  _{L^2(\mu_{\tilde \o}^\e)}\leq C $ for $\e \leq \e_0$. We fix a countable set $\cV \subset C_c(\bbR^d)$ such that $\cV$ is dense in $L^2(m dx)$.
We call $\cL$ the family of functions $\Phi$ of the form  $\Phi(x,\o) =  \sum_{i=1}^r a_i \varphi_i(x) g_i(\o)$, where   $r\in \bbN_+$, 
 $g_i \in \cG$, $\varphi_i \in \cV$ and  $a_i \in \bbQ$. $\cL$ is a  dense subset  of $L^2( m dx \times \cP_0)$ as $\cG$ is dense in $L^2(\cP_0)$.  By Schwarz inequality we have  
 \begin{equation}\label{cocchi}
\Big| \int   d \mu _{\tilde \o}^\e (x)  v_\e (x)  \Phi  (x ,  \theta _{x/\e} \tilde{\o} ) \Big | \leq  C 
\Big[ \int d \mu _{\tilde \o}^\e (x)  \Phi  (x ,  \theta _{x/\e} \tilde{\o} )  ^2  \Big]^{1/2}\,.
 \end{equation}
 Since  $\tilde \o \in \O_{\rm typ}\subset \cA[gg']$ for all $g,g'\in \cG$,    we have 
 \begin{multline}\label{CL}
\lim_{\e \da 0}  \int   d \mu _{\tilde \o}^\e (x) \Phi  (x ,  \theta _{x/\e} \tilde{\o} )  ^2  \\
= 
\sum_ i \sum_j a_i a_j \int dx\,m\, \varphi_i(x) \varphi_j(x) \bbE_0[ g_i g_j]= \| 
\Phi\|^2_{ L^2( m dx \times \cP_0)}\,.
\end{multline}
Due to  \eqref{CL} we get that the integral in  l.h.s. of \eqref{cocchi}  admits a convergent subsequence. 
  Since $\cL$ is    countable, by a diagonal procedure we can extract a subsequence $\e_k \da 0 $ such that the limit 
 $
 F( \Phi):=  \lim _{k \to \infty}   \int d \mu _{\tilde \o}^{\e_k} (x) 
 v_{\e_k} (x)  \Phi  (x ,  \theta _{x/\e_k} \tilde{\o} )   $ exists for any $\Phi \in \cL$  and it satisfies $| F(\Phi)| \leq C \| \Phi \|_{ L^2( m dx \times \cP_0) }$ by \eqref{cocchi} and \eqref{CL}. 
 Since $\cL$ is  a dense subset of  $L^2( m dx \times \cP_0)$, by Riesz's representation theorem  there exists a unique $v\in L^2( m dx \times \cP_0)$ such that $F(\Phi)= \int d\cP_0(\o) \int dx \,m \, \Phi(x, \o) v(x,\o)$ for any $\Phi \in \cL$. We also get $\| v\| _{ L^2( m dx \times \cP_0) }\leq C$.
 As $\Phi(x,\o):= \varphi(x) g(\o) $ - with $\varphi \in \cV $ and $b \in \cG$ - belongs to $\cL$, we get
 that \eqref{rabarbaro} is satisfied along the subsequence $\{\e _k\}$ for any  $\varphi \in \cV $, $b \in \cG$. 
 It remains to show that we can indeed take   $\varphi \in \cC_c(\bbR^d) $. To this aim we observe that we can take $\cV$ fulfilling the following properties: (i) for each $N\in \bbN_+$ $\cV$ contains a function  $\phi_N\in C_c(\bbR^d)$ with values in $[0,1]$ and equal to $1$ on $[-N,N]^d$; (ii)
 each $\varphi\in C_c(\bbR^d) $ can be approximated in uniform norm  by functions $\psi_n\in \cV$ such  that   $\psi_n$ has support inside $[-N,N]^d$, 
 where  $N=N(\varphi)$ is the minimal integer for which $\varphi$ has support inside 
 $[-N,N]^d$.
  By Schwarz inequality and the boundedness of $\{v_\e\}$  
 we can bound
$|  \int d \mu _{\tilde \o}^\e (x)  v_\e (x) [\varphi (x)-\psi_n] g ( \theta _{x/\e} \tilde{\o} )
 |^2$ by 
$  \leq C^2 \|\varphi - \psi_n(x)\|  _\infty ^2\int d \mu _{\tilde \o}^\e (x) \phi_N (x)
 g ( \theta _{x/\e} \tilde{\o} )^2$.
 Since $\tilde \o \in \O_{\rm typ}\subset  \cA[g^2]$ for all $g\in \cG$,  the last integral converges as $\e \da 0$ to $(C')^2:=\int dx \, m  \phi_N \ovo{(x)} \bbE_0[g^2]$. 
In particular,  using also  that $\psi_n \in \cV$, along the subsequence $\{ \e_k\}$ we have  
 \be
 \begin{split}  & \varlimsup_{\e \da 0} \int d \mu _{\tilde \o}^\e (x)  v_\e (x) \varphi (x) g ( \theta _{x/\e} \tilde{\o} )\\ & \leq C C'  \|\varphi - \psi_n\|  _\infty 
 + \varlimsup_{\e \da 0}\int d \mu _{\tilde \o}^\e (x)  v_\e (x) \psi _n(x)g ( \theta _{x/\e} \tilde{\o} ) \\
 &=C  C'    \|\varphi - \psi_n\|  _\infty  + \int d\cP_0(\o) \int dx \,m \,  v(x,\o)  \psi_n(x) g(\o)\,.
 \end{split}
 \en
 We now take the limit $n\to \infty$.
 By dominated convergence we conclude that, along the subsequence $\{ \e_k\}$,
 \be 
 \varlimsup_{\e \da 0} \int d \mu _{\tilde \o}^\e (x)  v_\e (x) \varphi (x) g ( \theta _{x/\e} \tilde{\o} ) \leq \int d\cP_0(\o) \int dx \,m \,   v(x,\o) \varphi(x) g(\o)\,.
 \en
 A similar result holds with the liminf, thus implying that \eqref{rabarbaro}  holds along the subsequence $\{ \e_k\}$ for any $\varphi \in  C_c(\bbR^d)$ and $g \in \cG$.

\subsection{Proof of Lemma \ref{compatto2}}
The proof of Lemma \ref{compatto2} is similar  to the proof of Lemma \ref{compatto1}. We only give some comments on some new steps. One has to replace $L^2( m\, dx \times \cP_0)$ with $L^2(  m\,dx \times \nu )$.  Now  $\cL$ is  the family of functions $\Phi$ of the form  $\Phi(x,\o,z) =  \sum_{i=1}^r a_i \varphi_i(x) b_i(\o,z)$, where   $r\in \bbN_+$, 
 $b_i \in \cH $, $\varphi \in \cV$ and  $a_i \in \bbQ$ ($\cV$ is as in the proof of Lemma \ref{compatto1}). Due to   Lemma \ref{cavallo} and since $\tilde \o \in \O_{\rm typ}
\subset \cA_1[bb']$ for all  $b,b'\in \cH$, we can write
 \begin{multline}\label{sirenehead}
 \int d \nu _{\tilde \o}^\e (x,z)   \Phi  (x ,  \theta _{x/\e} \tilde{\o},z  )  ^2  \\ = \sum _i \sum_j a_i a_j 
 \int d \nu _{\tilde \o}^\e (x,z)  \varphi_i(x) \varphi_j (x) b_i (  \theta _{x/\e} \tilde{\o},z) b_j (  \theta _{x/\e} \tilde{\o},z)\\
 = \sum _i \sum_j a_i a_j  \int d \mu _{\tilde \o}^\e (x) \varphi_i(x) \varphi_j (x) \widehat{b_i b_j}(  \theta _{x/\e} \tilde{\o})
  \,.
  \end{multline}
 Since $\tilde \o \in 
\O_{\rm typ}\subset \cA [\widehat{b b'}]$ for all  $b,b'\in \cH$, as $\e \downarrow 0$ the above r.h.s. converges to $\sum_ i \sum_j a_i a_j \int _{\bbR^d} dx\,m\, \varphi_i(x) \varphi_j(x) \bbE_0[ \widehat{b_i b_j}]= \| 
\Phi\|^2_{ L^2( m dx \times \nu)}$. At this point we can proceed as in the proof of Lemma \ref{compatto1} (recalling that $\cH$ is dense in $L^2(\nu)$).

\medskip

\noindent
{\bf Acknowledgements}: I thank G\"unter Last for useful comments on the literature concerning random measures, Pierre Mathieu for useful comments on a preliminary version of this preprint and on the minimality  of  the  assumptions and  Andrey Piatnitski for  useful discussions on 2-scale convergence.  
This work comes from the remains of the  manuscript \cite{Fhom}, which has its own
   story.      I thank my family for the support along that story.



\end{document}